\documentclass{article}

\usepackage{natbib}
\setcitestyle{authoryear}

\usepackage[english]{babel}
\usepackage[nottoc]{tocbibind}

\usepackage{xcolor}
  
\usepackage{amsmath,amsthm,amssymb,graphicx,comment,dsfont,fullpage, mathtools, tabu, multirow, enumitem, url, indentfirst}
\usepackage{algorithm2e, pdflscape, hyperref, makecell,sidecap}
\usepackage{authblk,ulem}
\RestyleAlgo{ruled}
\usepackage{floatrow}

\AtBeginDocument{}

\newtheorem{definition}{Definition}
\newtheorem{theorem}{Theorem}
\newtheorem{lemma}[theorem]{Lemma}
\newtheorem{proposition}[theorem]{Proposition}
\newtheorem{corollary}[theorem]{Corollary}

\hypersetup{linkcolor=blue, citecolor=blue, colorlinks=true, urlcolor=blue}

\title{The space of multifurcating ranked tree shapes: enumeration, lattice structure, and Markov chains}
\author[1]{Julie Zhang}
\author[2]{Noah A. Rosenberg}
\author[1,3]{Julia A. Palacios} 
\date{\today} 

\affil[1]{Department of Statistics, Stanford University}

\affil[2]{Department of Biology, Stanford University}

\affil[3]{Department of Biomedical Data Science, Stanford University}

\begin{document}

\maketitle

\begin{abstract}
Coalescent models of bifurcating genealogies are used to infer evolutionary parameters from molecular data. However, there are many situations where bifurcating genealogies do not accurately reflect the true underlying ancestral history of samples, and a multifurcating genealogy is required. The space of multifurcating genealogical trees, where nodes can have more than two descendants, is largely underexplored in the setting of coalescent inference. In this paper, we examine the space of rooted, ranked, and unlabeled multifurcating trees, which we denote by $\mathcal{MT}_N$. We recursively enumerate the space and then construct a partial ordering which induces a lattice on $\mathcal{MT}_N$. The lattice structure lends itself naturally to defining Markov chains that permit exploration on the space of multifurcating ranked tree shapes. Finally, we prove theoretical bounds for the mixing time of two Markov chains defined on the lattice, and we present simulation results comparing the distribution of trees and tree statistics under various coalescent models to the uniform distribution on $\mathcal{MT}_N$. 
\end{abstract}


\section{Introduction}


Phylogenetic inference methods typically model genealogical relationships of molecular samples as bifurcating trees, in particular in the context of Bayesian inference with coalescent models used in population genetics \citep{Kingman1982}. However, there has been recent interest in modeling multifurcating trees, also called multiple mergers, to apply these methods to a more general class of evolutionary processes. Multifurcating trees arise in various biological scenarios, including infectious disease transmission in the presence of superspreader events \citep{li2017quantifying, menardo2021multiple}, animal populations whose offspring distributions are highly skewed \citep{sargsyan2008coalescent, niwa2016reproductive, eldon2018evolution, byeon2019origin, eldon2020evolutionary}, and populations that undergo strong positive selection \citep{der2012dynamics, eldon2023sweepstakes,arnason2023sweepstakes}. Therefore, it is important to develop methods that explore the space of multifurcating trees for phylogenetic inference applicable to more general settings.\\ 

The $\Lambda$-coalescent is a commonly used class of multiple merger coalescent (MMC) models \citep{pitman1999coalescents, sagitov1999general}, which includes the Beta-coalescent \citep{berestycki2007beta} and the Psi-coalescent \citep{eldon2006coalescent}. It has been used to infer coalescent model parameters such as the effective population size given a fixed multifurcating coalescent tree \citep{hoscheit2019multifurcating, zhang2025multiple}. \cite{korfmann2024simultaneous} propose two methods to infer the effective population size and Beta-coalescent parameters from a given ancestral recombination graph (ARG), either assumed to be known or inferred from ARGweaver \citep{rasmussen2014genome}. The first method is an extension of the Multiple Sequentially Markovian Coalescent algorithm to allow for up to four lineages coalescing at once \citep{schiffels2020msmc}, and the second is a graph neural network model trained on ARGs simulated under various parameter values. \cite{zhang2021nonbifurcating} proposed a maximum likelihood (MLE) with regularization method via an adaptive LASSO algorithm, to estimate the branch lengths of a multifurcating genealogy for a given fixed tree topology. \cite{helekal2025inference} present an MCMC method to reconstruct dated and labeled multifurcating genealogies from an MLE estimated tree. The authors assumed the Beta-coalescent model and an extended Billera-Holmes-Vogtmann space embedding to sample the space of trees. \\

Since genealogies are not usually directly observed, model parameters can be estimated by sampling from the augmented posterior of model parameters and genealogies jointly, given molecular data. In \citet{lewis2005polytomies}, the authors propose a prior on unrooted, unranked and labeled multifurcating trees in which trees with $K$ internal nodes are $C$ times more likely than trees with $K+1$ internal nodes. The authors then proposed a Markov chain to explore this tree space and to do Bayesian inference of model parameters.\\


Inference on spaces of multifurcating trees is difficult because of its superexponentially growing cardinality compared to binary trees, which itself is a large space. The cardinality of the space of Kingman's coalescent trees, or labeled histories, is 
\begin{equation}\label{eq:l-r-binary}
    | \mathcal{T}_N^L | = \frac{N!(N-1)!}{2^{N-1}},
\end{equation}
where $\mathcal{T}_N^L$ denotes the space of bifurcating rooted, ranked, and labeled trees in which there is an ordering to the internal nodes and the tips are uniquely labeled \citep{edwards1970estimation}. The cardinality of $\mathcal{MT}_N^L$, the space of multifurcating rooted, ranked, and labeled trees is 
\begin{equation} \label{eq:l-r-multi}
    | \mathcal{MT}_N^L| \equiv f(N) = \sum_{k=1}^{N-1} S(N,k) f(k), \qquad f(1)=1,
\end{equation}
where $S(n,k)$ is the Stirling number of the second kind \citep{Murtagh1984}. Additionally, $S(N,k)f(k)$ is the number of multifurcating ranked and labeled trees with $N$ tips and $k$ internal nodes. For example, when $N=8$, $|\mathcal{T}_N^L |= 1,587,600$ and $| \mathcal{MT}_N^L| = 10,270,696$. \\

Different binary tree resolutions have been proposed to more efficiently explore the latent genealogical space. In \cite{palacios2019bayesian}, the authors propose to model the space of ranked and unlabeled binary trees with $N$ tips, known as Tajima's coalescent trees \citep{Sainudiin2015}. Here, we consider the space of unlabeled multifurcating ranked tree shapes, the extension of Tajima's trees to the multifurcating case, as an alternative for Bayesian inference. Although we do not tackle the inference problem in this work, we propose and study Markov chains in this space and provide different encodings for studying the space. In particular, we provide an encoding that allows us to enumerate the space---solving a problem posed by \cite{Murtagh1984} that to our knowledge has remained open until now. The second proposed encoding allows us to define a lattice on this space and to study Markov chains defined on the lattice. 

\begin{definition}
The space of rooted multifurcating ranked tree shapes with $N$ unlabeled tips is $\mathcal{MT}_{N}=\bigcup_{K=1}^{N-1}\mathcal{T}_{N,K}$, where $\mathcal{T}_{N,K}$ is the space of all multifurcating ranked tree shapes with $K$ internal nodes labeled in increasing order from the root to the tip. 
\end{definition}

\subsection{A procedure for studying tree spaces}

The paper lies in the tradition of numerous lines of investigation in mathematical phylogenetics that follow the procedure of (1) defining a concept of a space of trees, (2) enumerating the permissible trees in the space, (3) studying adjacency between trees in the space, and (4) exploring Markov chains that traverse the tree space. Here, using the multifurcating ranked tree shapes $\mathcal{MT}_N$ from \cite{Murtagh1984}, we pursue the subsequent steps of establishing the nature of the tree space and its traversal. \\

In the general procedure for analyzing tree spaces, once a type of tree structure is defined, the trees are enumerated. Enumeration of trees has been of mathematical interest since the second half of the 19th century \citep{cayley1856note, schroder1870vier}, connecting to mathematical phylogenetics in the second half of the 20th century~\citep{rosenberg2025mathematical}. \cite{felsenstein1978number} stated a recursive expression for counting the number of rooted trees with $n$ labeled tips and $m$ unlabeled internal nodes. Results for seven different spaces of dendrograms are surveyed in \citep{Murtagh1984}, which contains both bifurcating and multifurcating spaces, well-documented in the \href{https://oeis.org}{Online Encyclopedia of Integer Sequences} \citep{sloane2003line}. Other summaries of enumerations of structures of interest to phylogenetics appear in \cite{semple2003phylogenetics} and \cite{steel2016phylogeny}---and for multifurcating structures in particular, in \cite{wirtz2022enumeration} and \cite{dickey2025labelled}. \\


Next, with trees defined and enumerated, alternate representations or encodings of trees are often introduced to facilitate calculations, to define distances, and to construct Markov chains. For example, integer-valued square matrices (F-matrices) are used to bijectively encode ranked binary trees. This bijection is used to define distances between trees \citep{Kim2020}, and in supervised machine learning algorithms for adaptive selection inference \citep{mo2023domain}. The string representation on ranked binary trees \citep{donaghey1975alternating} is used to define a Markov chain on the space of ranked tree shapes \citep{samyak2024statistical}, and the matching pairs of phylogenetic trees is used to define both distances and random walks on the space of unranked and labeled binary trees \citep{diaconis1998matchings, diaconis2002,simper2022adjacent}. \\


In the next step of the analysis of tree spaces, toward describing Markov chains on a tree space, a notion of adjacency between trees is introduced and its consequences investigated. Bayesian phylogenetic inference methods rely on Markov chain Monte Carlo methods (MCMC) for exploring the tree space. Multiple Markov chains have been proposed and studied on binary tree spaces such as the random transpositions chain on both rooted and unrooted unranked and labeled binary trees \citep{diaconis1998matchings, aldous2000mixing}, the adjacent transpositions chain on rooted ranked binary trees (both labeled and unlabeled) \citep{simper2022adjacent}, the subtree-prune-and-regraft, and nearest-neighbor interchange \citep{spade2014note}. \\

With adjacency of trees characterized, Markov chains can be used in combinatorial optimization to find the maximum likelihood tree topology  \citep{friedman2001structural} and to find the Fr\'{e}chet mean of a sample of trees \citep{samyak2024statistical}. Studies of Markov chains on multifurcating trees are scarce. \cite{lewis2005polytomies} define a reversible jump Markov chain on the space of unrooted, labeled, and unranked multifurcating phylogenies and use Metropolis-Hastings sampling to get posterior distributions of multifurcating phylogenies. 
\cite{sorensen2024down} define Markov chains on unranked, labeled, planar multifurcating trees based on the $(\alpha, \gamma)$ splitting tree \citep{haas2008, chen2009} and conjecture the existence of diffusive scaling limits generalizing the ``Aldous diffusion'' as a continuum-tree-values process. \\

In this manuscript, we derive bijective representations of $T_{N,K} \in \mathcal{MT}_N$, a ranked tree shape with $N$ tips and $K$ internal nodes for $1 \leq K \leq N-1$. In particular, in Section~\ref{sec:string}, we extend the string representation \citep{donaghey1975alternating} to multifurcating trees and use it to enumerate the space $\mathcal{MT}_N$. We define a partial ordering on $\mathcal{MT}_N$ and prove it induces a lattice on $\mathcal{MT}_N$ using the multifurcating F-matrix representation we construct in Section~\ref{sec:lattice}. The connectivity between two trees has both a clear geometrical and F-matrix interpretation. Utilizing the lattice, we define two specific Markov chains on $\mathcal{MT}_N$ and prove bounds on the mixing time of these Markov chains in Section~\ref{sec:mc}. We then present some simulation results illustrating the differences between the uniform distribution of trees on $\mathcal{MT}_N$ simulated using a Markov chain on the lattice to the distribution of a specific $\Lambda-$coalescent \citep{pitman1999coalescents, sagitov1999general}. We finish by discussing further directions for which the proposed work could be used, such as defining tree distances and stochastic enumeration, in Section~\ref{sec:discussion}. 

\section{Enumerating $\mathcal{MT}_N$} \label{sec:string}

The first question we want to tackle is the unknown cardinality of $\mathcal{MT}_N$. The number of labeled multifurcating trees with $N$ tips is known, but having unlabeled trees presents a much larger challenge when enumerating the space. To aid the enumeration of $\mathcal{MT}_N$, we develop a string representation of a multifurcating tree that records the parents and offspring of each internal node. It can be seen as an extension of the string representation of binary trees proposed by \cite{donaghey1975alternating}, but not used for enumeration originally. \cite{callan2009note} used the string representation to count the number of ranked, unlabeled binary tree shapes with $N$ tips and $c$ cherries, where $c \in \{1,2,\ldots,\lfloor N/2 \rfloor\}$. \\ 

Let $T_{N,K}$ be a ranked tree shape with $N$ tips and $K$ internal nodes. Assume all $N$ samples are extant at the tips and the internal nodes are labeled from 1 to $K$ in increasing order from root to tip (see Figures~\ref{fig:all_string_N4} and \ref{fig:string-larger}). Only internal nodes are labeled, and the tips are unlabeled. 


\begin{definition}\label{def:string}
The \textit{\textbf{string representation}} of $T_{N,K}$ consists of two nonnegative, integer-valued vectors $\{t,l\}$ each with length $K$ where $t_i$ corresponds to the label of the parent of internal node $i$ and $l_i$ corresponds to the number of leaves that subtend from internal node $i$, i.e. the number of pendant edges of internal node $i$. 
\end{definition}

For example, $\{t=(0), l=(N)\}$ corresponds to the star tree with $N$ leaves in which a single node subtends $N$ tips. Figure~\ref{fig:all_string_N4} shows the string representation of all ranked tree shapes with $N=4$ tips. Note that while the vector $t$ alone bijectly encodes a binary ranked tree shape, it is not sufficient to encode $T_{N,K}$, so we must introduce the additional vector $l$. The ordered vector $t$ corresponds to edges between labeled nodes and the ordered vector $l$ corresponds to the number of edges connecting internal nodes and unlabeled leaves. Figure~\ref{fig:string-larger} shows the string representation for a larger tree with $N=12$ and $K=5$.  In the next proposition, we prove $\mathcal{T}_{N,K}$ is in bijection with a space of length $K$ vectors $\{t,l\}$ that satisfy specific constraints. 

\begin{figure}[h]
    \centering
    \includegraphics[width=0.95\linewidth]{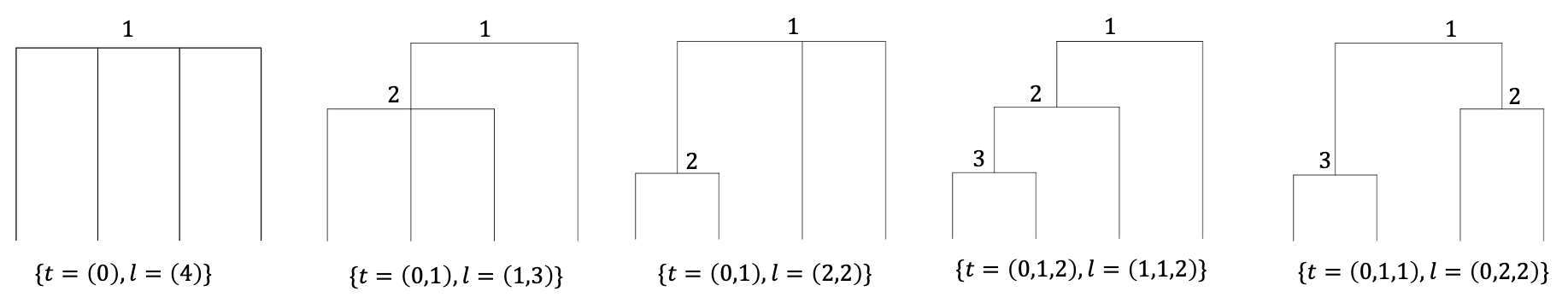}
    \caption{The string representation of all trees with 4 tips.}
    \label{fig:all_string_N4}
\end{figure}

\begin{figure}[h]
    \centering
    \includegraphics[width=0.65\linewidth]{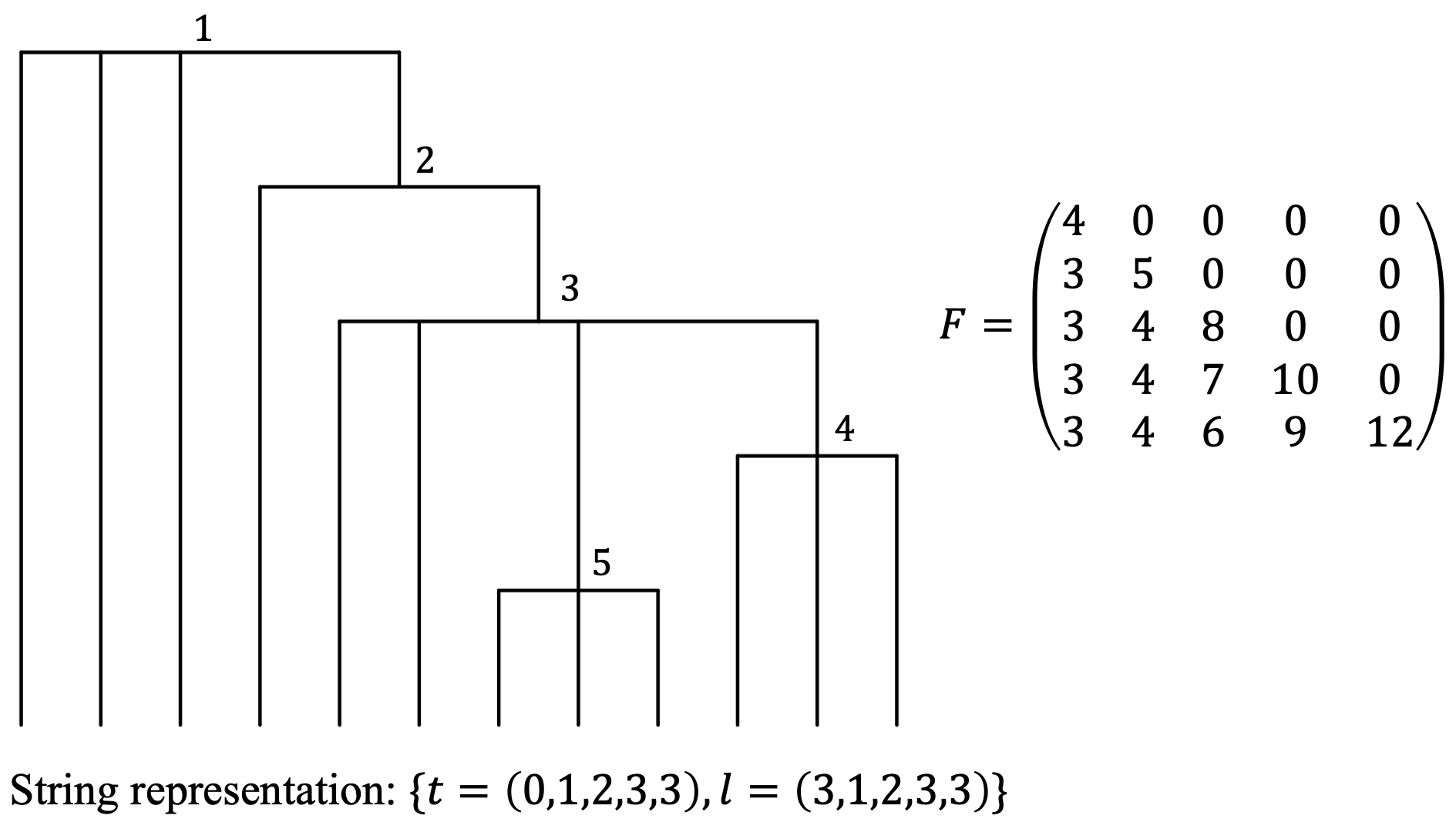}
    \caption{Two encodings of a ranked multifurcating tree with 12 tips and 5 internal nodes. The string representation $\{t,l\}$ is shown at the bottom of the tree. The parent node of 1 is 0 (not shown in the figure), so $t_1=0$. The parent node of 2 is 1, so $t_1=1$ and so forth. Node 1 subtends 3 leaves so $l_1=3$, node 2 subtends 1 leaf so $l_2=1$, and so forth. The F-matrix representation (introduced in Section~\ref{subsec:fmat}) is shown on the right of the tree. Let $u_{i}$ denote the time when node $i$ is created, then entry $F_{i,j}$ counts the number of extant lineages in the interval $(u_{j},u_{j+1})$ that do not furcate in $(u_{j},u_{i+1})$.}
    \label{fig:string-larger}
\end{figure}

\begin{proposition} \label{prop:string-bijection}
The space $\mathcal{T}_{N,K}$ is in bijection with the set of vectors $\{t, l\}$, each of length $K$ that satisfy the following constraints:  
\begin{enumerate}[label={S\arabic*.}]
    \item $t_1=0$ and $1\leq t_i\leq i-1$: the parent of the root is labeled 0 and the parent of node $i$ must have smaller rank (label) than $i$. 
    \item $\sum_{i=1}^K l_i=N$: the total number of leaves is $N$.
    \item $l_j\geq 2$ if $| \{ i: t_i= j\} | =0$: an internal node must subtend at least 2 leaves if it does not subtend an internal node. This also implies $l_K \geq 2$: the last internal node subtends at least 2 leaves and no internal nodes.
    \item $l_j\geq 1$ if $| \{ i: t_i= j\} | = 1$: an internal node must subtend at least 1 leaf if it is the parent of only one other internal node. 
\end{enumerate}
\end{proposition}
\begin{proof}
$T_{N,K}$ is a partially labeled directed graph where all internal nodes and edges connecting internal nodes form a directed graph $ G = \Big ( E_1 =\big \{ (t_i, i) \big \}_{i=2}^K, V= \big \{1,2,\ldots,K\big \} \Big )$ (see Figure~\ref{fig:construct_tree_from_string} Parts 1 and 2), together with a multiset of pendant edges $E_2 = \Big \{E_{2,i}= \big \{ e_i=(i,0),..., e_i=(i,0) \big \}_{i=1}^K \Big \}$, where $l_i = |E_{2,i}|$ (to get  Figure~\ref{fig:construct_tree_from_string} Part 3). The directed graph $G$ must satisfy $t_i \in \{1,2,\ldots,i-1\}$ for all $i=2,3,\ldots,K$ since the parent of node $i$ must exist before node $i$ itself, and the size of the multiset $E_2$ is $N$ since there are $N$ leaves. Additionally, the set of edges $E_1 \cup E_2$ must contain at least two edges of type $(i, \cdot)$ for all $i=1,2,\ldots,K$, representing the furcations of internal node $i$. It is clear that the set of all such $\{G, E_2\}$ is in bijection with $\mathcal{T}_{N,K}$. \\

\noindent To show the injection, for a given $G$ and $E_2$, we can uniquely define $t=(0,t_2,t_3,\ldots,t_K)$ and $l=(l_1,l_2,\ldots,l_K)$, where $l_i= |E_{2i}|$ for all $i=1,2,\ldots,K$. Clearly, the $t$ vector here satisfy conditions S1 and the $l$ vector satisfy conditions S2-S4 by the construction of $G$ and $E_2$. \\ 

\noindent Similarly to show the surjection, given a pair of ordered vectors $t=(0, t_2, t_3, \ldots, t_K), l=(l_1, l_2, \ldots ,l_K)$ that satisfy constraints S1-S4, we can construct the directed graph of edges connecting internal nodes is $ G = \Big ( E_1 = \big \{ (t_i, i) \big \}_{i=2}^K, V= \big \{ 1, 2, \ldots, K\big \} \Big )$ and the multiset $E_2$ of $N=\sum_{i=1}^K l_i$ undirected pendant edges is completely determined by $l$. Constraint S2 guarantees the tree has $N$ tips and $K$ internal nodes, constraint S1 guarantees that the parent of each internal node is another internal node that existed before it, and constraints S3 and S4 guarantee each internal node has at least two descendants. The result is a unique multifurcating ranked tree shape with $N$ tips and $K$ internal nodes, which shows the two spaces are in bijection. 
\end{proof}

In Figure~\ref{fig:construct_tree_from_string} we show how to reconstruct the tree from its string representation: $\{t=(0,1,2,3,3), l=(3,1,2,3,3)\}$ with the following steps: 
\begin{enumerate}
    \item Draw $K= |t|= |l|$ ranked internal nodes and $N= \sum_{i=1}^K l_i $ unlabeled tips 
    (S2). 
    \item Draw a directed edge between $(t_i,i)$ for $i=2,3,\ldots,K$. By S1, the $K$ internal nodes form a connected graph and $t_i <i$, which is the directed graph $G= (V,E_1)$. 
    \item Draw $l_i$ pendant edges from internal node $i$ to an unlabeled tip. This is the multiset $E_2$ of $N$ directed edges. By S3 and S4, each internal node now has at least two descendants, which shows it a multifurcating ranked tree shape with $N$ tips and $K$ internal nodes. 
\end{enumerate}

\begin{figure}[H]
    \centering
    \includegraphics[width=0.9\linewidth]{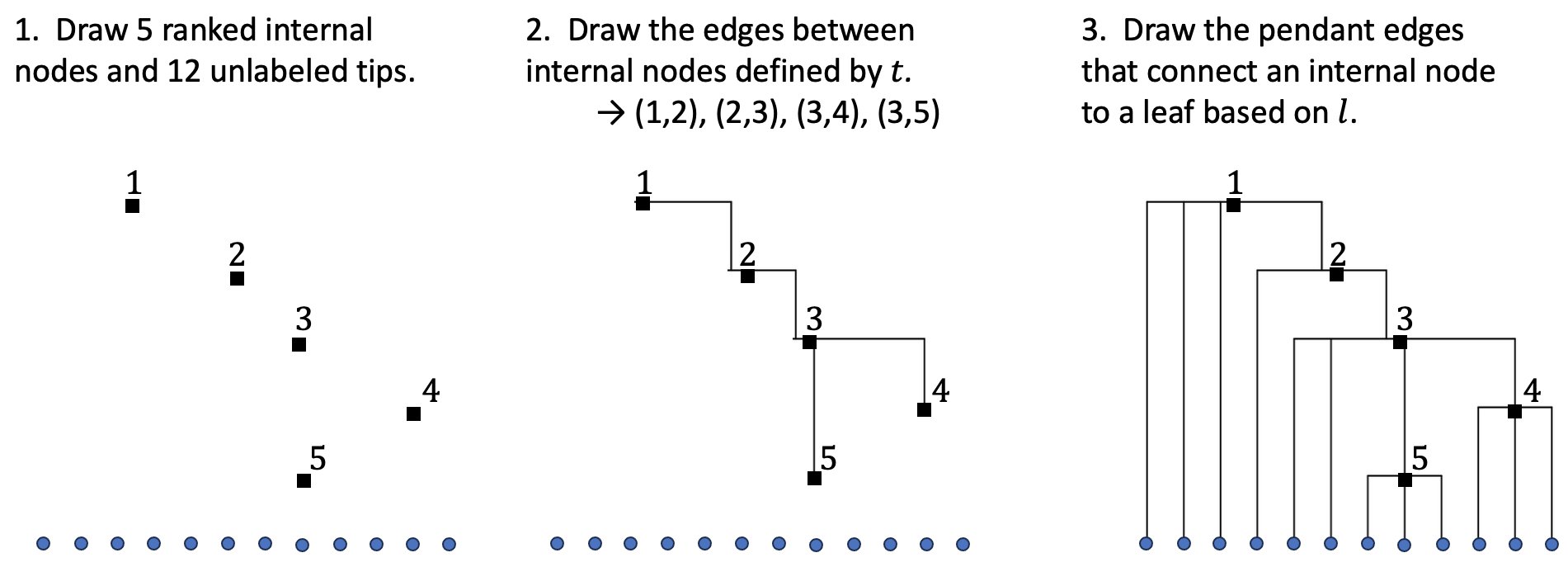}
    \caption{An example of the procedure to construct the tree from $\{t=(0,1,2,3,3),l=(3,1,2,3,3)\}$. Here $K=|t|=|l|=5$ and $M= \sum_{i=1}^K l_i=12$.}
    \label{fig:construct_tree_from_string}
\end{figure}

Based on the string representation, we define two quantities that are essential for enumerating the tree space. Let $K_0(t)$ be the number of internal nodes that do not appear in $t$ and $K_1(t)$ the number of internal nodes that appear exactly once in $t$, that is the number of internal nodes that subtend 0 and 1 internal nodes respectively. Mathematically, 
\begin{align*}
    K_0(t) &= \Big | \Big \{ j=1,2,\ldots,K : | \{ i: t_i= j\} | =0 \Big \} \Big |, \\
    K_1(t) &= \Big | \Big \{ j=1,2,\ldots,K : | \{ i: t_i= j\} | =1 \Big \} \Big |.
\end{align*}
Note in the definition of $K_1(t)$ the entry 0 does not count, as it is just a placeholder for the parent of the root and is not an actual internal node. $K_0(t), K_1(t)$ must satisfy $2K_0(t)+K_1(t) \leq N$ because the internal nodes that are not a parent to another internal node must subtend at least two leaves, and the internal nodes that are parent to exactly one internal node must subtend at least one leaf. We will utilize this fact about the ordered vector $t$ of the string representation to compute the cardinality of ranked multifurcating trees with $N$ tips and $K$ internal nodes $G(N,K):=|\mathcal{T}_{N,K}|$ for $K=1,2,\ldots,N-1$. \\

We first note that enumerative results for $K=1, 2$, and $N-1$ are already known: 
$G(N,1)= 1$ (star tree), $G(N,2) =  N-2$ as this number is simply dictated by the number of leaves descending from one of the nodes, and $G(N,N-1)$ is equal to the number of binary ranked tree shapes that corresponds to the Euler zig-zag number $1, 2, 5, 16, 61, 272, \ldots$ \citep{kuznetsov1994increasing,palacios2022enumeration}. For other values of $K$, our approach is to count the number of $l$ vectors that are compatible with one specific $t$ vector. This problem is equivalent to counting how many ways there are to assign leaves to each of the internal nodes, subject to certain constraints. Summing over all possible $t$ vectors will give the final enumerative result. We will use a well-known combinatorial result, often referred to as the ``stars and bars'' theorem: the number of ways to put $n$ indistinguishable balls into $k$ distinguishable boxes is $\binom{n-k+1}{k-1}$ \citep{batterson2011competition}. Here, our indistinguishable balls are the unlabeled leaves, and the distinguishable boxes are the labeled internal nodes. 

\subsection{The case of $K=3$}

Before we proceed to the general case, we will count the number of trees with $N$ tips and $K=3$ internal nodes. When $K=3$, there are only two possible $t$ vectors: $(0,1,1)$ and $(0,1,2)$. The number of multifurcating ranked tree shapes with $K=3$ is the sum of the number of all possible $l$ vectors compatible with the two $t$ vectors, that is, the number of ways to assign different number of leaves to each internal node in both $t$ vectors. \\

Consider the first case of $t=(0,1,1)$: since nodes 2 and 3 do not appear in $t$, we must have $l_2\geq 2$ and $l_3\geq 2$ and therefore we have $N-2\times 2=N-4$ leaves to be distributed among $K=3$ internal nodes. Since node 1 appears twice, there is no restriction on $l_1$. Therefore, the number of $l$ vectors compatible with $t=(0,1,1)$ is equivalent to the number of ways to put $N-4$ indistinguishable balls into 3 distinguishable bins: $\binom{N-4+3-1}{3-1} = \binom{N-2}{2}$. Here, we have $K_0(t)=2$ and $K_1(t)=0$, with $2K_0(t)+K_1(t)=4$. \\

In the second case of $t=(0,1,2)$, node 3 does not appear in $t$ so $l_3 \geq 2$. Node 1 and node 2 both appear once, so $l_1\geq 1, l_2 \geq 1$. The number of compatible $l$ vectors is again $\binom{N-4+3-1}{3-1} = \binom{N-2}{2}$. For this $t$, we have $K_0(t)=1$ and $K_1(t)=2$, with $2K_0(t)+K_1(t)=4$. Adding the two gives \[ G(N, 3) =  \binom{N-2}{2} +  \binom{N-2}{2}= N^2-5N+6= (N-2)(N-3).\] 

\subsection{The general case}

In the $K=3$ case, we saw how to find the number of compatible $l$ vectors for a given $t$ vector: calculating $K_0(t)$ and $K_1(t)$ and then using the stars and bars theorem with $N-(2K_0(t)+K_1(t))$ indistinguishable balls and $K$ distinguishable boxes. We can formulate the enumeration of $\mathcal{T}_{N,K}$ for general $K$ as follows:

\begin{proposition} \label{prop:trees-per-string}
The number of multifurcating ranked tree shapes with $N$ tips and ordered vector $t$ of the string representation of size $K$ is 
\[\binom{N-2K_0(t) - K_1(t) + K-1}{K-1}.\] 
The number of multifurcating ranked tree shapes with $N$ tips and $K$ internal nodes is \[ G(N,K)= \sum_{\substack{t \in \mathbb{N}^K: \; t_1=0, \\ t_i \in [i-1] \text{ for } i=2,3,\ldots,K}} \binom{N-2K_0(t) - K_1(t) + K-1}{K-1} \] where $[i]=\{1,2,\ldots,i\}$. 
\end{proposition}
\begin{proof}
For an arbitrary $t$, the values of $K_0(t)$ and $K_1(t)$ dictate how many leaves must already be assigned to certain parent nodes. The number of compatible $l$ vectors is therefore equivalent to the number of ways to place $N-2K_0(t) - K_1(t)$ balls into $K$ labeled boxes with no restrictions: $\binom{N-2K_0(t) - K_1(t) + K-1}{K-1}$. The second result then directly follows by summing over all possible $t$ vectors of length $K$. By definition of $t$, each element $t_i \in \{1,2,\ldots,i-1\}$ for $i=2,3,\ldots,K$ and $t_1=0$, which defines the set we are summing over. 
\end{proof}

The enumerative result in Proposition~\ref{prop:trees-per-string} requires summing over $(K-1)!$ elements because each element $t_i$ for $i=2,3, \ldots,K$ can take $i-1$ possible values. However, the summand only depends on $t$ via $K_0(t), K_1(t)$. For example, when $K=3$, the only two possible pairs are $(k_0,k_1)=(2,0)$ and $(k_0,k_1)=(1,2)$, each having one corresponding $t$ vector ($t=(0,1,1)$ and $t=(0,1,2)$ respectively). When $K=4$, there are three possible pairs for the six $t$-vectors: 
\begin{itemize}
    \item $(k_0,k_1)=(1,3)$ corresponds to 1 $t$ vector $t=(0,1,2,3)$,
    \item $(k_0,k_1)=(2,1)$ corresponds to 4 $t$ vectors: $t=(0,1,1,2), (0,1,1,3), (0,1,2,1), (0,1,2,2)$, 
    \item $(k_0,k_1)=(3,0)$ corresponds to 1 $t$ vector $t=(0,1,1,1)$.
\end{itemize}
Therefore, we can find a simplified and tractable expression for $G(N,K)$ by 1) finding the valid pairs of $(k_0,k_1)$ such that $K_0(t)=k_0$ and $K_1(t)=k_1$ for $t$ vectors of length $K$, and 2) counting the number of length $K$ $t$ vectors with $K_0(t)=k_0, K_1(t)=k_1$ for valid pairs $(k_0, k_1)$. \\

\noindent \textbf{Finding the set of valid $(k_0, k_1)$ pairs:} We can explicitly write down the set of pairs of $(k_0,k_1)$ such that $K_0(t)=k_0$ and $K_1(t)=k_1$ for $t$ vectors of length $K$. It is clear that $1 \leq K_0(t) \leq K-1$, and the next result details the possible value of $K_1(t)$ for each possible values of $K_0(t)$. 
\begin{proposition}\label{prop:valid-k0-k1-pairs}
Let $t$ be a valid length $K\geq 2$ vector. Then given $K_0(t)$, the values $K_1(t)$ can take are: 
\begin{align*}
    &K_0(t) = 1 \; \Rightarrow\; K_1(t) = K-1, \\
    & K_0(t) = K-1 \; \Rightarrow \; K_1(t) = 0, \\
    &K_0(t) = k_0  \in \{2, 3, \ldots, K-2\} \; \Rightarrow \; K_1(t) \in \Big \{ \max\{0, K-2k_0+1\}, ..., K-1-k_0 \Big \}.
\end{align*}
\end{proposition}
\begin{proof}
When $K_{0}=1$, the number of nodes that do not subtend internal nodes is 1, so the rest of $K-1$ internal nodes must subtend another internal nodes. In this case, the corresponding $t$ vector must be $t=(0,1,2,\ldots,K-1)$ and $K_1=K-1$. Now, if $K_0=K-1$, then only node 1 (the root) subtends internal nodes and the corresponding $t$ vector must be $t=(0,\underbrace{1,1,\ldots,1}_{K-1 \text{ times}})$ and so $K_1=0$. \\

\noindent For fixed $K_0(t)= k_0 \in \{2, 3, \ldots, K-2\}$, let $K_{\geq 2}(t)$ be the number of internal nodes that subtend at least 2 internal nodes.
The values $K_0(t), K_1(t), K_{\geq 2}(t)$ must satisfy $K_0(t) + K_1(t) + K_{\geq 2}(t) = K$ since the total number of internal nodes is $K$ and $K_1(t) + 2 K_{\geq 2}(t) \leq K-1$ because the total number of non-zero entries in $t$ is $K-1$ and $K_{\geq 2}(t)$ elements appear at least twice. Immediately, this implies 
\begin{align*}
    & K_1(t) + K_{\geq 2}(t) = K-K_0 \text{ and } K_1(t) + 2 K_{\geq 2}(t) \leq K-1 \\
    \Rightarrow \;\; & K_{\geq 2}(t) \leq K-1-(K-k_0) = k_0-1 \\
    \Rightarrow \;\; & K_1(t) \geq (K-k_0) -(k_0-1) = K-2k_0+1 
\end{align*}
Hence, for fixed $K_0(t)= k_0 \in \{2,...,K-2\}$, the values $K_1(t)$ can take are $K_1(t) \in \big \{\max\{0, K-2k_0+1\},..., K-1-k_0\big \}$. This the set of all possible $(k_0, k_1)$ pairs. 
\end{proof}

\begin{corollary}\label{eq:k0_k1_pairs_cardinality}
Let $\mathcal{K}_K$ denote the set of all valid $(k_0, k_1)$ pairs in Proposition~\ref{prop:valid-k0-k1-pairs}. The number of valid $(k_0, k_1)$ pairs given $K$ is $|\mathcal{K}_K| = \lfloor (K-1)^2/4\rfloor + 1$, which is OEIS sequence A033638.
\end{corollary}
The proof of Corollary~\ref{eq:k0_k1_pairs_cardinality} is in Appendix~\ref{subsec:string-proofs}. Instead of summing over $(K-1)!$ elements, we only need to sum over $\Theta(K^2)$ elements to find the cardinality of $\mathcal{T}_{N,K}$. \\ 

\noindent \textbf{Counting the number of length $K$ $t$ vectors with $K_0(t)=k_0, K_1(t)=k_1$:} After finding all possible pairs that $(K_0(t), K_1(t))$ can take, we need count how many $t$ vectors there are for any fixed pair. We can think of each length $K$ $t$ vector as a length $K-1$ $t$ vector with an additional element: $\tilde{t} = (t, t_K)$. Based on the added element $t_K$, we can express the changes in $K_0(\tilde{t})$ vs $K_0(t)$ and $K_1(\tilde{t})$ vs $K_1(t)$. For example, in the case of $K=4$, we can construct all six length $4$ $t$ vectors from the two length 3 $t$ vectors: 
\begin{itemize}
    \item $t=(0,1,1) \rightarrow \tilde{t} = (0,1,1,1)$: $t_4=1$ is an element that appeared twice in $t$, so $K_0(\tilde{t})= 3= K_0(t)+1$ and $K_1(\tilde{t}) = 0 = K_1(t)$. 
    \item $t=(0,1,1)\rightarrow \tilde{t} = (0,1,1,2)$ or $(0,1,1,3)$: $t_4=2$ or 3 is an element that did not appear in $t$, so $K_0(\tilde{t})= 2= K_0(t)$ and $K_1(\tilde{t}) = 1 = K_1(t)+1$. 
    \item $t=(0,1,2) \rightarrow \tilde{t} = (0,1,2,1)$ or $(0,1,2,2)$: $t_4=1$ or 2 is an element that appeared exactly once in $t$, so $K_0(\tilde{t})= 2= K_0(t)+1$ and $K_1(\tilde{t}) = 1 = K_1(t)-1$. 
    \item $t=(0,1,2) \rightarrow \tilde{t} =  (0,1,2,3)$: $t_4=3$ is an element that did not appear in $t$, so $K_0(\tilde{t})= 1= K_0(t)$ and $K_1(\tilde{t}) = 3 = K_1(t)+1$. 
\end{itemize}
There are four length 4 $t$ vectors with $K_0(t)=2$ and $K_1(t)=1$, two of which come from adding an element to $(0,1,1)$ and two of which come from adding an element to $(0,1,2)$. Breaking down the cases of $t_K$ being an element that appeared zero times, one time, or two or more times in $t$ allows us to recursively count how many length $K$ $t$ vectors there are with $K_0(t)=k_0, K_1(t)=k_1$. 

\begin{proposition} \label{prop:string-counting-recursion}
The number of $t$ vectors in the string representation of multifurcating ranked tree shapes with $K$ internal nodes and $N$ leaves with $K_{0}(t)=k_{0}$ and $K_{1}(t)=k_{1}$ is the number $A_K(k_{0},k_{1})$ that satisfy the following recursion:
\begin{align} \label{eq:K0_K1_recursionb}
    A_K(k_{0}, k_{1}) &= A_{K-1}(k_{0}-1, k_{1}) \times (K-k_{0}-k_{1}) + \\
    &\qquad A_{K-1}(k_{0}-1, k_{1}+1)\times (k_{1}+1) + \nonumber \\
    &\qquad A_{K-1}(k_{0}, k_{1}-1)\times k_{0}.  \nonumber
    \nonumber
\end{align}
with initial values: $A_K(1,K-1)=1$, $A_K(K-1,0)=1$. If $(k_0, k_1) \notin \mathcal{K}_K$, then $A_K(k_0, k_1)=0$.
\end{proposition}
\begin{proof}
For the base cases, we showed in the proof of Proposition~\ref{prop:valid-k0-k1-pairs} that there is only one $t$ vector corresponding to $K_0(t)=1, K_1(t)=K-1$ and only one $t$ vector corresponding to $K_0(t)=K-1, K_1(t)=0$, meaning $A_K(1,K-1)=1$, $A_K(K-1,0)=1$ for all $K$. \\

\noindent To find the recursion, consider a length $K-1$ $t$ vector and construct $\tilde{t} = (t, t_K)$ of length $K$, where $t_K \in \{1,2,...,K-1\}$. Let $K_0(\tilde{t}) = k_0, K_1(\tilde{t})= k_1$ and consider how many $t$ vectors and values of $t_K$ could result in this. 
\begin{itemize}
    \item If $t_K$ is equal to an element that appears twice or more in $t$, then $K_0(t)= k_0-1$ because $K$ is the new element that does not appear and $K_1(t) = k_1$ still. There are $K-k_0-k_1$ such elements to choose for $t_K$ and $A_{K-1}(k_0-1, k_1)$ such length $K-1$ $t$ vectors. 
    \item If $t_K$ is equal to an elements that appears once in $t$, it will now appear twice in $\tilde{t}$. This implies $K_1(t) = k_1 +1$ and $K_0(t)= k_0-1$. There are $k_1+1$ such elements to choose for $t_K$ and $A_{K-1}(k_0-1, k_1+1)$ such length $K-1$ $t$ vectors. 
    \item If $t_K$ is equal to an element that does not appears in $t$, then $K_0(t) = k_0 +1-1=k_0$ because the element $t_K$ now appears but $K$ does not appear, and $K_1(t) = k_1-1$. There are $k_{0}$ such elements to choose for $t_K$ and $A_{K-1}(k_0, k_1-1)$ such length $K-1$ $t$ vectors. Note this case includes $t_K=K-1$, which is an element that will never appear in $t$. 
\end{itemize}
Combining these three cases, we get the recursive expression of $A_K(k_{0}, k_{1})$, where by convention $A_K(k_0, k_1)=0$ if $(k_0, k_1) \notin \mathcal{K}_K$.
\end{proof}

\begin{corollary}\label{eq:k0_cardinality}
The number of length $K$ $t$-vectors with $K_0(t)=k_0$ is equal to the Eulerian number $E(K-1, k_0)$.
\end{corollary}

The proof is similar to that of Proposition~\ref{prop:string-counting-recursion} and is in Appendix~\ref{subsec:string-proofs}. Our final enumerative formula for $G(N,K)=|\mathcal{T}_{N,K}|$ combines together the three components: 1) the number of compatible $l$ vectors to a length $K$ $t$ vector with $K_0(t)= k_0, K_1(t)= k_1$; 2) the number of length $K$ $t$ vectors with $K_0(t)= k_0, K_1(t)= k_1$; and 3) the set of valid $(K_0(t), K_1(t)) = (k_0, k_1)$ pairs. 

\begin{theorem} \label{thm:enumeration-result-final}
The number of multifurcating ranked tree shapes with $N$ tips and $K$ internal nodes is
\[ G(N,K) = \binom{N-2}{K-1}+\binom{N-K+1}{K-1}+\sum_{k_0=2}^{K-2} \sum^{K-1-k_{0}}_{k_1=\max\{0,K-2k_{0}+1\}} A_K(k_0, k_1) \binom{N-2k_0 - k_1 + K-1}{K-1}, \] where $A_K(k_0, k_1)$ is equal to the number of $t$ vectors of length $K$ with $K_0(t) = k_0$ and $K_1(t)=k_1$.
\end{theorem} 
\begin{proof} 
Combining Propositions~\ref{prop:trees-per-string}-\ref{prop:string-counting-recursion} directly gives the desired result. The first two terms on the right hand side of $G(N,K)$ correspond to $(K_{0},K_{1})$ pairs $(1,K-1)$ and $(K-1,0)$. By convention, if $2k_0+k_1 > N$, the binomial coefficient term will be 0. We knew from the bijection of the string representation that there are no trees with $N$ tips and $K$ internal nodes whose string representation $t$ satisfies $2K_0(t)+ K_1(t) >N$, and this is implied by our enumeration. 
\end{proof}

\noindent \textbf{Remark:} The integer sequence of $\{A_K(k_0,k_1): (k_0,k_1) \in \mathcal{K}_K\}$ when sorted in increasing values of $k_0$ and $k_1$ is a refinement of the Eulerian numbers as expected, while another related sequence (OEIS Sequence A145271) is a further refinement of $\{A_K(k_0, k_1): (k_0, k_1) \in \mathcal{K}_K\}$. This sequence has connections to rooted tree objects in Hopf algebra, derivative operators, and Lagrange inversion \citep{copeland2010mathemagical}. Our integer sequence is not recorded in OEIS. For $K=8$, the three sequences are:  
\begin{itemize}
    \item $\{ B_K(k_0): k_0=1,...,K-1\}$: $1, 120, 1191, 2416, 1191, 120, 1$.
    \item $\{A_K(k_0,k_1): (k_0,k_1) \in \mathcal{K}_K\}$: $1, 120, 768, 423, 496, 1494, 426, 294, 741, 156, 98, 22, 1$.
    \item Integer sequence A145271 (reversed): $1, 120, 768, 423, 496, 1494, 426, 294, 267, 474, 156, 56, 42, 22, 1$. 
\end{itemize}

\subsection{Enumerative results}

We implemented the computation of $G(N,K)$ in R. The values of $G(N, N-1)$ computed by this formula match the Euler zig-zag numbers for binary ranked tree shapes. The enumerative results for $G(N,K)$ up to $N=12$ are shown in Table~\ref{tab:num_multif}, with the last column showing the values of $G(N) = |\mathcal{MT}_N| = \sum_{K=1}^{N-1} G(N, K)$. Note $G(N, 1)= 1, G(N, 2) = N-2$ for all $N$, so we only state results for $G(N,K), 3\leq K \leq N-1$. 

\begin{table}[H]
\begin{tabular}{l|lllllllll|r}
$N, K$ & 3  & 4   & 5    & 6    & 7    & 8     & 9  &10 & 11   & Total \\ \hline 
4    & 2  &     &      &      &      &       &   &&  & 5 \\
5    & 6  & 5   &      &      &      &       &   &&  &16 \\
6    & 12 & 21  & 16   &      &      &       &   &  & & 54\\
7    & 20 & 54  & 87   & 61   &      &       &   & &  & 228\\
8    & 30 & 110 & 276  & 413  & 272  &       &   & &  & 1,108\\
9    & 42 & 195 & 670  & 1,574 & 2,218 & 1,385  &   & & &6,092 \\
10   & 56 & 315 & 1,380 & 4,470 & 9,931 & 13,291 & 7,936 && & 37,388\\
11   &  72  & 476 & 2,541 &10,555 &32,475& 68,715& 87,963& 50,521 & & 253,327\\
12 & 90   & 684 &  4,312 & 21,931 & 86,885 &255,386 & 517,692 & 637,329 & 353,792 & 1,878,111
\end{tabular}
\caption{Enumerative results for the number of multifurcating ranked tree shapes with $N$ tips and $K$ internal nodes. Each row is a value of $N$ and each column is a value of $K$ for $K=3,...,N-1$. The last column is the value of $G(N)$, the total number of multifurcating ranked tree shapes with $N$ tips.}
\label{tab:num_multif}
\end{table}


To understand the growth order of $\mathcal{MT}_N$, from Proposition~\ref{prop:trees-per-string}, we have the result. 
\begin{equation}\label{eq:g_n_k_bound}
    G(N,K) \leq (K-1)! \times \binom{N-K-1}{K-1} \stackrel{(\ast)}{\leq} (K-1)! \times \frac{(N-K-1)^{K-1}}{(K-1)!} \leq N^{K-1}
\end{equation}
where $(\ast)$ holds because $\binom{N}{K}\leq \frac{N^K}{K!}$. In Appendix~\ref{subsec:polynomial_GNK}, we computed explicit polynomial expressions of $G(N,K)$ up to $K=8$, where $G(N,K)$ is a polynomial with leading term $N^{K-1}$. Specifically, Equation~\ref{eq:g_n_k_bound} implies
\begin{equation}
G(N)= |\mathcal{MT}_N| \leq \sum_{K=1}^{N-1} N^{K-1} \sim O(N^{N-2}), \qquad \log G(N) \sim O(N\log N). 
\end{equation}
See Appendix~\ref{subsec:cardinality_growth_rate} for some more numerical results on the growth rate of $G(N)$. \\

The cardinality of the space of ranked, unlabeled multifurcating trees is still drastically smaller than that of labeled multifurcating trees, whether ranked or unranked. For example for $N=11$, the number of ranked and unlabeled multifurcating trees is a bit more than 250,000 but there are more than 2 trillion ranked and labeled multifurcating trees. Enumerative results up to $N=12$ are shown in Table~\ref{tab:num_multif_compare}, where the numerical results for the two labeled multifurcating tree spaces are taken from the OEIS database \citep{sloane2003line}. Performing inference on the smaller tree space will allow for more efficient explorations of the latent genealogical space.

\begin{table}[H]
\begin{tabular}{r|rrr|}
$N$ & Ranked + Unlabeled & Unranked + Labeled & Ranked + Labeled \\ \hline 
3   & 2                  & 4                  & 4                \\
4   & 5                  & 26                 & 32               \\
5   & 16                 & 236                & 436              \\
6   & 54                 & 2,752              & 9,012            \\
7   & 228                & 39,208             & 262,760          \\
8   & 1,108              & 660,032            & 10,270,696      \\ 
9 & 6,092 & 12,818,912 & 518,277,560 \\ 
10 & 37,388 & 282,137,824 & 32,795,928,016 \\ 
11 & 253,327 & 6,939,897,856 & 2,542,945,605,432 \\
12 & 1,878,111 & 188,666,182,784 & 237,106,822,506,952
\end{tabular}
\caption{\textbf{Enumeration of three different multifurcating tree spaces including $\mathcal{MT}_N$.} Although the growth rate of the number of ranked and unlabeled trees is super exponential, it still grows much slower than the number of labeled trees (ranked or unranked). } 
\label{tab:num_multif_compare}
\end{table}

\section{Lattices on $\mathcal{MT}_N$} 
\label{sec:lattice}

In order to facilitate the study of the space of ranked, unlabeled multifurcating trees with $N$ tips, we construct a lattice on $\mathcal{MT}_N$ that exploits the natural relationship between tree topologies. A lattice would allow for a natural way to explore the space of multifurcating trees, and lead to definitions of distances and Markov chains on this space. Lattices have been defined on the set of binary perfect phylogenies in order to enumerate the number of compatible binary trees \citep{palacios2022enumeration}. In doing so, we also propose the multifurcating F-matrix representation of $T_{N,K}$, which is an extension of the F-matrix representation of binary trees \citep{Kim2020}. \\

We first establish some necessary definitions from lattice theory; more details can be found in \cite{nation1998notes}. A \textit{\textbf{partially ordered set}} (also called a poset) is $\mathcal{P} = (P, \leq)$, where $P$ is a nonempty set and $\leq$ is a binary relation on $P$ satisfying reflexivity, anti-symmetry, and transitivity. A \textit{\textbf{lattice}} $L=(P,\wedge, \vee)$ is a partially ordered set $\mathcal{P}$ with two binary operators meet $\wedge$ and join $\vee$ such that any two element set $\{x, y\} \subset L$ has a meet $x \wedge y= \text{GLB}\{x,y\}$ and a join $x \vee y=\text{LUB}\{x,y\}$. That is, any two elements has a unique least upper bound (LUB) and a unique greatest lower bound (GLB). Elements $x,y$ in a lattice $P$ satisfy $x\leq y$ if and only if $x \wedge y=x$. \\



We provide some more necessary definitions applicable to partially ordered sets. A poset $P$ is \textit{\textbf{bounded}} if it has a maximal element and a minimal element. For elements $x,z \in P$, we say $x$ \textit{\textbf{covers}} $z$ (denoted $z \prec x$) if $z \leq x$ and there does not exist $y$ such that $z \leq y \leq x$. The rank of a poset $P$ rk$(P)$ is the length of the longest chain in $P$, where a chain $\mathcal{C}$ is a subset of $P$ such that for all $x,y \in \mathcal{C}$, we have $x\leq y$ or $y\leq x$. The \textit{\textbf{Hasse diagram}} is a diagram to visualize finite posets: elements $x$ to $y$ are connected if $y$ covers $x$. We will use the following theorem to prove our defined poset $\mathcal{MT}_N$ is a lattice.

\begin{theorem}[Lemma 2.1 of \cite{bjorner1990hyperplane}] \label{thm:poset_2}
Let $P$ be a bounded poset with finite rank such that for all $x, y , z\in P$, such that if $x,y$ covers $z$ then $x \vee y = LUB(x,y)$ exists. Then $P$ is a lattice. 
\end{theorem}

\subsection{Partial ordering on $\mathcal{MT}_N$}

We first need to establish a partial ordering on the space of multifurcating trees. We will do so by defining an operation on trees which collapses an edge $(e,e+1)$ between consecutive internal nodes. We call this operation the \textit{\textbf{collapse edge operation}}, and after applying the collapse edge operation, the new tree will have one less internal node. Labeling the root with label 1 and the first coalescent event (back in time) with label $K$, the edge $(1,2)$ always exists in every tree. Therefore, this operation is well-defined. Figure~\ref{fig:edge_collapse_def_example} shows two  examples of the collapse edge operation on a tree with 7 tips and 5 internal nodes. Here, there are two edges with consecutive nodes: $(1,2)$ and $(3,4)$. The right trees in Figure~\ref{fig:edge_collapse_def_example} shows the resulting trees after the corresponding collapse edge operation. \\

\begin{figure}[h]
    \centering
    \includegraphics[width=0.65\linewidth]{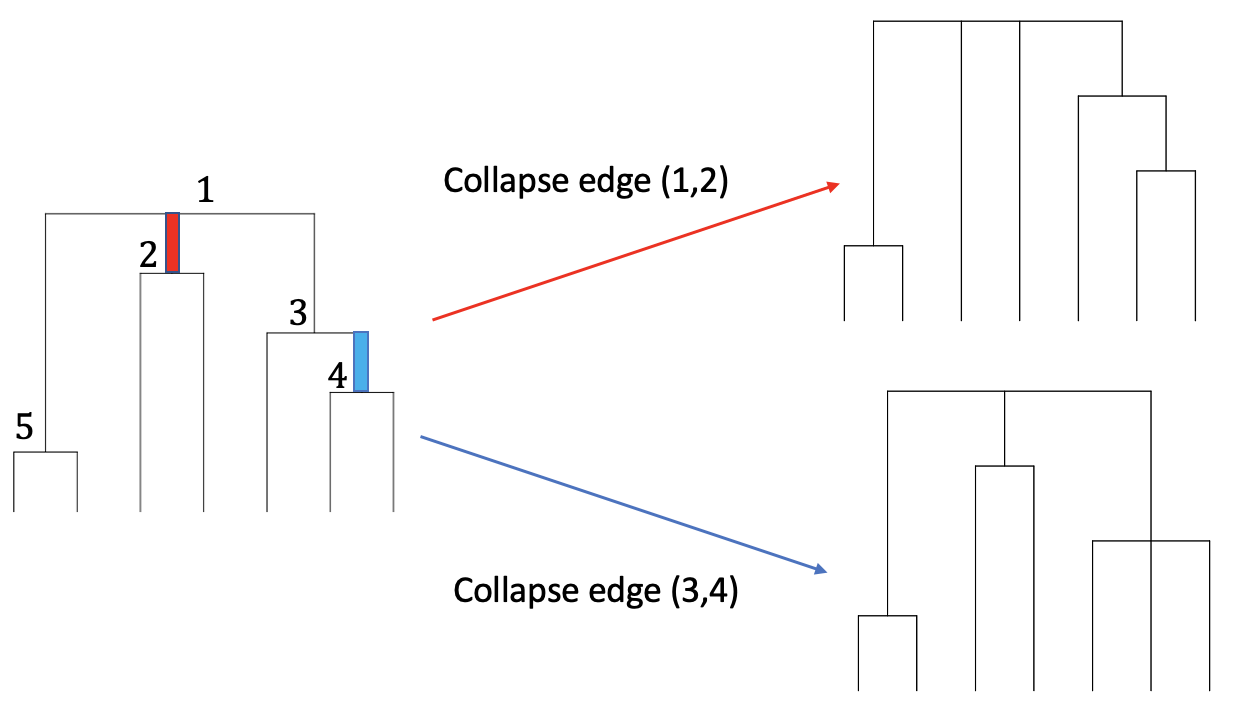}
    \caption{An illustration of the collapse edge operation on a tree with 7 tips and 5 internal nodes. There are two edges $(e,e+1)$ of consecutive nodes eligible for the collapse edge operation. The two trees on the right cover the tree on the left, as both can be obtained from one collapse edge operation.}
    \label{fig:edge_collapse_def_example}
\end{figure}

Using this operation, we can define a partially ordered set on $\mathcal{MT}_N \cup \emptyset$, where we include $\emptyset$ so that the set contains a unique minimal element, namely $\emptyset$. For $T_X, T_Y \in \mathcal{MT}_N$, we define $T_X \leq T_Y$ if $T_X=T_Y$ or $T_Y$ can be obtained from $T_X$ by sequentially collapsing edges that connect two consecutive internal nodes. If $T_X \leq  T_Y$, we call $T_X$ a refinement of $T_Y$. From this definition of a partial ordering, $T_Y$ covers $T_X$ if $T_Y$ is obtained from $T_X$ by only one collapse edge operation. The minimum element is $\emptyset$ and the maximum element is the star tree. We present the full Hasse diagrams for $N=4$ and $N=5$ in Figures~\ref{fig:lattice-4}-\ref{fig:lattice-5}, with the trees organized into columns of trees sharing the same number of internal nodes and arrows denoting the covering relations $T_X \prec T_Y$. It can be verified from the Hasse diagrams that every two elements have a unique least upper bound and greatest lower bound, meaning that this partial ordering induces a lattice. The unique LUB of $T_X, T_Y$ is the smallest tree that is refined by both (it will have the maximum number of internal nodes out of all trees that are refined by both) and the unique GLB of $T_X, T_Y$ is the largest tree that refines both (it will have the minimum number of internal nodes out of all trees that refine both). \\

Defining the partial ordering in terms of collapsing edges $(e,e+1)$ that connect consecutive internal nodes instead of collapsing any internal edge is key. The poset induced by the partial ordering of collapsing any internal edge is not a lattice (see Figure~\ref{fig:fake_lattice_N_5} for a counterexample). Additionally in this case, collapsing two different edges in the same tree could result in the same multifurcating ranked tree shape and the rankings of the internal nodes would not be preserved.\\ 

\begin{figure}[H]
    \centering
    \includegraphics[width=0.75\linewidth]{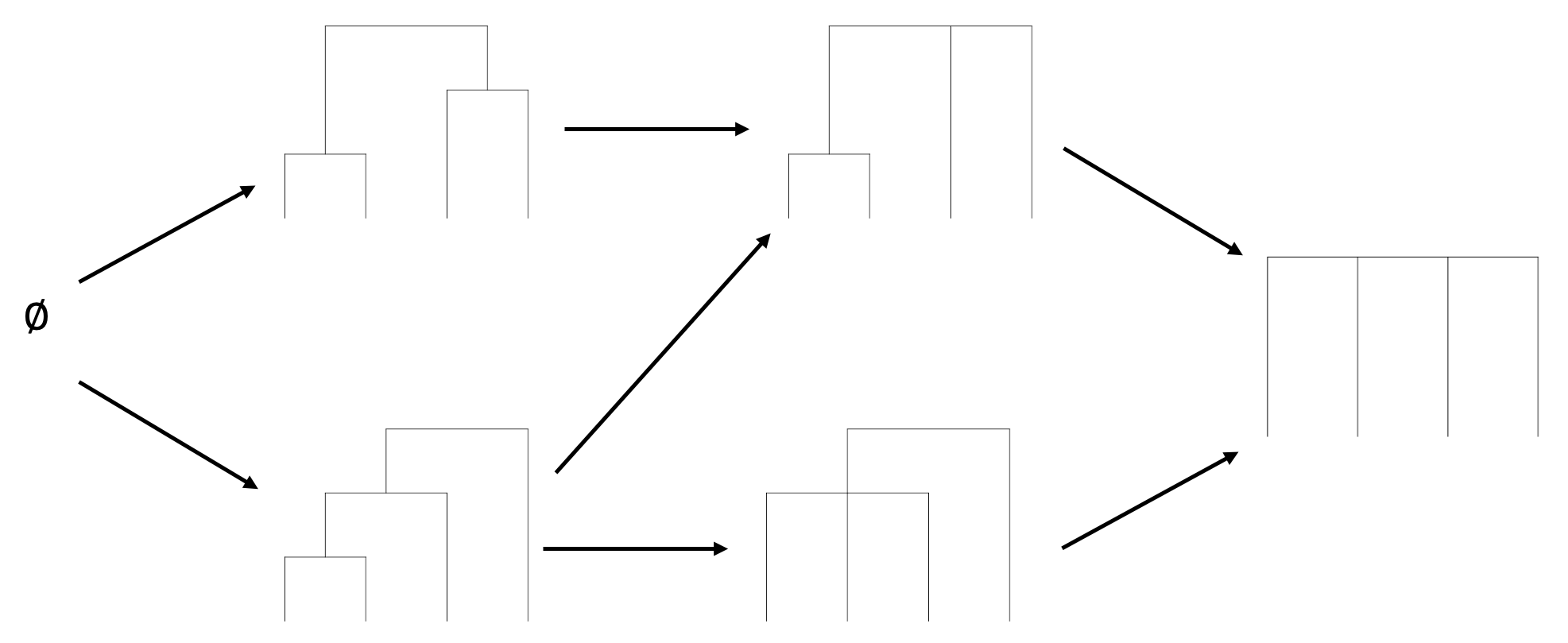}
    \caption{\textbf{The Hasse diagram for $\mathcal{MT}_4 \cup \emptyset$.} Each column contains trees with the same number of internal nodes, from the largest number of internal nodes $4$ in column 2, to the smallest number $1$ in the last column to the right.}
    \label{fig:lattice-4}
\end{figure}

\begin{figure}[H]
    \centering
    \includegraphics[width=0.95\linewidth]{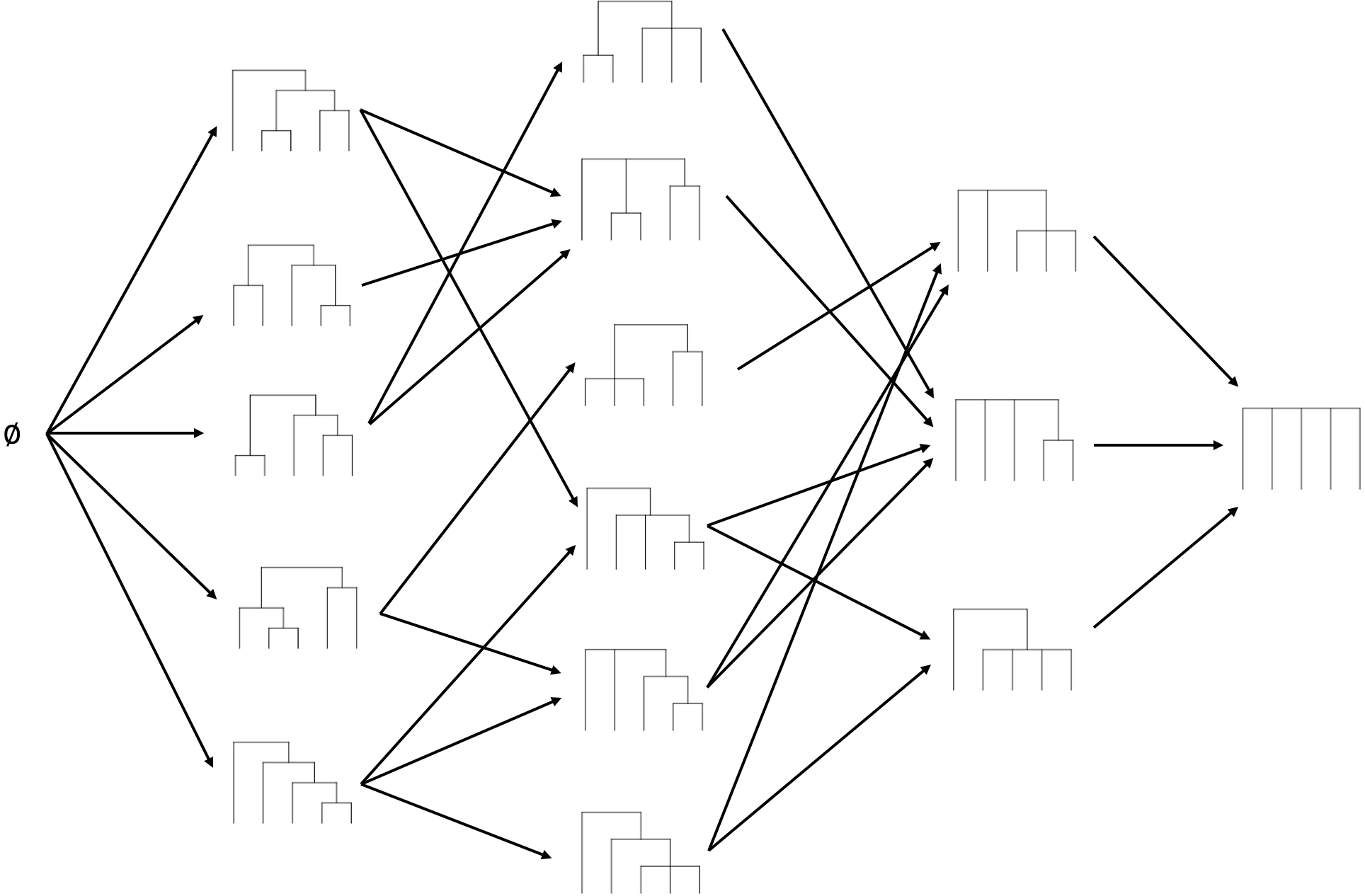}
    \caption{\textbf{The Hasse diagram for $\mathcal{MT}_5 \cup \emptyset$.} Each column contains trees with the same number of internal nodes, from the largest number of internal nodes $5$ in column 2, to the smallest number $1$ in the last column to the right.}
    \label{fig:lattice-5}
\end{figure}

\begin{figure}[H]
    \centering
    \includegraphics[width=0.95\linewidth]{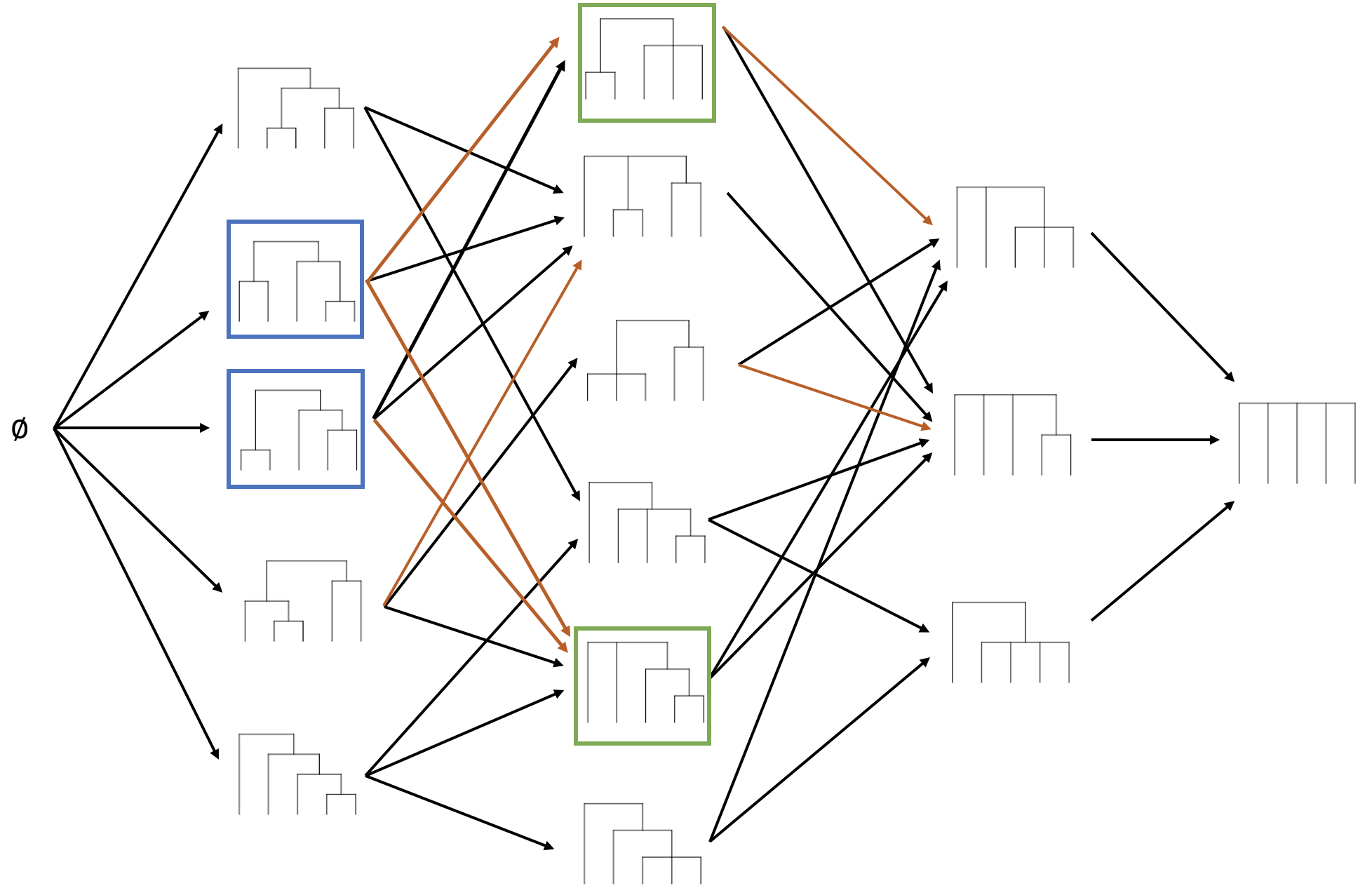}
    \caption{\textbf{The Hasse diagram for $\mathcal{MT}_5$ with partial ordering defined by collapsing any internal edge.} The black arrows are the coverings valid under the collapse edge operation, while the brown arrows denote the additional valid coverings under the collapse any internal edge operation. This partial ordering does not define a lattice because the two binary trees highlighted by the blue boxes do not have a unique least upper bound: both trees highlighted by the green boxes are valid least upper bounds which themselves are not comparable.} 
    \label{fig:fake_lattice_N_5}
\end{figure}

To prove $(\mathcal{MT}_N \cup \{\emptyset\}, \wedge, \vee)$ is an algebraic lattice, we introduce our second representation of a multifurcating tree: the multifurcating F-matrix. This representation is both essential in the proof that our defined partial ordering induces a lattice structure, and for operational use in finding least upper bounds. We will later show that the collapse edge operation corresponds to removing a column and a row of the F-matrix. 

\subsection{F-matrix for multifurcating trees}
\label{subsec:fmat}

We now consider an extension of the F-matrix representation of rooted, ranked, unlabeled binary trees \citep{Kim2020} to rooted, ranked, unlabeled multifurcating trees. As before, let $T_{N,K}$ be a ranked tree shape with $N$ tips and $K$ internal nodes with isochronous tips at time $u_{K+1}=0$. Let the $K$ coalescence times be $(u_K,..., u_1)$, where time is considered going backwards. We can encode $T_{N,K}$ by a $K\times K$ lower triangular matrix $F$ such that $F_{i,j}=0$ for $i<j$ and all $1 \leq i,j \leq k$, and $F_{i,j}$ for $1 \leq j \leq i$ defined to be the number of extant lineages in the interval $(u_{j}, u_{j+1})$ that do not furcate in the entire interval $(u_{j}, u_{i+1})$ traversing forward in time (root to tip). The F-matrix for the star tree is just the one-by-one matrix $(N)$. Figure~\ref{fig:F-matrix-4} shows all ranked tree shapes with $N=4$ tips. Figure~\ref{fig:string-larger} showed the F-matrix for a tree with 5 internal nodes and 12 tips, implying the F-matrix has dimension $5\times 5$ and $F_{5,5} = 12$. \\

\begin{figure}[H]
    \centering
    \includegraphics[width=0.975\linewidth]{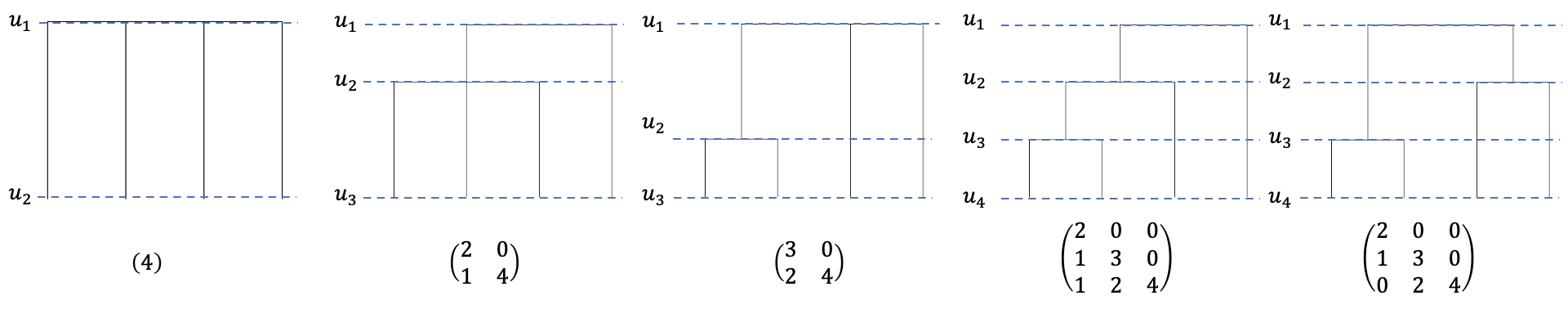}
    \caption{F-matrix representation of all trees with 4 tips.}
    \label{fig:F-matrix-4}
\end{figure}


\begin{theorem}\label{thm:fmat_multi}
The space $\mathcal{T}_{N,K}$ of ranked tree shapes with $N$ leaves and $K$ internal nodes is in bijection with the space $\mathcal{F}_{N,K}$ of $K\times K$ lower triangular matrices with non-negative integer entries that obey the
following constraints:
\begin{enumerate}[label={F\arabic*.}]
    \item The diagonal elements $F_{i,i}$ are strictly increasing: $F_{1,1}< F_{2,2}<...<F_{K,K}$ with $F_{K,K}=N$. The subdiagonal elements satisfy $F_{i+1,i} = F_{i,i}-1$ for all $i=1,...,K-1$.
    \item The elements $F_{i,1}, i=3,...,K$ in the first column satisfy \[ \max\{0, F_{i-1,1}- 1\} \leq F_{i,1} \leq F_{i-1,1}\] This just implies the column elements are decreasing by at most 1 if $F_{i-1,1}$ is non-zero. 
    \item All other elements $F_{i,j}, i=4,...,K, j=2,...,i-2$ satisfy these three conditions. 
    \begin{enumerate}[label={\alph*.}]
        \item $\max\{0, F_{i,j-1} \} \leq F_{i,j}$: Row-wise non-decreasing. 
        \item $F_{i-1, j}-1 \leq F_{i,j} \leq F_{i-1, j}$: Column-wise change is either 1 or 0.
        \item $0\leq (F_{i-1, j} - F_{i,j}) - (F_{i-1, j-1}- F_{i,j-1}) \leq 1$: in any $2\times 2$ grid, sum of off-diagonal elements minus sum of on-diagonal elements is 1 or 0. 
    \end{enumerate}
\end{enumerate}
The space of such matrices is denoted $\mathcal{F}_{N,K}$. 
\end{theorem}

The proof of Theorem~\ref{thm:fmat_multi} can be found in Appendix~\ref{subsec:lattice-proofs}. There is a straightforward connection between the F-matrix representation and the partial ordering on $\mathcal{MT}_N$: collapsing the edge $(e,e+1)$ corresponds to deleting the $e$th row and $e$th column of the F-matrix. An example using the tree from Figure~\ref{fig:edge_collapse_def_example} is shown in Figure~\ref{fig:Fmat_collapse_edge}. This connection will allow us to prove that the partial ordering on $\mathcal{MT}_N$ induces a lattice. \\

\begin{figure}[h]
    \centering
    \includegraphics[width=0.85\linewidth]{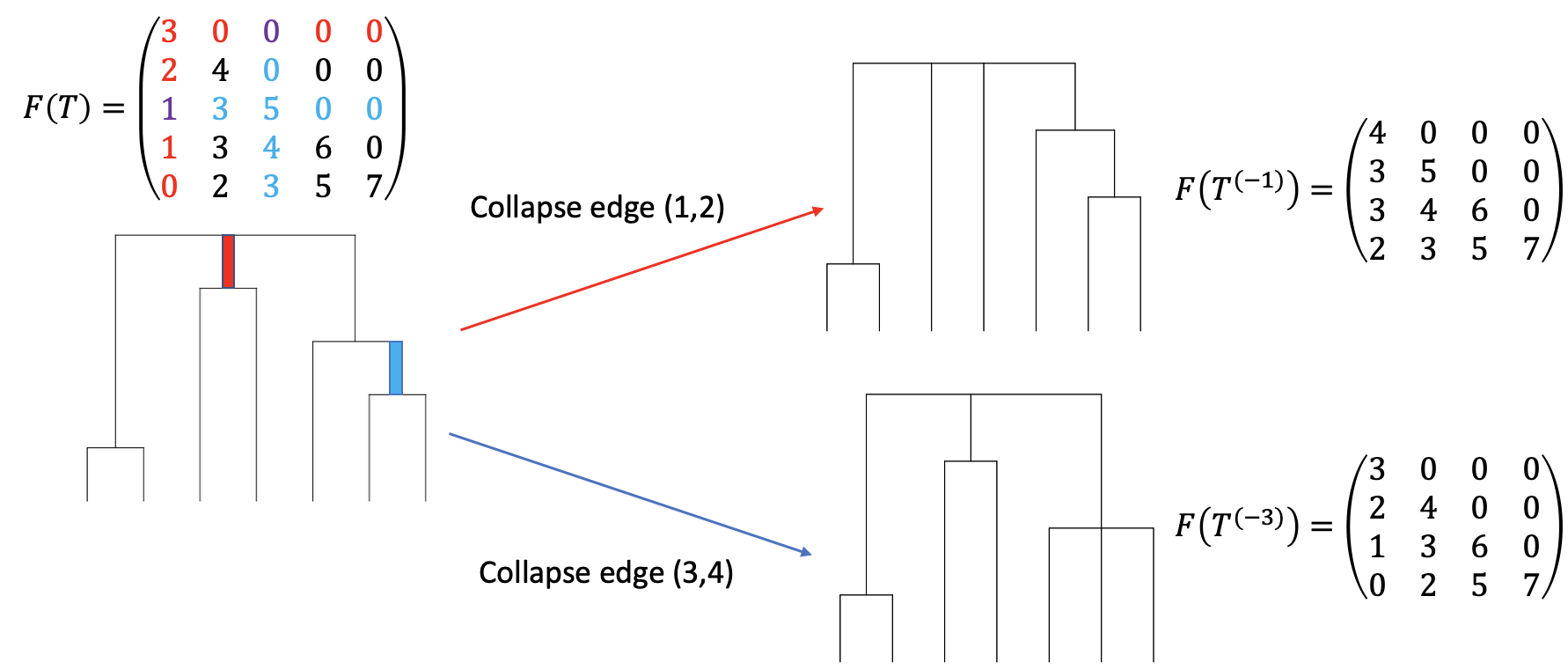}
    \caption{An example of the duality of the collapse edge operation and F-matrix row/column deletion using the tree from Figure~\ref{fig:edge_collapse_def_example}. We show the modified tree topologies and their F-matrices obtained by deleting the corresponding row and column from the original F-matrix.}
    \label{fig:Fmat_collapse_edge}
\end{figure}


First, we establish how to check the existence of the edge $(e,e+1), e\geq 2$ directly from the F-matrix. Having the edge $(e,e+1)$ means one descendant of node $e$ branched at the next time point. Since only one branching event can occur at a time, it implies none of the pre-existing branches furcated. Therefore, the F-matrix entries in the previous columns must remain constant: the edge $(e,e+1)$ existing in the tree for $e\geq 2$ is equivalent to 
\begin{equation}\label{eq:f-edge-condition}
    F_{e,j} = F_{e+1,j} \; \forall \; j=1,..., e-1.
\end{equation}
We prove the duality between row and column deletion and the collapse operation by first proving a lemma stating that deletion of such row and column still yields a valid F-matrix. The proof of Lemma~\ref{lemma:fmat_remove} is in Appendix~\ref{subsec:lattice-proofs}.

\begin{lemma} \label{lemma:fmat_remove}
Let $F \in \mathcal{F}_{N,K}$ and $(e,e+1)$ be an edge in the tree corresponding to $F$. That is, $e=1$, or $e\geq 2$ satisfies Equation~\ref{eq:f-edge-condition}. Then $F^{(-e,-e)}$ is a valid F-matrix, where the superscript denotes removing row and column $e$ from $F$. 
\end{lemma}

\begin{theorem} \label{thm:fmat_collapse}
Let $T_{N,K} \in \mathcal{T}_{N,K}$ with the edge $(e,e+1)$ for some $e \in \{1,...,K-1\}$. Let $T_{N,K-1}^{(-e)}$ be obtained from $T_{N,K}$ by collapsing edge $(e, e+1)$. Then $F(T_{N,K-1}^{(-e)})  = F(T_{N,K})^{(-e, -e)}$, where $F(T)$ returns the F-matrix corresponding to $T$. 
\end{theorem}
\begin{proof}
Lemma~\ref{lemma:fmat_remove} shows that deleting the $e$th row and column of the $F(T_{N,K})$ results in a valid F-matrix, so we must show the F-matrix of the tree collapsing edge $(e, e+1)$ is indeed $F(T_{N,K})^{(-e,-e)}$. Collapsing edge $(e, e+1)$ keeps all events in $(u_1, u_{e-1})$ the same and so the upper left submatrix $[1:(e-1),1:(e-1)]$ does not change in $F(T_{N,K-1}^{(-e)})$. The values of the lower left submatrix $[(e+1):K,1:(e-1)]$ would just be shifted to the $[e:(K-1),1:(e-1)]$ lower left submatrix in $F(T_{N,K-1}^{(-e)})$ because the count of furcating branches created before $u_{e}$ does not change. \\

\noindent Similarly, the values of the right lower submatrix $[(e+1):K,(e+1):K]$ refers to branches extant after $u_{e}$ and so their values would be shifted to $[e:(K-1),e:(K-1)]$ in $F(T_{N,K-1}^{(-e)})$. For example in Figure~\ref{fig:Fmat_collapse_edge}, when we collapse edge $(3,4)$, the entries in the bottom right $2\times 2$ remain the same and do not change. This is precisely $F^{(-e, -e)}$, which means $F(T_{N,K-1}^{(-e)}) = F^{(-e, -e)}$.
\end{proof}

To find the least upper bound $T_{X} \vee T_{Y}$, we first note that since the collapse edge operation can be easily written in terms of a row and column deletion of the F-matrix, so the partially ordered set on $\mathcal{MT}_N$ can be expressed entirely in terms of the F-matrices. Second, if $T_{X} \vee T_{Y}$ exists, then by Lemma~\ref{lemma:fmat_remove}, $F(T_{X}\vee T_{Y})$ is a submatrix of both $F(T_{X})$ and $F(T_{Y})$. Moreover, $F(T_{X}\vee T_{Y})$ is the largest shared submatrix of $F(T_{X})$ and $F(T_{Y})$. We exploit these two facts in Algorithm~\ref{alg:tree_binary_lub} together with the following two corollaries to find the unique least upper bound of two binary trees. Algorithm~\ref{alg:tree_lub} extends Algorithm~\ref{alg:tree_binary_lub} to find the unique least upper bound for any pair of multifurcating trees.

\begin{corollary}\label{cor:F_row_diff}
If $F_{i,j} = F_{i+1, j}+1$ for some $j\in \{1,...,i-1\}$, then $F_{i,k}= F_{i+1, k}+1$ for all $j\leq k \leq i-1$.
\end{corollary}
\begin{proof}
This results follows directly from the definition of $F_{i,j}$: $F_{i,j} = F_{i+1, j}+1$ implies one lineage that existed in the interval $(u_j,u_{j+1})$ furcated in the interval $(u_{i+1}, u_{i+2})$, and so $F_{i,k}= F_{i+1, k}+1$ also holds for all $j\leq k \leq i-1$. 
\end{proof}

As an example, see the F-matrix $F(T)$ in Figure~\ref{fig:Fmat_collapse_edge}. In rows 4 and 5, we have $F_{4,1} - F_{5,1}=1$, which implies $F_{4,i}- F_{5,1}=1$ must hold for $i=2,...,4$, which we see is indeed true in the F-matrix. 

\begin{corollary}\label{cor:F_delete_illegal_row}
If $e$ does not satisfy Equation~\ref{eq:f-edge-condition}, then $\tilde{F}_{e-1, e-1} = \tilde{F}_{e-1, e}+2$ where $\tilde{F}= F^{(-e, -e)}$. That is, deleting a column and row in the F-matrix that does not correspond to an edge $(e, e+1)$ will yield an invalid F-matrix, specifically one column will have a consecutive element difference greater than one. 
\end{corollary}
\begin{proof}
Let index $j \in \{1,..., e-1\}$ be the one such that $F_{e,j} \neq F_{e+1, j}$. Then by Corollary~\ref{cor:F_row_diff}, $F_{e, e-1} = F_{e+1, e-1}+1$. By condition F1 of the definition of the F-matrix, we know $F_{e-1, e-1} = F_{e, e-1} +1$ and hence $F_{e-1, e-1} = F_{e+1, e-1}+2$. Taking $\tilde{F}= F^{(-e, -e)}$ gives the result $\tilde{F}_{e-1, e-1} = \tilde{F}_{e-1, e}+2$. 
\end{proof}

\begin{algorithm}[H]
\caption{Algorithm to find the least upper bound of two binary ranked tree shapes}
\label{alg:tree_binary_lub}
\vspace{0.1in}
\KwData{$T_X, T_Y$: two binary ranked tree shapes with $N$ tips.} 
\vspace{0.1in}

\begin{enumerate}
    \item Find the set of column indices for which $F(T_X)$ and $F(T_Y)$ differ: \[ I = \{i: F(T_X)_{\cdot i} \neq F(T_Y)_{\cdot i} \}. \] and define the largest shared submatrix of $F(T_X), F(T_Y)$ as \[ S_F(T_X, T_Y) = F(T_X)^{(-I, -I)} = F(T_Y)^{(-I, -I)}.\]
    \item Find all columns $j$ that violate conditions F1 and F3b of the F-matrix definition:  
    \begin{itemize}
        \item Consecutive elements in the $j$th column (starting at entry $j,j$) differ by 2 or more. 
        \item $S_F(T_X, T_Y)_{j,j} -1 \neq S_F(T_X, T_Y)_{j+1,j}$: the subdiagonal element must be one less than the diagonal element.
    \end{itemize}
    Call this set of column indices $D$ and update $S_F(T_X, T_Y)\gets S_F(T_X, T_Y)^{(-D, -D)}$. 
    \item Repeat Step 2 until no such columns remain. 
    \item Return $S_F(T_X, T_Y)$: this is the LUB of $T_X, T_Y$. 
\end{enumerate}
\end{algorithm}
\vspace{0.1in}

Note this algorithm will always terminate because the matrix $S_F(T_X, T_Y)$ will always have the lower right $2\times 2$ matrix of $\begin{pmatrix} N-1 & 0 \\ N-2 & N \end{pmatrix}$ that does not violate the conditions. An example of this algorithm is given in Appendix~\ref{subsec:lub-algs}. 

\begin{proposition}\label{prop:binary-lub}
Algorithm~\ref{alg:tree_binary_lub} returns the unique least upper bound of any two binary trees. 
\end{proposition}
\begin{proof}
We must prove that 1) the output of Algorithm~\ref{alg:tree_binary_lub} is an F-matrix and 2) it has the maximum possible dimension out of all possible submatrices of $F(T_X)$ and $F(T_Y)$ that are F-matrices. \\

\noindent First, the final output of Algorithm~\ref{alg:tree_binary_lub}, $S_F(T_X, T_Y)$ is a submatrix of $F(T_X)$ since it is obtained by deleting a set of rows and columns from $F(T_X)$. Lemma~\ref{lemma:fmat_remove} and Corollary~\ref{cor:F_delete_illegal_row} together imply that only deleting rows and columns that correspond to an edge $(e, e+1)$ will yield a submatrix with columns in which consecutive values differences below the diagonal are less than 2. Since $S_F(T_X, T_Y)$ does not have any columns with consecutive column element differences below the diagonal greater than 2, we must have deleted only rows and columns from $F(T_X)$ that satisfied Equation~\ref{eq:f-edge-condition}. Hence the output of Algorithm~\ref{alg:tree_binary_lub} is an F-matrix. Although Algorithm 1 deletes columns that violate a certain condition in the F-matrix definition, these columns correspond to a sequence of collapse edge operations in the right order, as it is the only way that row and column deletion operations would yield a valid F-matrix.\\ 

\noindent We will show that Algorithm~\ref{alg:tree_binary_lub} returns the maximum shared submatrix by contradiction. Suppose there is an extra row and column deleted by the algorithm (column $e$ in the original F-matrix) that should not have been deleted. Then, column-row $e$ could have been deleted for two possible reasons: 1) the $e$th columns of $F(T_X)$ and $F(T_Y)$ differ, or 2) the consecutive elements of column $e$ below the diagonal differed by 2 or more. Case 1 is not possible because then column-row $e$ would not be in any shared submatrix of $F(T_X)$ and  $F(T_Y)$. In Case 2, this implies that deleting another invalid column, and therefore removing some row elements of column $e$, would fix the violations in column $e$. However, this is not possible because the elements in column $e$ are non-increasing, and the deletion of any element would only increase the consecutive element differences in column $e$. Therefore, $S_F(T_X, T_Y)$ obtained via Algorithm 1 has the maximum possible dimension of a valid submatrix that is an F-matrix and such matrix is unique.
\end{proof}

The algorithm to find the least upper bound of any two multifurcating ranked tree shapes with $N$ tips only needs an additional first step: to get the largest submatrices of $F(T_X)$ and $F(T_Y)$ such that both have the same dimension and the same diagonal values. \\

\begin{algorithm}[H]
\caption{Algorithm to find the least upper bound of two multifurcating ranked tree shapes}
\label{alg:tree_lub}
\vspace{0.1in}
\KwData{$T_X, T_Y$: two multifurcating ranked tree shapes with $N$ tips.} 
\vspace{0.1in}

\begin{enumerate}
    \item Find the set of indices for $F(T_X)$ and $F(T_Y)$ where the diagonal elements differ: 
    \begin{align*}
        I_X &= \{i: F(T_X)_{i,i} \not\subset \text{diag}(F(T_Y))\} \\
        I_Y &= \{i: F(T_Y)_{i,i} \not\subset \text{diag}(F(T_X))\}
    \end{align*}
    Set \[ F^-(T_X) = F^-(T_X)^{(-I_X, -I_X)}, \qquad F^-(T_Y) = F^-(T_Y)^{(-I_Y, -I_Y)}\] The matrices $F^-(T_X), F^-(T_Y)$ now have the same dimension and the same diagonal. 
    \item Apply Algorithm~\ref{alg:tree_binary_lub} with inputs $F^-(T_X), F^-(T_Y)$. 
\end{enumerate}
\end{algorithm} 
\vspace{0.1in}

Algorithm~\ref{alg:tree_lub} terminates because the matrix $F(T_X), F(T_Y)$ will always share the lower right element $(N)$. An example is shown in Appendix~\ref{subsec:lub-algs}. To see the output of Algorithm~\ref{alg:tree_lub} is the unique least upper bound of two trees $T_{X}$ and $T_{Y}$, first note that Algorithm~\ref{alg:tree_binary_lub} did not use any property of binary trees other than the fact that F-matrices of binary trees have the same diagonal and the same dimension. Now, if $F(T_{X}\vee T_{Y})$ exists, there exists a submatrix of $F(T_{X})$ and a submatrix of $F(T_{Y})$ obtained by row-column deletion operations, whose diagonals coincide. Call these matrices $F^-(T_X)$ and $F^-(T_Y)$. If $ F^-(T_X)= F^-(T_Y)$, then this is $F(T_{X}\vee T_{Y})$ by the same arguments in the proof of Algorithm~\ref{alg:tree_binary_lub}. If they are not the same, then we have the same conditions as in Algorithm~\ref{alg:tree_binary_lub} and it will find $F(T_{X} \vee T_{Y})$. We can now combine everything to prove our main result. 

\begin{theorem}\label{thm:tree_lattice}
The partial ordering $\leq $ defined by collapsing edges $(e,e+1)$ induces a lattice on the space $\mathcal{L}(\mathcal{MT}_N\cup \{\emptyset\}, \leq  )$. 
\end{theorem}
\begin{proof}
Because we proved the F-matrix encodes a unique multifurcating tree and can be used to find the LUB of two trees, we can work directly with F-matrices. The minimum and maximum element of $\mathcal{MT}_N$ are $\emptyset$ and $(N)$ respectively. We will use Theorem~\ref{thm:poset_2} because the rank of $\mathcal{MT}_N \cup \{\emptyset\}$ is finite (it is equal to $N$). \\

\noindent We first consider the case where our two elements cover a tree and not $\emptyset$ (note only binary trees cover $\emptyset$). Let $T_X, T_Y, T_Z$ be trees with $N$ tips where $T_X$ is obtained by collapsing edge $(e_X, e_X+1)$ of $T_Z$ and $T_Y$ is obtained by collapsing edge $(e_Y, e_Y+1)$ of $T_Z$. Without loss of generality, let $e_X < e_Y$. Since in $T_Z$ both edges $(e_X, e_X+1)$ and $(e_Y, e_Y+1)$ exists, it implies $T_X$ has the edge $(e_Y-1, e_Y)$ due to the shifting of internal nodes after collapsing an edge, and $T_Y$ has the edge $(e_X, e_X+1)$. The F-matrices of $T_X$ and $T_Y$ are $F(T_X)=F(T_Z)^{( -e_X, -e_X  )}$ and $F(T_Y)=F(T_Z)^{ ( -e_Y, -e_Y )}$. Then \[ F(T_X)^{\big ( -(e_Y-1),-( e_Y-1)\big)} = F(T_Y) ^{( -e_X, e_X)}= F(T_Z)^{\big ( -\{e_X, e_Y\}, -\{e_X, e_Y\} \big)}. \]
This is the largest submatrix shared between the two matrices, and it is an F-matrix by the proof of Proposition~\ref{thm:fmat_collapse}. 
Therefore, the tree with the F-matrix $F(T_Z)^{\big( -\{e_X, e_Y\}, -\{e_X, e_Y\} \big) }$ is the least upper bound of $T_X$ and $T_Y$. \\

\noindent The second case is when $T_X$ and $T_Y$ are both binary trees and cover $\emptyset$. Algorithm~\ref{alg:tree_binary_lub} and Proposition~\ref{prop:binary-lub} show that the unique least upper bound of any two binary trees exist. Hence, we have shown the unique LUB of $T_X, T_Y$ exists if $T_X, T_Y$ cover $T_Z \in \mathcal{MT}_N\cup \{\emptyset\}$. Hence, $T_X \vee T_Y$ exists for all $T_X, T_Y \in \mathcal{MT}_N\cup \{\emptyset\}$ and $\mathcal{L}(\mathcal{MT}_N\cup \{\emptyset\}, \wedge, \vee)$ is a lattice by Theorem~\ref{thm:poset_2}.
\end{proof}

Although Algorithm~\ref{alg:tree_lub} directly proves any two elements has a unique least upper bound, using Theorem~\ref{thm:poset_2} presents a cleaner proof that $\mathcal{L}(\mathcal{MT}_N\cup \{\emptyset\}, \leq )$ is indeed a lattice. The computation complexity of Algorithm~\ref{alg:tree_lub} is at most $O(N^3)$, since one can check column violations a maximum of $N$ times, and each check takes $O(N^2)$ operations. In practice, the computation time is less than $O(d^3)$, where $d$ is the dimension of the largest shared submatrix of $F(T_X), F(T_Y)$ and will be much smaller than $N$. The key to this lattice is that every multifurcating tree with $N$ tips has a path to the star tree with $N$ tips: sequentially collapsing the edge $(1,2)$ from the root will lead eventually to the star tree. This property also guarantees $\mathcal{MT}_N$ is connected. After defining the lattice and connecting tree properties to F-matrix properties, we can utilize the lattice to define Markov chains on $\mathcal{MT}_N$. Before we do so, we establish how to calculate the degree of any ranked tree shape to understand the graph induced by the lattice. 

\subsection{Degree of $T_{N,K}$}

Let $T_{N,K}$ be a ranked tree shape $T_{N,K}$ with $N$ tips and $K$ internal nodes. Define the backwards and forwards degree of $T_{N,K}$ respectively as  
\begin{align*}
    \text{deg}^-(T_{N,K}) &= \Big | \{ T_{N,K+1}: T_{N,K+1} \prec T_{N,K}\} \Big |, \\
    \text{deg}^+(T_{N,K}) &= \Big | \{ T_{N,K-1}: T_{N,K} \prec T_{N,K-1}\} \Big | .
\end{align*}
Visually, $\text{deg}^-(T_{N,K})$ counts how many arrows point towards $T_{N,K}$ and $\text{deg}^+(T_{N,K})$ counts how many arrows emanate from $T_{N_K}$ in the Hasse diagram. The forwards degree counts how many edges of type $(e,e+1)$ exist in $T_{N,K}$. The backwards degree of a tree counts how many ways there are to split one of the existing multifurcating events into two consecutive multifurcating events. For this purpose, we exclude $\emptyset$ which was used to construct the lattice. For example, all binary trees have backwards degree equal to 0. The star tree has backwards degree equal to $N-2$ and forwards degree equal to 0. We will also use assumption this in Section 4 when we define our Markov chains.

\begin{theorem}\label{thm:tree-deg}
The forwards degree of a tree $T_{N,K}$ is \[  \text{deg}^+(T_{N,K}) = 1 + \sum_{e=2}^{K} \mathds{1}(F_{e, 1:e-1} = F_{e+1, 1:e-1} ). \]
The backwards degree of $T_{N,K}$ is \[ \text{deg}^-(T_{N,K}) \coloneq \sum_{i=1}^K U(k_i, l_i) = \sum_{i=1}^K \Big ((l_i+1)2^{k_i} -k_i-3 + \mathds{1}(l_i=0) \Big ), \] where $l_i, k_i, i=1,...,K$ are the number of leaves and internal nodes that subtend from node $i$ respectively, and $U(k_i, l_i)$ denotes the number of ways a single collapse edge operation can result in an internal node with ranking $i$ that subtends $l_i$ leaves and $k_i$ internal nodes.
\end{theorem}

Note the values of $l_i,k_i$ are calculated directly from the string representation $\{t,l\}$ of $T_{N,K}$: $(l_1,...,l_k)$ is precisely the $l$ vector of the string representation, and $k_i = \big |\{j \in \{1,...,K\}: t_j=i\} \big |$. The proof is in Appendix~\ref{subsec:lattice-proofs}. The total degree of $T_{N,K}$ is $\text{deg}^+(T_{N,K}) + \text{deg}^-(T_{N,K})$. \\

In addition to calculating the degree of every tree, we want to find the most connected tree in $\mathcal{MT}_N$, which is tree with the maximum degree. Knowing the maximum degree, which we will denote by $M_N$, gives an upper bound to the total degree of the graph and a better understanding of graph connectivity. For $N=4$ and $N=5$, we can look at Figures~\ref{fig:lattice-4} and \ref{fig:lattice-5} to see $M_4=3$ and $M_5=5$. Figure~\ref{fig:max_backwards_degree_trees} shows the ranked tree shapes with largest total degree for $N=4,...,9$ and we see they have a similar pattern. These trees have the greatest backwards and total degree on $\mathcal{MT}_N$ and their forward degree is equal to 1. Intuitively, this tree has the maximum backwards degree because in order to maximize the number of possible splits, we want to maximize both the number of internal nodes that subtend the node, as well as the total number of multifurcations in that node (distinct nodes plus leaves). Since each internal node must subtend at least two leaves, there cannot be more than $\lfloor N/2 \rfloor + 1$ internal nodes. In terms of shape, the maximum degree tree also resembles the lodgepole tree \citep{disanto2015coalescent}, a family of binary, unranked, and unlabeled trees that grows recursively by adding a tree of two leaves to a common root. 

\begin{figure}[h]
    \centering
    \includegraphics[width=0.65\linewidth]{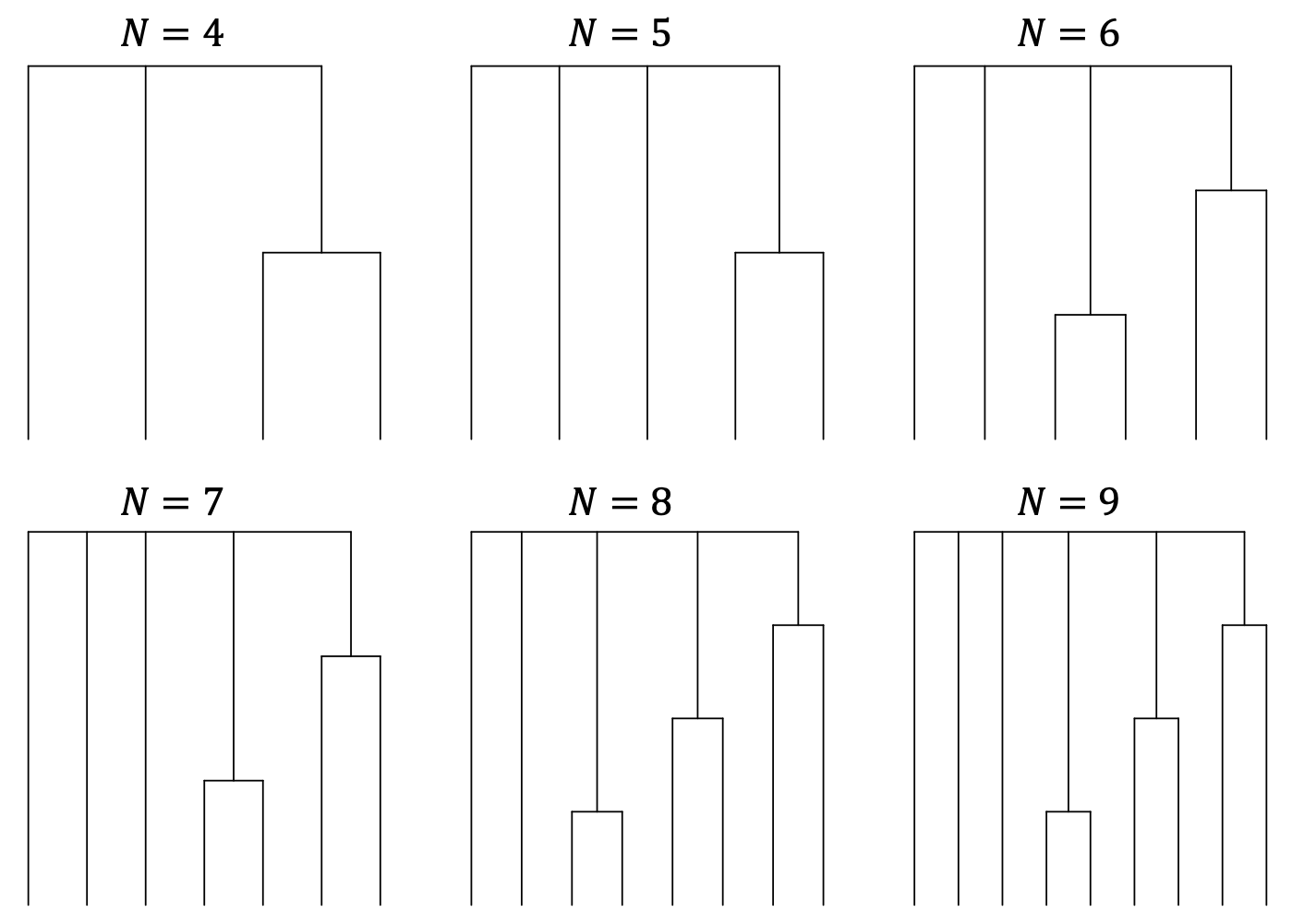}
    \caption{Tree with the maximum degree in $\mathcal{MT}_N$ for $N=4,5,6,7,8,9$. Notice that they all have a similar shape. The maximum backwards degree are $2,4,7,11,18,26$ for each tree with $N$ tips respectively. Each of these trees has forward degree 1.}
    \label{fig:max_backwards_degree_trees}
\end{figure}

We will explicitly find the tree in $\mathcal{MT}_N$ with the maximum backwards degree, and then show it has the maximum total degree. The maximization of the backwards degree (Equation~\ref{eq:backwards_degree}) can be written as the integer programming problem: \\ 
\begin{equation}\label{eq:M_opt_problem} 
\begin{split}
    \max_{\substack{\scriptstyle K=1,...,N-1 \\ l_1,...,l_K \\ k_1, ...,k_K}} & \sum_{i=1}^K (l_i+1)2^{k_i} -k_i-3 + \mathds{1}(l_i=0) \\ 
    \text{subject to } & \sum_{i=1}^K l_i= N, \;\; \sum_{i=1}^K k_i = K-1 \\
    & l_i + k_i \geq 2 \; \forall \; i=1,...,K \\
    & k_i \leq K-i, k_1 > 0 \text{ if } K \geq 2
\end{split}
\end{equation}
The first two constraints correspond to trees having  $N$ total leaves and $K-1$ non-root internal nodes. The third constraint corresponds to each internal node having at least two children. Finally, the last two constraints correspond to the $i$th internal node having at most $K-i$ internal nodes as children, and the root (first internal node) must have at least one internal node as a child. The last constraint also implies that $k_K=0$, since the last internal node can only have leaves as children. The next two results find the tree with the maximum backwards degree and prove it is the tree with the maximum total degree. The proof of Theorem~\ref{thm:max-deg-tree} is in Appendix~\ref{subsec:lattice-proofs}. 

\begin{theorem}\label{thm:max-deg-tree}
The tree $\ddot{T}_N$ with the maximum backwards degree in $\mathcal{MT}_N$ is the multifurcating tree with the following properties 
\begin{enumerate}
    \item It has $K=\lfloor \frac{N-2}{2} \rfloor+1$ internal nodes.
    \item Each of the $\lfloor \frac{N-2}{2} \rfloor$ non-root internal nodes subtend distinct cherries descending from the root.
    \item The rest of the $L=2 + (N \text{ mod } 2)$ tips are leaves descending from the root.
\end{enumerate}
The backwards degree of $\ddot{T}_N$ is \[ \text{deg}^-(\ddot{T}_N) = \big (3+ (N \text{ mod } 2) \big) 2^{\left \lfloor \frac{N}{2} \right \rfloor} - \left \lfloor \frac{N}{2} \right \rfloor - 3 .\]
\end{theorem}

\begin{theorem}
The tree $\ddot{T}_N$ with the maximum backwards degree in $\mathcal{MT}_N$ also has the maximum total degree $M_{N}=1+ \deg^-(\ddot{T}_N)$.
\end{theorem}
\begin{proof}
By construction, $\ddot{T}_N$ has forward degree equal to one because the tree has edges $(1,2), (1,3),..., (1,\lfloor \frac{N}{2}\rfloor)$, meaning \[ \text{deg}(\ddot{T}_N) = \big (3+ (N \text{ mod } 2) \big) 2^{\left \lfloor \frac{N}{2} \right \rfloor} - \left \lfloor \frac{N}{2} \right \rfloor - 2. \] 
To show that no other tree with higher forward degree will have total degree exceeding deg($\ddot{T}_{N}$), note that for every increase by 1 in the forward degree, the backwards degree will decrease by more than 1, leading to a smaller total degree. This is because for edge $(e,e+1)$, $k_e \geq 1$ which implies the other internal nodes will subtend less internal nodes, as $\sum^{K}_{i=1}k_{i}=K-1$. The backwards degree will decrease by at least $U(k, l) - U(k-1,l) = (l+1) 2^{(k-1)} -1 >1$. Hence, the tree $\ddot{T}_N$ has the maximum total degree. 
\end{proof}

The maximum degree is on the order of $M_N \sim O(2^{N/2})$, and so the total degree of the graph on $\mathcal{MT}_N$ is very loosely upper bounded by $G(N) M_N$. The maximum degree will be used in the next section to bound the mixing time of Markov chains on $\mathcal{MT}_N$. 

\section{Markov chains on multifurcating trees} \label{sec:mc}

Establishing a lattice on $\mathcal{MT}_N$ allows us to define Markov chains on the graph, where we exclude the element $\emptyset$. Here, we analyze two Markov chains (MC) in terms of their mixing times and apply them to study the empirical distribution of certain statistics of multifurcating trees. For a Markov chain on the space $\Omega$, with transition probability matrix $P$, and stationary distribution $\pi$, the \textbf{\textit{mixing time}} is defined to be \[ t^{\text{mix}}(\epsilon)= \sup_{x_0 \in \Omega} \inf \{ t>0: \lVert P^t(x_0, \cdot) - \pi(\cdot) \rVert_{TV} < \epsilon\}. \] We will compare mixing times in terms of $t^{\text{mix}} \equiv t^{\text{mix}}(1/4)$. \\

We will use various techniques to find bounds for the mixing times. The first technique for the lower bound is the diameter bound for irreducible and aperiodic MCs. Let $P$ be the transition matrix of an MC on $\Omega$ and construct the graph with vertex set $\Omega$ and edges $(x,y)$ if $P(x,y) + P(y,x) >0$. The \textbf{\textit{diameter}} $L$ of a Markov chain is the maximal graph distance between distinct vertices.
\begin{theorem}[Section 7.1.2 of \cite{levin2017markov}] \label{thm:diameter-bound}
The mixing time is lower bounded by $t^{\textup{mix}}\geq \frac{L}{2}.$
\end{theorem}
The second approach for irreducible and aperiodic MCs is based on the \textit{\textbf{bottleneck ratio}} of a set $S\subset \Omega$ defined as:
\[ \Phi(S) = \frac{Q(S, S^C)}{\pi(S)},\] 
and the bottleneck ratio of the whole Markov chain as 
\begin{equation}\label{eq:bottleneck-def}
    \Phi_* = \min_{S: \pi(S) \leq \frac{1}{2}} \Phi(S),
\end{equation}
where $Q(x,y) = \pi(x) P(x,y)$ and $Q(A,B) = \sum_{x\in A, y \in B} Q(x,y)$ is the probability of moving from $A$ to $B$ in one step when starting from the stationary distribution $\pi$. 
\begin{theorem}[Theorem 7.3 of \cite{levin2017markov}] \label{thm:bottleneck_bound}
The mixing time is lower bounded by $t^{\textup{mix}} \geq \frac{1}{4\Phi_*}.$
\end{theorem}

A useful technique for deriving upper bounds of the mixing time of a reversible MC is based on the \textit{\textbf{relaxation time}}, defined as $t_{\text{rel}} = \frac{1}{\gamma_*}$. The absolute spectral gap $\gamma_{*}:=1-\lambda_{*}$ is a function of $\lambda_{*}$, the largest eigenvalue of $P$, in absolute value, that is different from 1. For a lazy chain, all the eigenvalues are positive, so $\gamma_* = \gamma=1-\lambda_{2}$, where $\lambda_{2}$ is the second largest eigenvalue of $P$. However for non-lazy chains, $\gamma^* \leq \gamma$. We use the following two theorems: 

\begin{theorem}[Theorem 12.3 of \cite{levin2017markov}] \label{thm:mix_rel_bound}
Let $P$ be the transition matrix of a reversible, irreducible Markov chain with state space $\Omega$. Then \[ t^{\textup{mix}}(\epsilon) \leq \log\left ( \frac{1}{\epsilon \min_{x\in\Omega} \pi(x)} \right) t_{\textup{rel}}. \]
\end{theorem}

\begin{theorem}[Theorem 13.14 of \cite{levin2017markov}] \label{thm:bottleneck_rel_bound}
Let $P$ be a reversible transition matrix. Then the spectral gap satisfies \[ \frac{\Phi_*^2}{2}\leq \gamma \leq 2\Phi_*. \]
\end{theorem}
\noindent Computing the upper bound of the mixing time also utilizes the bottleneck ratio. Because the relaxation time of a lazy chain is equal to $1/\gamma$, both upper bounds we derive will be for a lazy version of the chains.

\subsection{Symmetric lattice Markov chain}

The first Markov chain we define moves from $T_{X}$ to any neighbor $T_{Y}$ with equal probability as follows: 
\begin{align*}
    P(T_X, T_Y) &= \begin{cases} \frac{1}{M_N} & \text{ if } T_X \prec T_{Y} \text{ or } T_Y \prec T_{X}, \\
    1-\frac{\text{ deg}(T_{X})}{M_{N}} & \text{ if } T_{Y}=T_{X},\\
    0 & \text{ otherwise}. \end{cases}
\end{align*}
We call this Markov chain the symmetric lattice MC. It is symmetric by definition and irreducible because you can always trace a path between any two trees through their least upper bound. Therefore it has a uniform stationary distribution. The MC is also aperiodic because there exists $T_X \in \mathcal{MT}_N$ such that $P(T_X, T_X) >0$. \\

Let $S \subset \mathcal{MT}_N$, then the bottleneck ratio of the symmetric lattice MC with uniform stationary distribution is  
\begin{align*}
    \Phi(S) &= \frac{\sum_{T_X\in S, T_Y \in S^C} \pi(T_X) P(T_X,T_Y) }{\pi(S)} \\
    &= \frac{\sum_{T_{X}\in S, T_{Y} \in S^c} \mathds{1} (T_{X} \prec T_{y} \text{ or } T_{Y} \prec T_{X}) }{M_N |S|} 
\end{align*}
Now, since each element in $S$ must be connected to at least one other tree, then $\Phi(S) \geq \frac{1}{M_N}$. The minimum is then achieved when $S$ is a single tree with degree 1. One such tree is the binary tree whose edges are $(1,2)$ and $\{(i, i+2)\}_{i=1}^{N-3}$. By Theorem~\ref{thm:bottleneck_bound}, a lower bound on the mixing time of the lattice MC on $\mathcal{MT}_N$ is 
\begin{equation} \label{eq:mixing_lb}
    t^{\textup{mix}}_N \geq  \frac{M_N}{4} \sim O(2^{N/2}).
\end{equation} 

For an upper bound of the mixing time, we consider the lazy version of the symmetric lattice MC, where at each step we generate a random coin flip $X \sim \text{Bernoulli}(1/2)$. If $X=1$, then move according to the transition probability of the symmetric chain, and if $X=0$, the chain does not move. In this case, $t_{\text{rel}} = \frac{1}{\gamma}$ and the bottleneck ratio of the lazy lattice MC is $\Phi_*^l = \frac{1}{2M_N}$. By Theorem~\ref{thm:bottleneck_rel_bound}, \[ t_{\text{rel}} = \frac{1}{\gamma} \leq \frac{2}{(\Phi_*^l)^2} = 8 M_N^2 .\] Applying Theorem~\ref{thm:mix_rel_bound} gives the upper bound of the mixing time of the lazy lattice MC 
\begin{equation} \label{eq:mixing_ub}
    \tilde{t}^{\textup{mix}}_N \leq 8 M_N^2  \log(4 G(N)) \sim O\big (2^N \times N \log N\big ), 
\end{equation}
where $G(N)$ is the cardinality of $\mathcal{MT}_N$. The symmetric lattice MC mixes very slowly, and so we introduce the second Markov chain. 

\subsection{Random walk on the lattice}

Since the lattice induces a graph structure on $\mathcal{MT}_N$, we can naturally define a random walk on $\mathcal{MT}_N$: the transition probability is \[  P(T_X, T_Y) = \begin{cases} \frac{1}{\text{deg}(T_X)} & \text{ if } T_X \prec T_Y \text{ or } T_Y \prec T_X \\ 0 & \text{ otherwise} \end{cases}. \] Then, the stationary distribution is \[ \pi(T_X) =  \frac{\text{deg}(T_X)}{\sum_{T \in \mathcal{MT}_N} \text{deg}(T)} \quad \forall \; T_X \in \mathcal{MT}_N .\] 
Notice the random walk has reflecting boundaries: if the Markov chain is at a binary tree, it must transition to a tree with $N-2$ internal nodes and if the Markov chain is at the star tree, it must transition to a tree with two internal nodes. Again, we want to find bounds on the mixing time of this chain. Since there are no self-transitioning states, we expect a faster mixing time for this chain. \\ 

We use the diameter bound (Theorem~\ref{thm:diameter-bound}) to lower bound the mixing time. For two binary trees, the maximum possible least upper bound is the tree with the F-matrix $\begin{pmatrix}
    N-1 & 0 \\
    N-2 & N
\end{pmatrix}$. This implies $L \geq 2(N-3)$ because the graph distance of any binary tree to this tree is $N-3$, and
\begin{equation} \label{eq:rw_mixing_lb}
    t^{\textup{mix}}_N \geq \frac{2(N-3)}{2} \sim O(N)
\end{equation}

For the upper bound, we consider the lazy version of the chain with a fair coin as before and use the bottleneck ratio approach. The bottleneck ratio of a set $S$ for this random walk  is \[ \Phi(S) =  \frac{\sum_{T_{X}\in S, T_{Y} \in S^c} \mathds{1}(T_{X} \prec T_{Y} \text{ or }T_{Y}\prec T_{X})}{\sum_{T \in S} \text{deg}(T)}\] by Equation (7.7) of \cite{levin2017markov}. Notice that $\Phi(S)=1$ if $S$ is the single binary tree with degree 1, and this must be the minimum because the denominator cannot be smaller than the numerator. This implies $\Phi_*=1$ and the bottleneck ratio for the lazy version is $\Phi^l_*=1/2$. Combining Theorems~\ref{thm:mix_rel_bound} and \ref{thm:bottleneck_rel_bound} gives \[ \tilde{t}^{\textup{mix}}_N \leq 8 \log\left (4 \sum_{T \in \mathcal{MT}_N} \text{deg}(T) \right ). \]
We do not explicitly know the total degree of the graph,, but it can be very loosely upper bounded by $M_N G(N)$. Therefore, the mixing time for the lazy random walk is upper bounded by 
\begin{equation} \label{eq:rw_mixing_ub}
    \tilde{t}^{\textup{mix}}_N \leq 8 \log (4 M_N G(N)) \sim O(N\log N). 
\end{equation}
The mixing time of this Markov chain is significantly faster than the symmetric lattice MC. In fact, the mixing time of this Markov chain is faster than the $O(N^2)$ mixing time of Aldous' random transpositions chain on cladograms \citep{aldous2000mixing,schweinsberg2002n2}. 

\subsection{Simulation results} 
We can use the Markov chains defined on the lattice to sample from any distribution on $\mathcal{MT}_N$ via Metropolis-Hastings. Specifically, we are interested in sampling from the uniform distribution on $\mathcal{MT}_N$, and we use the random walk MC as the proposal distribution: $g(T' |T_X) = \frac{1}{\text{deg}(T_X)}$. We compare the distribution of certain tree statistics under uniform distribution of multifurcating ranked tree shapes to the Beta$(1,1)$-coalescent, also called Bolthausen-Sznitman coalescent \citep{bolthausen1998ruelle}. \\

The Beta$(1,1)$-coalescent belongs to a class of multiple merger coalescent models called the $\Lambda$-coalescent, which generates timed multifurcating ranked tree shapes  \citep{sagitov1999general,pitman1999coalescents}. The measure $\Lambda$ governs the rate at which lineages merge back in time: $\lambda_{b,k}, b\geq k$ is the rate at which any $k$-tuple of lineages merge to form a single lineage (forms a block of size $k$) when there are $b$ total lineages: \[ \lambda_{b,k} =\int_0^1 x^{k-2} (1-x)^{b-k} \Lambda(dx) .\] Here, the measure $\Lambda$ affects both the tree topology and the branch lengths because the total rate at which a coalescent event occurs is $\sum_{k=2}^b \binom{b}{k} \lambda_{b,k}$. Conditional on a coalescent event, $k$ is chosen proportional to $\binom{b}{k} \lambda_{b,k}$. One common sub-family used for inference is the Beta-coalescent, where $\Lambda=$ Beta$(2-\alpha, \alpha)$ for $\alpha \in (1,2]$. \cite{zhang2025multiple} showed that the parameter $\alpha$, which now characterizes the coalescent process, can be inferred from the block sizes only, ignoring the branch lengths. \\ 

The Beta$(1,1)$-coalescent is the $\Lambda$-coalescent with base measure $\Lambda=U[0,1]=\text{Beta}(1,1)$, and arises in connection to spin glass theory in statistical physics. It is the only $\Lambda$-coalescent where the branch lengths are independent from the topology, and can be characterized by random recursive trees \citep{Goldschmidt2015}. Therefore, it is a suitable model for comparing the topology of trees because the topology is not confounded with the branch lengths. We do not expect the Beta$(1,1)$-coalescent to generate the same distribution of tree shapes as the uniform distribution, but rather it is an interesting point of comparison to see how the tree probability spaces differ. \\

We compare the two empirical distributions of the following tree statistics:
\begin{enumerate}
    \item The number of internal nodes $K$. 
    \item Maximum block size $M$: For each tree, we can find the size of the largest coalescent event, which is a measure of tree shape. It is constrained by the number of internal nodes in the tree, for example, binary trees only have block sizes of size 2. 
    \item Average block size $A$: The average size of all coalescent events in a tree. Again, for binary trees this measure will always be 2. 
    \item Average number of $m$-tip cherries: A cherry, or 2-cherry,  is a subtree with two leaves. The expected number of cherries under the Yule model for binary tree shapes is $N/3$ with variance of $2N/45$ while under the uniform distribution on binary tree shapes, it is $\approx N/4$ with variance of $\approx N/16$ \citep{McKenzie2000}. Since we are working with multifurcating trees, we can look at cherries with more than two leaves. We will consider $m=2,3,4$ since higher numbers of $m$ are lower probability events. 
\end{enumerate}

We will simulate multifurcating trees with $N=20$ and $N=100$ tips. For each $N$, we run $N-1$ parallel chains for $200N$ iterations and pool all chains together to get our sample from the uniform distribution on $\mathcal{MT}_N$, 
(see Algorithm~\ref{alg:semi_random_sampling} in Appendix~\ref{sec:semi-random} for how we initialized the chains). The computational time for one chain on trees of 100 tips with 20,000 iterations takes around 25 minutes, with the computation time scaling linearly in $N$ and the number of iterations. The total acceptance rate of the Metropolis-Hastings sampling is approximately 90\% for $N=20$ and 97\% for $N=100$. Trace plots indicate well-mixing and average effective sample size per chain is more than 400 for $N=20$ and 700 for $N=100$. For the Beta$(1,1)$-coalescent trees, we generate $2000N$ samples directly using \texttt{rbeta\_coal()} from the R package \texttt{phylodyn} \citep{karcher2017phylodyn}.\\ 

Figure~\ref{fig:sim_tree_statistics} compares three different tree statistics between the Beta$(1,1)$-coalescent and the uniform distribution: the number of internal nodes, the maximum block size, and the average block size. The uniform distribution has a greater mass on trees with more internal nodes while the Beta$(1,1)$-coalescent being roughly bell-shaped and centered at less than $N/2$. Table~\ref{tab:summary_K_M_A_sims} presents the mean and median values of $K$ and $M$ per distribution and $N$. We see that median value of $K$ for $N=100$ under the uniform distribution is 85 while it is only 30 for the Beta$(1,1)$-coalescent. Both distributions for the maximum block size are right skewed, but the modal value of $M$ is smaller for the uniform distribution than for the Beta$(1,1)$-coalescent, which makes sense because the uniform distribution generates trees with more internal nodes. From Table~\ref{tab:summary_K_M_A_sims}, the median and mean $M$ for the Beta$(1,1)$-coalescent grows substantially with $N$, while under the uniform distribution it stays between 3-4. The last panel of Figure~\ref{fig:sim_tree_statistics} shows the distribution of the average block size, where there is not much variability due to the statistic already being an average. Overall, the Beta$(1,1)$-coalescent still has a larger average block size, which is consistent with more trees having a lesser number of internal nodes. The mean/median of the average block size is only slightly more than 2 (the average block size for binary trees) in the uniform case, while for the Beta$(1,1)$-coalescent, the mean/median are around 3 for $N=20$ and 4.5 for $N=100$. 

\begin{figure}[H]
    \centering
    \includegraphics[width=0.9\linewidth]{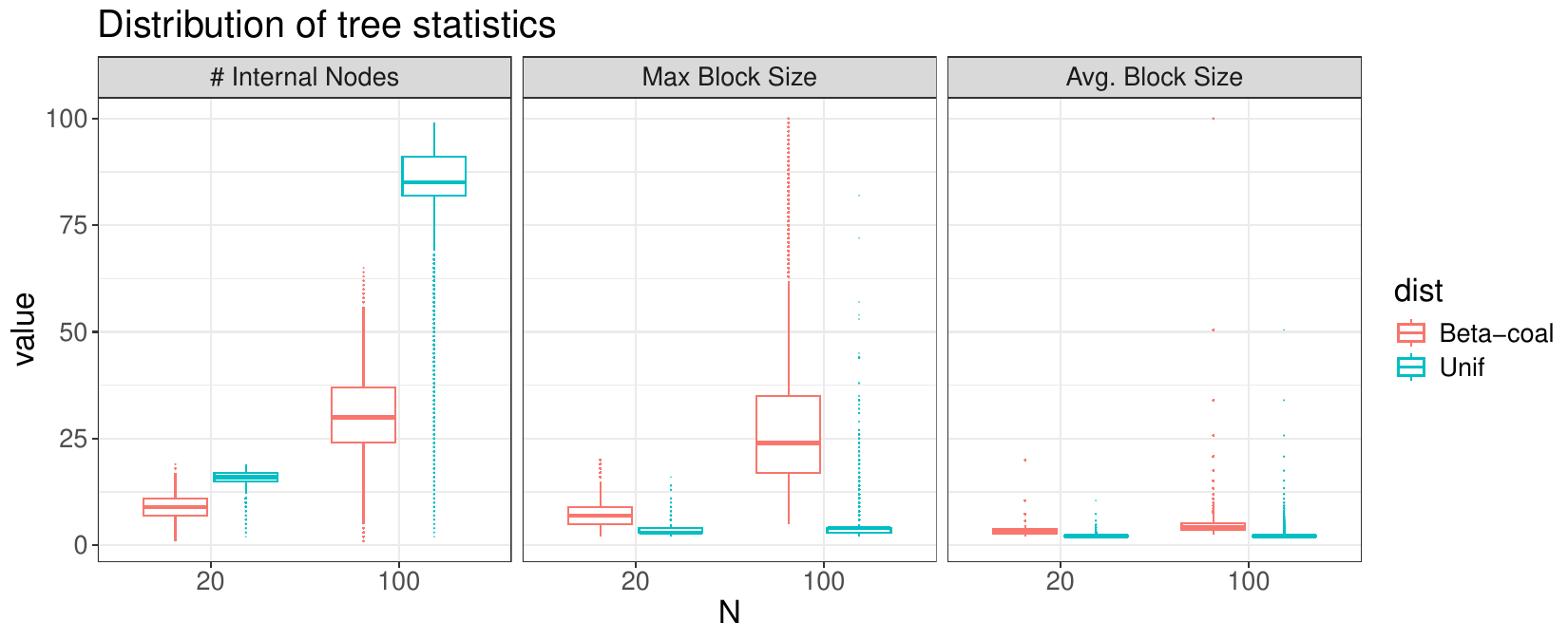}
    \caption{\textbf{Boxplots of three tree statistics in simulated $\mathcal{MT}_N$ samples:} The panels display the box plots for the number of internal nodes, the maximum block size, and the average block size of trees in each simulated sample, with red being samples from the Beta$(1,1)$-coalescent and blue being the samples drawn from the uniform distribution using Metropolis-Hastings. The Beta$(1,1)$-coalescent generates more trees with fewer number of internal nodes and larger maximum block size compared to the uniform distribution. The average block size distribution is more similar and less variable between the two samples.}
    \label{fig:sim_tree_statistics}
\end{figure}

\begin{table}[H]
\begin{tabular}{r|lrrrrrr}
                       & Distribution      & Mean $K$ & Median $K$ & Mean $M$ & Median $M$ & Mean $A$ & Median $A$\\  \hline
\multirow{2}{*}{$N=20$}  & Beta-coal & 9.09 & 9 & 7.74 & 7 & 3.45 & 3.11 \\ 
                        & Uniform      & 16.14 & 16 & 3.31 & 3 & 2.19 & 2.19 \\ \hline 
\multirow{2}{*}{$N=100$} & Beta-coal & 30.21 & 30 & 28.16 & 24 & 4.79 & 4.3 \\ 
                        & Uniform      &  86.03 & 85 & 3.92 & 4 & 2.16 & 2.16 \\ 
\end{tabular}
\caption{\textbf{Summary statistics of the internal node distribution $K$, the maximum $M$ and average block size $A$ distributions:} The means and medians of $K, M, A$ are overall very similar. In the Beta$(1,1)$-coalescent case, the values of $K$ and $M$ both grow in $N$, but there is not much increase in the average block size. On the other hand, the growth rate of the mean/median of $K$ in the uniform sample is much faster, but the mean/median of $M$ and $A$ stay essentially constant in $N$. }
\label{tab:summary_K_M_A_sims}
\end{table}

The last metric we calculate on the tree samples is the cherry spectrum. Figure~\ref{fig:sim_tree_cherry_statistics} illustrates the average number of $m$-tip cherries standardized by $N$. Overall, the uniform distribution produces more 2-tip and 3-tip cherries than the Beta$(1,1)$-coalescent, while it is the opposite for 4-tip cherries. The expected number of 2-tip cherries in the uniform sample is around $0.31N$, which is larger than $\approx 0.25 N$ in the uniform distribution on binary ranked tree shapes. The expected number of 3-tip cherries is $\approx 0.045N$ for $m=3$, and $\approx 0.008N$ for 4-tip cherries. In general, we have demonstrated that the uniform distribution on $\mathcal{MT}_N$ and the Beta$(1,1)$-coalescent are very different. 

\begin{figure}[H]
    \centering
    \includegraphics[width=0.9\linewidth]{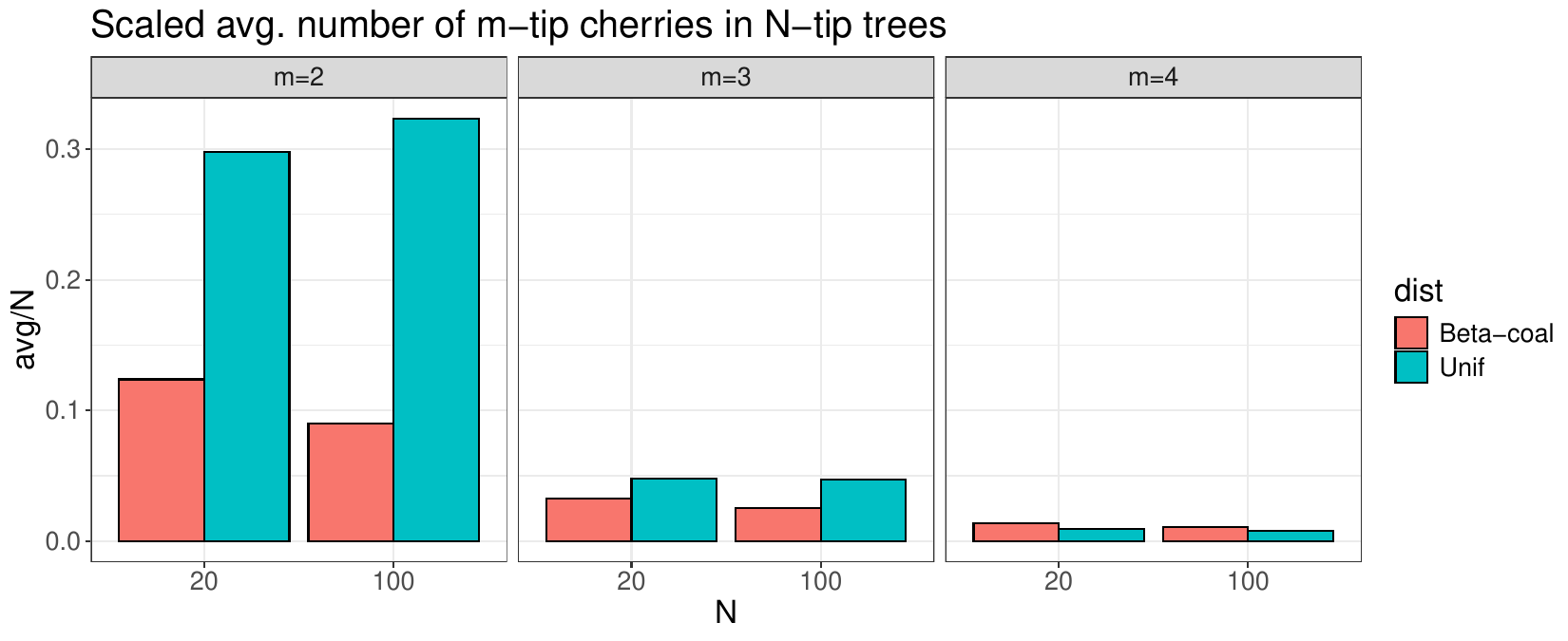}
    \caption{\textbf{Scaled average number of $m$-tip cherries:} The average number of $m$-tip cherries is computed per sample and then divided by $N$ for ease of comparison. The uniform distribution on $\mathcal{MT}_N$ produces more 2-tip and 3-tip cherries, while the Beta$(1,1)$-coalescent has more 4-tip cherries.}
    \label{fig:sim_tree_cherry_statistics}
\end{figure}

\section{Discussion} \label{sec:discussion}

We introduced two representations of multifurcating ranked tree shapes: the string representation (Proposition~\ref{prop:string-bijection}), and the F-matrix representation (Theorem~\ref{thm:fmat_multi}) which were used to enumerate $\mathcal{MT}_N$ and construct a lattice on the space, respectively. The crux of the lattice construction (Section~\ref{sec:lattice}) was to define the partial ordering on $\mathcal{MT}_N$ using the operation of collapsing edges $(e,e+1)$ connecting consecutive internal nodes. There is a nice duality between collapsing edge $(e,e+1)$ and deleting row and column $e$ of the corresponding F-matrix. This duality allows us to prove several results about the lattice such as the maximum degree (Theorem~\ref{thm:max-deg-tree}). Although, the collapse edge operation can also be performed using the string representation (see Appendix~\ref{sec:string-collapse}),  it is not as intuitive. Collapsing edges between consecutive internal nodes is also used in \cite{collienne2021computing} for defining the nearest neighbor interchange (NNI) move between ranked labeled binary trees. \\ 

While the F-matrix encoding of ranked binary trees was proposed to define tree distances, we cannot directly use the multifurcating F-matrix to define distances between multifurcating trees with the same number of tips. Multifurcating trees with different number of internal nodes will have F-matrices of different dimensions. However, we can use the lattice on $\mathcal{MT}_N$ to define a graph distance. Graph distances are often defined by minimum path length, such as in \cite{collienne2021computing} for ranked, labeled binary trees and \cite{collienne2024ranked} for ranked binary trees connected via the subtree prune and regraft (SPR) and NNI operations. However, finding minimum path lengths is usually very computationally costly. \\

An alternative is to define a distance between two multifurcating trees with $N$ tips as the minimal path distance of two trees to their least upper bound. More formally, let $T_X, T_Y \in \mathcal{MT}_N$  with $K_X, K_Y$ internal nodes respectively. Let LUB$(T_X,T_Y)$ be their least upper bound with $K_{XY}$ internal nodes. Define the lattice distance between $T_X, T_Y$ by $d_L(T_X, T_Y) = (K_X-K_{XY}) + (K_Y-K_{XY})$. By definition, $K_{XY} \leq K_X, K_Y$, so this is a symmetric non-negative function with $d_L(T_X, T_Y)=0$ if and only if $T_X=T_Y$. By definition, the triangle inequality holds: the distance from $T_X$ to $T_Y$ is the shortest path from $T_X$ to $T_Y$ without going ``backwards'' in the lattice which must be less than or equal to the shortest path from $T_X$ to $T_Z$ plus the shortest path from $T_Y$ to $T_Z$ without going ``backwards''. Hence this is a valid distance metric on $\mathcal{MT}_N$. Note that for any $T_X, T_Y$ with $K_X>K_Y$, we have $K_X-K_Y \leq d_L(T_X, T_Y) \leq K_X+K_Y-2$. If $T_X, T_Y$ have the same number of internal nodes, then $2 \leq d_L(T_X, T_Y)$. One potential drawback of this distance is that it is not granular enough: lattice distances only take on integer values between 1 and $2N-4$, and the distance between any two trees is further bounded by their number of internal nodes. However, other commonly used distances, such as the SPR distance between two rooted binary phylogenetic trees \citep{bordewich2005computational}, are also linear in $N$. The computation of the SPR distance is NP-hard, while the computation of our lattice distance can be done in less than cubic time. Assessing whether the proposed distance constitutes a meaningful metric for inference and analyses lies beyond the scope of this manuscript and is left for future research. \\ 

We use the lattice on $\mathcal{MT}_N$ to define two Markov chains: a symmetric walk and random walk on the lattice. We constructed our first Markov chain to have a uniform stationary distribution, so that we were able to prove mixing time bounds with more ease. The key to defining the transition probabilities was finding the $N$-tip tree with maximum degree. We showed this tree has a specific topology and derived an explicit formula for its degree. On the other hand, the random walk on the graph is a classic Markov chain with a much faster mixing time than the first symmetric Markov chain. Additional Markov chains can be defined on $\mathcal{MT}_N$ based on the graph structure of the lattice for other purposes. All the code developed as part of this paper can be found in the R package \texttt{phylodyn} \citep{karcher2017phylodyn}. \\ 

It is also possible to do a stochastic enumeration of the space of multifurcating ranked tree shapes using the tools provided in this manuscript, specifically using MCMC or sequential importance sampling (SIS) methods \citep{sinclair2012algorithms}. \cite{cappello2020sequential} propose an SIS method for counting constrained multi-resolution Kingman–Tajima coalescent trees. \\ 

This manuscript considered only isochronous multifurcating ranked tree shapes, however, the multifurcating F-matrix can be defined for heterochronous multifurcating ranked tree shapes, similar to the F-matrix for heterochronous binary ranked tree shapes defined in \cite{Kim2020}. Heterochronous genealogies are used to model rapidly evolving organisms, such as viruses, and it is left as future research. \\


\newpage
\bibliographystyle{abbrvnat}
\bibliography{multifurcating}

\begin{thebibliography}{68}
\providecommand{\natexlab}[1]{#1}
\providecommand{\url}[1]{\texttt{#1}}
\expandafter\ifx\csname urlstyle\endcsname\relax
  \providecommand{\doi}[1]{doi: #1}\else
  \providecommand{\doi}{doi: \begingroup \urlstyle{rm}\Url}\fi

\bibitem[Aldous(2000)]{aldous2000mixing}
D.~J. Aldous.
\newblock Mixing time for a {Markov} chain on cladograms.
\newblock \emph{Combinatorics, Probability and Computing}, 9\penalty0 (3):\penalty0 191--204, 2000.

\bibitem[{\'A}rnason et~al.(2023){\'A}rnason, Koskela, Halld{\'o}rsd{\'o}ttir, and Eldon]{arnason2023sweepstakes}
E.~{\'A}rnason, J.~Koskela, K.~Halld{\'o}rsd{\'o}ttir, and B.~Eldon.
\newblock Sweepstakes reproductive success via pervasive and recurrent selective sweeps.
\newblock \emph{Elife}, 12:\penalty0 e80781, 2023.

\bibitem[Batterson(2011)]{batterson2011competition}
J.~Batterson.
\newblock \emph{Competition math for middle school}.
\newblock Art of Problem Solving Alpine, CA, 2011.

\bibitem[Berestycki et~al.(2007)Berestycki, Berestycki, and Schweinsberg]{berestycki2007beta}
J.~Berestycki, N.~Berestycki, and J.~Schweinsberg.
\newblock Beta-coalescents and continuous stable random trees.
\newblock \emph{The Annals of Probability}, pages 1835--1887, 2007.

\bibitem[Bj{\"o}rner et~al.(1990)Bj{\"o}rner, Edelman, and Ziegler]{bjorner1990hyperplane}
A.~Bj{\"o}rner, P.~H. Edelman, and G.~M. Ziegler.
\newblock Hyperplane arrangements with a lattice of regions.
\newblock \emph{Discrete \& computational geometry}, 5\penalty0 (3):\penalty0 263--288, 1990.

\bibitem[Bolthausen and Sznitman(1998)]{bolthausen1998ruelle}
E.~Bolthausen and A.-S. Sznitman.
\newblock On ruelle's probability cascades and an abstract cavity method.
\newblock \emph{Communications in mathematical physics}, 197:\penalty0 247--276, 1998.

\bibitem[Bordewich and Semple(2005)]{bordewich2005computational}
M.~Bordewich and C.~Semple.
\newblock On the computational complexity of the rooted subtree prune and regraft distance.
\newblock \emph{Annals of combinatorics}, 8:\penalty0 409--423, 2005.

\bibitem[Byeon et~al.(2019)Byeon, Oh, Kim, Yun, Kang, Park, and Lee]{byeon2019origin}
S.~Y. Byeon, H.~Oh, S.~Kim, S.~H. Yun, J.~H. Kang, S.~R. Park, and H.~J. Lee.
\newblock The origin and population genetic structure of the ‘golden tide’ seaweeds, \textit{Sargassum horneri}, in {Korean} waters.
\newblock \emph{Scientific Reports}, 9\penalty0 (1):\penalty0 7757, 2019.

\bibitem[Callan(2009)]{callan2009note}
D.~Callan.
\newblock A note on downup permutations and increasing 0-1-2 trees.
\newblock \emph{preprint}, 2009.

\bibitem[Cappello and Palacios(2020)]{cappello2020sequential}
L.~Cappello and J.~A. Palacios.
\newblock Sequential importance sampling for multiresolution {Kingman-Tajima} coalescent counting.
\newblock \emph{The Annals of Applied Statistics}, 14\penalty0 (2):\penalty0 727, 2020.

\bibitem[Cayley(1856)]{cayley1856note}
A.~Cayley.
\newblock Note sur une formule pour la reversion des séries.
\newblock \emph{Journal für die reine und angewandte Mathematik}, 1856\penalty0 (52):\penalty0 276--284, 1856.
\newblock \doi{doi:10.1515/crll.1856.52.276}.

\bibitem[Chen et~al.(2009)Chen, Ford, and Winkel]{chen2009}
B.~Chen, D.~Ford, and M.~Winkel.
\newblock {A new family of Markov branching trees: the alpha-gamma model}.
\newblock \emph{Electronic Journal of Probability}, 14:\penalty0 400 -- 430, 2009.

\bibitem[Collienne and Gavryushkin(2021)]{collienne2021computing}
L.~Collienne and A.~Gavryushkin.
\newblock Computing nearest neighbour interchange distances between ranked phylogenetic trees.
\newblock \emph{Journal of Mathematical Biology}, 82\penalty0 (1):\penalty0 8, 2021.

\bibitem[Collienne et~al.(2024)Collienne, Whidden, and Gavryushkin]{collienne2024ranked}
L.~Collienne, C.~Whidden, and A.~Gavryushkin.
\newblock Ranked subtree prune and regraft.
\newblock \emph{Bulletin of Mathematical Biology}, 86\penalty0 (3):\penalty0 24, 2024.

\bibitem[Copeland(2010)]{copeland2010mathemagical}
T.~Copeland.
\newblock Mathemagical forests, 2010.
\newblock URL \url{https://tcjpn.wordpress.com/wp-content/uploads/2008/06/mathemagicalforestswp.pdf}.

\bibitem[Der et~al.(2012)Der, Epstein, and Plotkin]{der2012dynamics}
R.~Der, C.~Epstein, and J.~B. Plotkin.
\newblock Dynamics of neutral and selected alleles when the offspring distribution is skewed.
\newblock \emph{Genetics}, 191\penalty0 (4):\penalty0 1331--1344, 2012.

\bibitem[Diaconis and Holmes(2002)]{diaconis2002}
P.~Diaconis and S.~Holmes.
\newblock {Random Walks on Trees and Matchings}.
\newblock \emph{Electronic Journal of Probability}, 7:\penalty0 1 -- 17, 2002.

\bibitem[Diaconis and Holmes(1998)]{diaconis1998matchings}
P.~W. Diaconis and S.~P. Holmes.
\newblock Matchings and phylogenetic trees.
\newblock \emph{Proceedings of the National Academy of Sciences}, 95\penalty0 (25):\penalty0 14600--14602, 1998.

\bibitem[Dickey and Rosenberg(2025)]{dickey2025labelled}
E.~H. Dickey and N.~A. Rosenberg.
\newblock Labelled histories with multifurcation and simultaneity.
\newblock \emph{Philosophical Transactions B}, 380\penalty0 (1919):\penalty0 20230307, 2025.

\bibitem[Disanto and Rosenberg(2015)]{disanto2015coalescent}
F.~Disanto and N.~A. Rosenberg.
\newblock Coalescent histories for lodgepole species trees.
\newblock \emph{Journal of Computational Biology}, 22\penalty0 (10):\penalty0 918--929, 2015.

\bibitem[Donaghey(1975)]{donaghey1975alternating}
R.~Donaghey.
\newblock Alternating permutations and binary increasing trees.
\newblock \emph{Journal of Combinatorial Theory, Series A}, 18\penalty0 (2):\penalty0 141--148, 1975.

\bibitem[Edwards(1970)]{edwards1970estimation}
A.~W. Edwards.
\newblock Estimation of the branch points of a branching diffusion process.
\newblock \emph{Journal of the Royal Statistical Society: Series B (Methodological)}, 32\penalty0 (2):\penalty0 155--164, 1970.

\bibitem[Eldon(2020)]{eldon2020evolutionary}
B.~Eldon.
\newblock Evolutionary genomics of high fecundity.
\newblock \emph{Annual Review of Genetics}, 54:\penalty0 213--236, 2020.

\bibitem[Eldon and Stephan(2018)]{eldon2018evolution}
B.~Eldon and W.~Stephan.
\newblock Evolution of highly fecund haploid populations.
\newblock \emph{Theoretical Population Biology}, 119:\penalty0 48--56, 2018.

\bibitem[Eldon and Stephan(2023)]{eldon2023sweepstakes}
B.~Eldon and W.~Stephan.
\newblock Sweepstakes reproduction facilitates rapid adaptation in highly fecund populations.
\newblock \emph{Molecular Ecology}, 2023.

\bibitem[Eldon and Wakeley(2006)]{eldon2006coalescent}
B.~Eldon and J.~Wakeley.
\newblock Coalescent processes when the distribution of offspring number among individuals is highly skewed.
\newblock \emph{Genetics}, 172\penalty0 (4):\penalty0 2621--2633, 2006.

\bibitem[Felsenstein(1978)]{felsenstein1978number}
J.~Felsenstein.
\newblock The number of evolutionary trees.
\newblock \emph{Systematic zoology}, 27\penalty0 (1):\penalty0 27--33, 1978.

\bibitem[Friedman et~al.(2001)Friedman, Ninio, Pe'er, and Pupko]{friedman2001structural}
N.~Friedman, M.~Ninio, I.~Pe'er, and T.~Pupko.
\newblock A structural {EM} algorithm for phylogenetic inference.
\newblock In \emph{Proceedings of the fifth annual international conference on Computational biology}, pages 132--140, 2001.

\bibitem[Goldschmidt and Martin(2005)]{Goldschmidt2015}
C.~Goldschmidt and J.~Martin.
\newblock {Random Recursive Trees and the {Bolthausen-Sznitman} Coalesent}.
\newblock \emph{Electronic Journal of Probability}, 10\penalty0 (none):\penalty0 718 -- 745, 2005.

\bibitem[Haas et~al.(2008)Haas, Miermont, Pitman, and Winkel]{haas2008}
B.~Haas, G.~Miermont, J.~Pitman, and M.~Winkel.
\newblock {Continuum tree asymptotics of discrete fragmentations and applications to phylogenetic models}.
\newblock \emph{The Annals of Probability}, 36\penalty0 (5):\penalty0 1790 -- 1837, 2008.

\bibitem[Helekal et~al.(2025)Helekal, Koskela, and Didelot]{helekal2025inference}
D.~Helekal, J.~Koskela, and X.~Didelot.
\newblock Inference of multiple mergers while dating a pathogen phylogeny.
\newblock \emph{Systematic Biology}, page syaf003, 2025.

\bibitem[Hoscheit and Pybus(2019)]{hoscheit2019multifurcating}
P.~Hoscheit and O.~G. Pybus.
\newblock The multifurcating skyline plot.
\newblock \emph{Virus Evolution}, 5\penalty0 (2):\penalty0 vez031, 2019.

\bibitem[Karcher et~al.(2017)Karcher, Palacios, Lan, and Minin]{karcher2017phylodyn}
M.~D. Karcher, J.~A. Palacios, S.~Lan, and V.~N. Minin.
\newblock phylodyn: an {R} package for phylodynamic simulation and inference.
\newblock \emph{Molecular Ecology Resources}, 17\penalty0 (1):\penalty0 96--100, 2017.

\bibitem[Kim et~al.(2020)Kim, Rosenberg, and Palacios]{Kim2020}
J.~Kim, N.~A. Rosenberg, and J.~A. Palacios.
\newblock Distance metrics for ranked evolutionary trees.
\newblock \emph{Proceedings of the National Academy of Sciences}, 117\penalty0 (46):\penalty0 28876--28886, 2020.

\bibitem[Kingman(1982)]{Kingman1982}
J.~F.~C. Kingman.
\newblock The coalescent.
\newblock \emph{Stochastic Processes and their Applications}, 13\penalty0 (3):\penalty0 235--248, 1982.

\bibitem[Korfmann et~al.(2024)Korfmann, Sellinger, Freund, Fumagalli, and Tellier]{korfmann2024simultaneous}
K.~Korfmann, T.~P.~P. Sellinger, F.~Freund, M.~Fumagalli, and A.~Tellier.
\newblock Simultaneous inference of past demography and selection from the ancestral recombination graph under the beta coalescent.
\newblock \emph{Peer Community Journal}, 4, 2024.

\bibitem[Kuznetsov et~al.(1994)Kuznetsov, Pak, and Postnikov]{kuznetsov1994increasing}
A.~G. Kuznetsov, I.~M. Pak, and A.~E. Postnikov.
\newblock Increasing trees and alternating permutations.
\newblock \emph{Russian Mathematical Surveys}, 49\penalty0 (6):\penalty0 79, 1994.

\bibitem[Levin and Peres(2017)]{levin2017markov}
D.~A. Levin and Y.~Peres.
\newblock \emph{Markov chains and mixing times}, volume 107.
\newblock American Mathematical Soc., 2017.
\newblock URL \url{https://www.stat.berkeley.edu/~aldous/260-FMIE/Levin-Peres-Wilmer.pdf}.

\bibitem[Lewis et~al.(2005)Lewis, Holder, and Holsinger]{lewis2005polytomies}
P.~O. Lewis, M.~T. Holder, and K.~E. Holsinger.
\newblock Polytomies and {Bayesian} phylogenetic inference.
\newblock \emph{Systematic biology}, 54\penalty0 (2):\penalty0 241--253, 2005.

\bibitem[Li et~al.(2017)Li, Grassly, and Fraser]{li2017quantifying}
L.~M. Li, N.~C. Grassly, and C.~Fraser.
\newblock Quantifying transmission heterogeneity using both pathogen phylogenies and incidence time series.
\newblock \emph{Molecular Biology and Evolution}, 34\penalty0 (11):\penalty0 2982--2995, 2017.

\bibitem[McKenzie and Steel(2000)]{McKenzie2000}
A.~McKenzie and M.~Steel.
\newblock Distributions of cherries for two models of trees.
\newblock \emph{Mathematical Biosciences}, 164\penalty0 (1):\penalty0 81--92, 2000.

\bibitem[Menardo et~al.(2021)Menardo, Gagneux, and Freund]{menardo2021multiple}
F.~Menardo, S.~Gagneux, and F.~Freund.
\newblock Multiple merger genealogies in outbreaks of \textit{Mycobacterium tuberculosis}.
\newblock \emph{Molecular Biology and Evolution}, 38\penalty0 (1):\penalty0 290--306, 2021.

\bibitem[Mo and Siepel(2023)]{mo2023domain}
Z.~Mo and A.~Siepel.
\newblock Domain-adaptive neural networks improve supervised machine learning based on simulated population genetic data.
\newblock \emph{PLoS Genetics}, 19\penalty0 (11):\penalty0 e1011032, 2023.

\bibitem[Murtagh(1984)]{Murtagh1984}
F.~Murtagh.
\newblock Counting dendrograms.
\newblock \emph{Discrete Applied Mathematics}, 7\penalty0 (2):\penalty0 191--199, 1984.

\bibitem[Nation(1998)]{nation1998notes}
J.~Nation.
\newblock Notes on lattice theory.
\newblock \emph{Cambridge studies in advanced mathematics}, 60, 1998.

\bibitem[Niwa et~al.(2016)Niwa, Nashida, Yanagimoto, and editor: W.~Stewart~Grant]{niwa2016reproductive}
H.-S. Niwa, K.~Nashida, T.~Yanagimoto, and H.~editor: W.~Stewart~Grant.
\newblock Reproductive skew in {Japanese} sardine inferred from {DNA} sequences.
\newblock \emph{ICES Journal of Marine Science}, 73\penalty0 (9):\penalty0 2181--2189, 2016.

\bibitem[Palacios et~al.(2019)Palacios, V{\'e}ber, Cappello, Wang, Wakeley, and Ramachandran]{palacios2019bayesian}
J.~A. Palacios, A.~V{\'e}ber, L.~Cappello, Z.~Wang, J.~Wakeley, and S.~Ramachandran.
\newblock Bayesian estimation of population size changes by sampling {Tajima’s} trees.
\newblock \emph{Genetics}, 213\penalty0 (3):\penalty0 967--986, 2019.

\bibitem[Palacios et~al.(2022)Palacios, Bhaskar, Disanto, and Rosenberg]{palacios2022enumeration}
J.~A. Palacios, A.~Bhaskar, F.~Disanto, and N.~A. Rosenberg.
\newblock Enumeration of binary trees compatible with a perfect phylogeny.
\newblock \emph{Journal of Mathematical Biology}, 84\penalty0 (6):\penalty0 54, 2022.

\bibitem[Pitman(1999)]{pitman1999coalescents}
J.~Pitman.
\newblock Coalescents with multiple collisions.
\newblock \emph{Annals of Probability}, pages 1870--1902, 1999.

\bibitem[Rasmussen et~al.(2014)Rasmussen, Hubisz, Gronau, and Siepel]{rasmussen2014genome}
M.~D. Rasmussen, M.~J. Hubisz, I.~Gronau, and A.~Siepel.
\newblock Genome-wide inference of ancestral recombination graphs.
\newblock \emph{PLoS genetics}, 10\penalty0 (5):\penalty0 e1004342, 2014.

\bibitem[Rosenberg et~al.(2025)Rosenberg, Stadler, and Steel]{rosenberg2025mathematical}
N.~A. Rosenberg, T.~Stadler, and M.~Steel.
\newblock " a mathematical theory of evolution": phylogenetic models dating back 100 years, 2025.

\bibitem[Sagitov(1999)]{sagitov1999general}
S.~Sagitov.
\newblock The general coalescent with asynchronous mergers of ancestral lines.
\newblock \emph{Journal of Applied Probability}, 36\penalty0 (4):\penalty0 1116--1125, 1999.

\bibitem[Sainudiin et~al.(2015)Sainudiin, Stadler, and Veber]{Sainudiin2015}
R.~Sainudiin, T.~Stadler, and A.~Veber.
\newblock Finding the best resolution for the {Kingman–Tajima} coalescent: theory and applications.
\newblock \emph{Journal of Mathematical Biology}, 70:\penalty0 1207--1247, 2015.

\bibitem[Samyak and Palacios(2024)]{samyak2024statistical}
R.~Samyak and J.~A. Palacios.
\newblock Statistical summaries of unlabelled evolutionary trees.
\newblock \emph{Biometrika}, 111\penalty0 (1):\penalty0 171--193, 2024.

\bibitem[Sargsyan and Wakeley(2008)]{sargsyan2008coalescent}
O.~Sargsyan and J.~Wakeley.
\newblock A coalescent process with simultaneous multiple mergers for approximating the gene genealogies of many marine organisms.
\newblock \emph{Theoretical Population Biology}, 74\penalty0 (1):\penalty0 104--114, 2008.

\bibitem[Schiffels and Wang(2020)]{schiffels2020msmc}
S.~Schiffels and K.~Wang.
\newblock {MSMC and MSMC2: the multiple sequentially Markovian coalescent}.
\newblock In \emph{Statistical population genomics}, pages 147--165. Humana, 2020.

\bibitem[Schr{\"o}der(1870)]{schroder1870vier}
E.~Schr{\"o}der.
\newblock Vier combinatorische probleme.
\newblock \emph{Z. Math. Phys}, 15:\penalty0 361--376, 1870.

\bibitem[Schweinsberg(2002)]{schweinsberg2002n2}
J.~Schweinsberg.
\newblock An {O(n2)} bound for the relaxation time of a {Markov} chain on cladograms.
\newblock \emph{Random Structures \& Algorithms}, 20\penalty0 (1):\penalty0 59--70, 2002.

\bibitem[Semple and Steel(2003)]{semple2003phylogenetics}
C.~Semple and M.~Steel.
\newblock \emph{Phylogenetics}, volume~24.
\newblock Oxford University Press, 2003.

\bibitem[Simper and Palacios(2022)]{simper2022adjacent}
M.~Simper and J.~A. Palacios.
\newblock An adjacent-swap {M}arkov chain on coalescent trees.
\newblock \emph{Journal of Applied Probability}, 59\penalty0 (4):\penalty0 1243--1260, 2022.

\bibitem[Sinclair(2012)]{sinclair2012algorithms}
A.~Sinclair.
\newblock \emph{Algorithms for random generation and counting: a {Markov} chain approach}.
\newblock Springer Science \& Business Media, 2012.

\bibitem[Sloane et~al.(2003)]{sloane2003line}
N.~J. Sloane et~al.
\newblock {The On-Line Encyclopedia of Integer Sequences}, 2003.

\bibitem[S{\o}rensen(2024)]{sorensen2024down}
F.~S{\o}rensen.
\newblock A down-up chain with persistent labels on multifurcating trees.
\newblock \emph{Random Structures \& Algorithms}, 64\penalty0 (2):\penalty0 354--400, 2024.

\bibitem[Spade et~al.(2014)Spade, Herbei, and Kubatko]{spade2014note}
D.~A. Spade, R.~Herbei, and L.~S. Kubatko.
\newblock A note on the relaxation time of two markov chains on rooted phylogenetic tree spaces.
\newblock \emph{Statistics \& Probability Letters}, 84:\penalty0 247--252, 2014.

\bibitem[Steel(2016)]{steel2016phylogeny}
M.~Steel.
\newblock \emph{Phylogeny: discrete and random processes in evolution}.
\newblock SIAM, 2016.

\bibitem[Wirtz(2023)]{wirtz2022enumeration}
J.~Wirtz.
\newblock On the enumeration of leaf-labelled increasing trees with arbitrary node-degree.
\newblock \emph{The Art of Discrete and Applied Mathematics}, 7, 2023.

\bibitem[Zhang et~al.(2021)Zhang, Dinh, and Matsen~IV]{zhang2021nonbifurcating}
C.~Zhang, V.~Dinh, and F.~A. Matsen~IV.
\newblock Nonbifurcating phylogenetic tree inference via the adaptive {LASSO}.
\newblock \emph{Journal of the American Statistical Association}, 116\penalty0 (534):\penalty0 858--873, 2021.

\bibitem[Zhang and Palacios(2025)]{zhang2025multiple}
J.~Zhang and J.~A. Palacios.
\newblock Multiple merger coalescent inference of effective population size.
\newblock \emph{Philosophical Transactions of the Royal Society B: Biological Sciences}, 380\penalty0 (1919), 2025.
\newblock \doi{10.1098/rstb.2023.0306}.

\end{thebibliography}

\newpage
\section{Appendix}

\subsection{Proofs of results in Section~\ref{sec:string}} \label{subsec:string-proofs}

\begin{proof}[Proof of Corollary~\ref{eq:k0_k1_pairs_cardinality}]
The base cases of $K=2$ and $K=3$ give $\mathcal{K}_2 = \{(1,1)\}, \mathcal{K}_2 = \{(1,2), (2,1)\}$ and $|\mathcal{K}_K| = 1=\lfloor (K-1)^2/4\rfloor + 1 $ holds for both. For $K\geq 4$, we split into the cases of $K=2c$ and $K=2c+1$. Consider the even case of $K=2c$ first. The set of valid $(K_0, K_1)$ pairs is $(1,2c-1), (2c-1, 0)$ and  \[ K_0(t) = k_0  \in \{2,...,2c-2\} \; \Rightarrow \; K_1(t) \in \Big \{ \max\{0, 2c-2k_0+1\}, ..., 2c-1-k_0 \Big \} \] 
We can split these cases into 
\begin{align*}
    & K_0(t) = k_0 \in \{ 2,..., c\} \; \Rightarrow \; K_1(t) \in \Big \{ 2c-2k_0+1, ..., 2c-1-k_0 \Big \} \\
    & K_0(t) = k_0 \in \{c +1,..., 2c-2 \} \; \Rightarrow \; K_1(t) \in \Big \{ 0, ..., 2c-1-k_0 \Big \} 
\end{align*}
This implies 
\begin{align*}
    |\mathcal{K}_K| &= 2 + \sum_{k_0=2}^c \Big( 2c-1-k_0 - (2c-2k_0+1) +1 \Big ) + \sum_{k_0=c+1}^{2c-2} (2c-1-k_0+1) \\
    &= 2 + \sum_{k_0=2}^c (k_0-1) + \sum_{k_0=c+1}^{2c-2} (2c-k_0) \\
    &= 2 + \frac{c(c-1)}{2} + \frac{(c+1)(c-2)}{2} \\
    &= c^2-c+1 
\end{align*} 

\noindent Next, let's consider the odd case of $K=2c+1$. The set of valid $(K_0, K_1)$ pairs is $(1,2c), (2c, 0)$ and  \[ K_0(t) = k_0  \in \{2,...,2c-1\} \; \Rightarrow \; K_1(t) \in \Big \{ \max\{0, 2c-2k_0+2\}, ..., 2c-k_0 \Big \}  \] 
We can split these cases into 
\begin{align*}
    & K_0(t) = k_0 \in \{ 2,..., c+1\} \; \Rightarrow \; K_1(t) \in \Big \{ 2c-2k_0+2, ..., 2c-k_0 \Big \} \\
    & K_0(t) = k_0 \in \{c +2,..., 2c-1 \} \; \Rightarrow \; K_1(t) \in \Big \{ 0, ..., 2c-k_0 \Big \} 
\end{align*}
This implies 
\begin{align*}
    |\mathcal{K}_K| &= 2 + \sum_{k_0=2}^{c+1} \Big( 2c-k_0 - (2c-2k_0+2) +1 \Big ) + \sum_{k_0=c+2}^{2c-1} (2c-k_0+1) \\
    &= 2 + \sum_{k_0=2}^{c+1} (k_0-1) + \sum_{k_0=c+2}^{2c-2} (2c-k_0+1) \\
    &= 2 + \frac{c(c+1)}{2} + \frac{(c+2)(c-3)}{2} \\
    &= c^2+1 
\end{align*} 
For both $K$ even and $K$ odd, we get $|\mathcal{K}_K| = \lfloor (K-1)^2/4\rfloor + 1$. 
\end{proof}

\begin{proof}[Proof of Corollary~\ref{eq:k0_cardinality}]
The Eulerian numbers (OEIS Sequence A008292) are given by the recursion $E(n,k)=(n-k+1) E(n-1,k-1)+k E(n-1,k)$ for $1\leq k \leq n$ with base case $E(n,1)=E(n,n) = 1$ for all $n$. Let $B_K(k_0)$ be the number of length $K$ $t$ vectors with $K_0(t)=k_0$. Similar to the proof of Proposition~\ref{prop:string-counting-recursion}, we again consider a length $K-1$ $t$ vector and construct $\tilde{t} = (t, t_K)$ of length $K$, where $t_K \in \{1,2,...,K-1\}$. Let $K_0(\tilde{t}) = k_0$. There are two cases: 
\begin{itemize}
    \item If $t_K$ is equal to an element that does appear in $t$, then $K_0(t) = k_0 -1$ because the new element $K$ does not appear in $\tilde{t}$ while nothing else changes. There are $K-1-(k_0-1)=K-k_0$ choices for $t_K$ and $B_{K-1}(k_0-1)$ such length $K-1$ $t$ vectors. 
    \item If $t_K$ is equal to an element that does not appear in $t$, then we still have $K_0(t)= k_0$ because the new element $K$ does not appear but another element now does appear. There are $k_0$ choices for $t_K$ and $B_{K-1}(k_0)$ such length $K-1$ $t$ vectors.
\end{itemize}
Putting everything together, the recursion we get is \[ B_K(k_0) = (K-k_0) B_{K-1}(k_0-1) + k_0 B_{K-1}(k_0)\] The base cases are $B_K(1)= B_K(K-1)= 1$ by the proof of Proposition~\ref{prop:valid-k0-k1-pairs}. This recursion is equivalent to the recursion for the Eulerian numbers with $B_K(k_0)= E(K-1, k_0)$ which concludes the proof.  
\end{proof}

\subsection{Polynomial expressions of $G(N,K)$ for $K\leq 8$} \label{subsec:polynomial_GNK}

We used Theorem~\ref{thm:enumeration-result-final} to find explicit polynomial expressions of $G(N,K)$ up to $K=8$. 
\begin{align*}
    G(N,2) &= N-2 \\
    G(N,3) &= N^2 - 5N +6\\
    & = (N-2)(N-3) \\ 
    G(N, 4) &= N^3-\frac{19}{2}N^2+\frac{59}{2}N-30  \\
    & = (N-3)(N-4) \;\; \frac{2N-5}{2} \\
    G(N, 5) &= N^4 -\frac{31}{2} N^3+\frac{179}{2}N^2-229N+220 \\
    & = ( N-4)(N-5) \;\; \frac{2N^2-13N+22}{2} \\ 
    G(N, 6) &= N^5 - 23N^4 + \frac{634}{3}N^3 + \frac{1945}{2}N^2 + \frac{13489}{6}N - 2095 \\
    & = (N-5)(N-6) \;\; \frac{6 N^3 - 72 N^2 + 296 N - 419 }{6}  \\ 
    G(N,7) &= N^6 - 32 N^5 + \frac{3417}{8}N^4 -  \frac{36601}{12} N^3 + \frac{98519}{8} N^2 - \frac{320483}{12}N + 24346\\
    &= (N-6)(N-7)\;\; \frac{24 N^4 - 456 N^3 + 3315 N^2 - 10955 N + 13912 }{24} \\
    G(N,8) &= N^7 -\frac{85N^6}{2} + \frac{93091}{120} N^5 - \frac{94739}{12} N^4 + \frac{116277}{24} N^3 - \frac{2154193}{12} N^2 + \frac{3722867}{10} N - 333676 \\
    &=  (N-7)(N-8) \;\; \frac{120 N^5 - 3300 N^4 + 36871 N^3 - 209525 N^2 + 606234 N - 715020}{120} 
\end{align*}
We hypothesize the same pattern to hold for larger $K$: 
\begin{align*}
    G(N,K) &= N^{K-1} - \frac{a(K)}{2} N^{K-2} + \Theta(N^{K-3}) \\
    &=  (N-K)(N-K+1) \times \left (N^{K-3} - \frac{b(K-2)}{2} N^{K-4} + \Theta(N^{K-5}) \right ) 
\end{align*}
where $a(K)=  3K(K-1)/2 + 1$ are the centered triangular numbers (OEIS sequence A005448) and $b(K) = K(3K+7)/2 = b(K-1)+3K+2$ (OEIS sequence A140090).  

\subsection{Growth rate of $G(N)$} 
\label{subsec:cardinality_growth_rate}

In Section~\ref{sec:string}, we found that the order of $G(N)$ was $\log G(N) \sim O(N \log N)$. We use our enumeration formula to confirm the growth rate of $G(N)$ by calculating $G(N)$ up to $N=20$ and then fitting three different models to $\log G(N)$: linear, quadratic, and log-linear. We find that the log-linear model performs the best, as expected (see Figure~\ref{fig:cardinality_growth}). Coefficient estimates and residuals for the three models are shown in Table~\ref{table:cardinality_growth_regression}, where we see the RMSE is the lowest for the log-linear model and we can conclude that the order of $\log G(N)$ is greater than $O(N)$. 

\begin{table}[H]
\begin{center}
\begin{tabular}{l c c c}
\hline
& Model 1 (Linear) & Model 2 (Quadratic) & Model 3 (Log-linear) \\
\hline
(Intercept) & $-7.60 \; (0.69)^{***}$ & $-2.64 \; (0.21)^{***}$ & $1.57 \; (0.04)^{***}$  \\
$N$           & $1.94 \; (0.05)^{***}$  & $0.85 \; (0.04)^{***}$  & $-1.35 \; (0.01)^{***}$ \\
$N^2$         &                         & $0.05 \; (0.00)^{***}$  &                         \\
$N \log N$       &                         &                         & $0.97 \; (0.00)^{***}$  \\
\hline
R$^2$       & $0.99$                  & $1.00$                  & $1.00$                  \\
Adj. R$^2$  & $0.99$                  & $1.00$                  & $1.00$                  \\
Num. obs.   & $18$                    & $18$                    & $18$                    \\
F statistic & $1242.46$               & $29805.25$              & $2466110.84$            \\
RMSE        & $1.21$                  & $0.18$                  & $\mathbf{0.02}$                  \\
\hline
\multicolumn{4}{l}{\scriptsize{$^{***}p<0.001$; $^{**}p<0.01$; $^{*}p<0.05$}}
\end{tabular}
\caption{Linear, quadratic, and log-linear model results for the growth of $\log G(N)$. The log-linear model has the smallest residual standard error, while the linear model is clearly not a good fit, supporting our conclusion that $\log G(N) \sim O(N\log N)$.}
\label{table:cardinality_growth_regression}
\end{center}
\end{table}

\begin{figure}[H]
    \centering
    \includegraphics[width=0.8\linewidth]{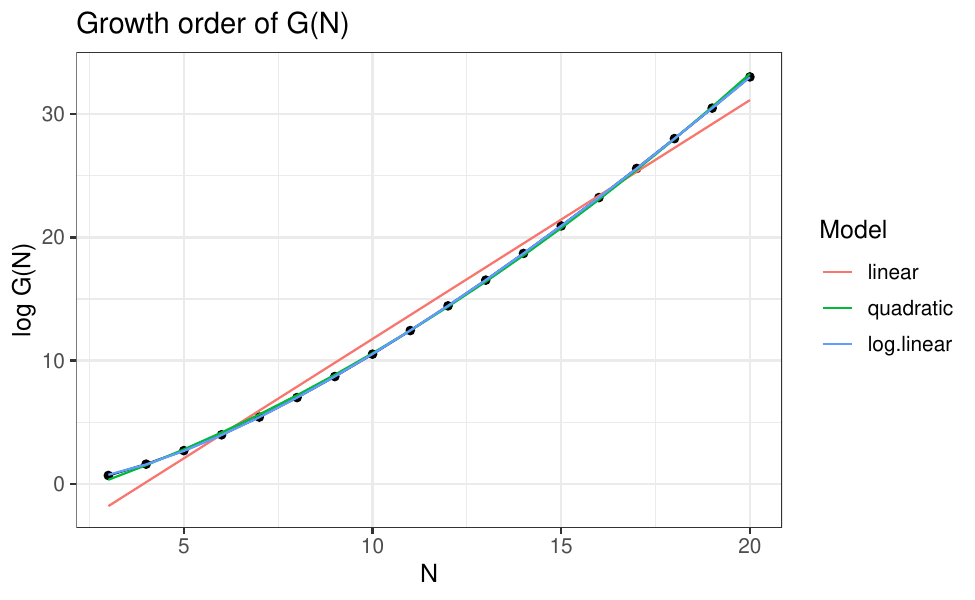}
    \caption{\textbf{Plot of $N$ versus $\log G(N)$:} The black points are the true counts, while the colored curves are the fitted linear, quadratic, and log-linear models. The log-linear model is the best fit.}
    \label{fig:cardinality_growth}
\end{figure}

\subsection{Proofs of results in Section~\ref{sec:lattice}} \label{subsec:lattice-proofs}

\begin{proof}[Proof of Theorem~\ref{thm:fmat_multi}] 
We follow a proof similar to that of Theorem 1 in \cite{samyak2024statistical}, which proves this result for F-matrices of ranked binary tree shapes. First, we label the internal nodes by 1 to $K$, where the root is 1 and the next internal node is 2, etc. Let $u_i$ be the time at which node $i$ furcates. \\

\noindent We define an intermediate matrix $D$ that is also a $K\times K$ lower triangular matrix with nonnegative integer entries. The entry $D_{i,j}$ denotes the number of direct descendants of node $j$ that have not furcated until time $u_i$. Then $D$ must satisfy the following conditions. 
\begin{enumerate}[label={D\arabic*.}]
    \item $\sum_{j=1}^K D_{K, j} = N$: There are a total of $N$ tips.
    \item $D_{i,i}\geq 2$ for $i=1,...,K$: There are $K$ furcations. 
    \item For all $j\geq i$, $D_{i-1, j} -1\leq D_{i, j} \leq D_{i-1, j}$: At most one branch can furcate at a time, so $D_{i-1,j}, D_{i,j}$ differ by at most 1. 
    \item Exactly 1 branch furcates at time $u_i, i=1,...,K$, that is, for all $i=1,...,K$, \[ \sum_{j=1}^{i-1} \mathds{1}( D_{i-1, j} - D_{i,j} ) = 1 \] 
\end{enumerate}
This matrix $K\times K$ matrix $D$ is uniquely determined for any ranked tree shape $T_{N,K}$. We reconstruct $T_{N,K}$ given a $K\times K$ matrix $D$ by starting at the root, and successively furcating branches into $D_{i,i}$ branches according to the column of $D$ where $D_{i-1, j} - D_{i,j} = 1$. The only difference between this and the binary case is that $D_{i,i}=2$ for all $i$, meaning branches only bifurcate. Hence the space $\mathcal{T}_{N,K}$ is in bijection with $\mathcal{D}_{N,K}$, the space of D-matrices that obey the above conditions. It remains to show $\mathcal{D}_{N,K}$ is in bijection with $\mathcal{F}_{N,K}$. Figure~\ref{fig:D-F-matrix-ex} shows a tree with 8 tips and 5 internal nodes and its corresponding D and F matrices. \\

\begin{figure}[h]
    \centering
    \includegraphics[width=0.55\linewidth]{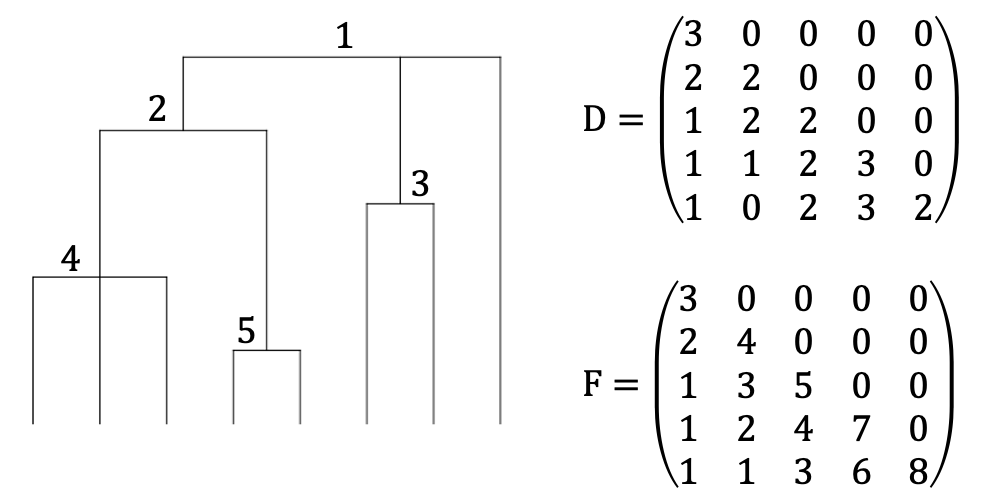}
    \caption{A ranked multifurcating tree shape with $N=8, K=5$ and its corresponding D and F matrices. The ranking of the 5 internal nodes from the root downwards are marked.}
    \label{fig:D-F-matrix-ex}
\end{figure}

\noindent Define the function $\phi: \mathcal{D}_{N,K} \to \mathcal{F}_{N,K} $ by $F_{i,j} = \sum_{n=1}^j D_{i,n}$. It is clear $\phi$ is injective. Define the inverse mapping $\phi^{-1}: \mathcal{F}_{N,K} \to \mathcal{D}_{N,K} $ by $D_{i,j} = F_{i,j} - F_{i,j-1}$ with $F_{i,0}=0$. To show bijection, we need to show that $\phi(\mathcal{D}_{N,K}) \subset \mathcal{F}_{N,K}$ and $\phi^{-1}(\mathcal{F}_{N,K}) \subset \mathcal{D}_{N,K}$. Notice that by definition of $D$, we must have \[ \sum_{n=1}^{i-1} D_{i,n} = \left (\sum_{n=1}^{i-1} D_{i-1, n} \right ) -1 \]

\noindent First consider the range of $\phi$. We have $F_{1,1} =D_{1,1} > 1$ and \[ F_{i,i} = \sum_{n=1}^i D_{i,n} = D_{i,i} + \sum_{n=1}^{i-1} D_{i,n} =D_{i,i} + \sum_{n=1}^{i-1} D_{i-1,n} -1 = F_{i-1, i-1} + D_{i,i}-1 > F_{i-1,i-1} \] In addition, \[ F_{i+1, i} = \sum_{n=1}^i D_{i+1, n} = \sum_{n=1}^i D_{i,n}-1 = F_{i,i}-1 \] and $F_{K,K} = \sum_{n=1}^K D_{K, n} = N$. This proves $\phi(D)$ satisfies F1. Next, using condition D2 with $j=1$ shows condition F2 because $F_{i,1}= D_{i,1}$ for all $i$. Lastly, F3a follows from $D \geq 0$ and monotonically decreasing columns of $D$ and F3b follows from D3. F3c is equivalent to 
\begin{align*}
    0\leq (F_{i-1, j} - F_{i,j}) - (F_{i-1, j-1}- F_{i,j-1}) \leq 1 &\Leftrightarrow 0\leq  \sum_{n=1}^j D_{i-1,n}- \sum_{n=1}^j D_{i-1,n}- \sum_{n=1}^{j-1} D_{i-1,n} +  \sum_{n=1}^{j-1} D_{i,n} \leq 1 \\
    &\Leftrightarrow 0\leq D_{i-1, j} - D_{i,j} \leq 1
\end{align*}
by D2 and D3. Hence $\phi(\mathcal{D}_{N,K}) \subset \mathcal{F}_{N,K}$. \\

\noindent To show the reverse condition, D1 follows from F1 because \[ D_{i,i} = F_{i,i} - F_{i,i-1} = F_{i,i}-F_{i-1,i-1} + 1 > 1 \] The arguments above show F3b implies D3 and F3c implies D2. This completes the proof: We have shown the bijection between $\mathcal{T}_{N_K}$ and $\mathcal{F}_{N,K}$.
\end{proof}

\begin{proof}[Proof of Lemma~\ref{lemma:fmat_remove}]
We first prove that removing the first row and first column of an F-matrix yields a valid F-matrix. Denote the matrix without row and column 1 by $F^{(-1, -1)}$ which has dimension $(K-1) \times (K-1)$. Then $F^{(-1,-1)}_{i,j} = F_{i+1, j+1}$ for all $1\leq j <i \leq K-1$. It is straightforward to check $F^{(-1,-1)}$ satisfies conditions F1 and F3. Applying F3b on the original F-matrix implies F2 holds for the new F-matrix. Therefore, $F^{(-1,-1)}$ is a valid F-matrix. \\ 

\noindent New, we consider removing column and row $e\geq 2$. We introduce notation for $F^{(-e, -e)}\equiv F'$ (for sake of notation), a $(K-1) \times (K-1)$ matrix. This matrix is still lower triangular. Then
\begin{itemize}
    \item $F'_{m,n} = F_{m,n}$ for $1 \leq n \leq m \leq e-1$.
    \item $F'_{m,n} = F_{m+1, n}$ for $e \leq m \leq K-1, 1 \leq n \leq e-1$. In particular, $F'_{e, n} = F_{e+1, n} = F_{e, n}$ for $1 \leq n \leq e-1$. 
    \item $F'_{m,n} = F_{m+1, n+1}$ for $e \leq n <m \leq K-1$. 
\end{itemize}
We check that $F'$ satisfies conditions F1-F3. For F1, all the diagonal elements will still be non-decreasing and $F'_{K-1, K-1}=N$. We check the subdiagonal condition: 
\begin{align*}
    m< e-1: \qquad & F'_{m,m} = F_{m,m} \stackrel{*}{=} F_{m+1, m} + 1 = F'_{m+1,m}\\
    m = e-1: \qquad & F'_{e-1, e-1} = F'_{e-1, e-1} \stackrel{*}{=} F_{e, e-1} - 1 \stackrel{**}{=}F_{e+1, e-1} -1 = F'_{e, e-1}-1    \\
    m> e-1: \qquad & F'_{m,m} = F_{m+1,m+1} \stackrel{*}{=} F_{m+2, m+1} + 1 = F'_{m+1,m} 
\end{align*}
The equalities marked by $*$ hold by F2 on the original F-matrix and the equality marked by $**$ holds by Equation~\ref{eq:f-edge-condition}. F2 holds because the only change to the first column is $F'_{e, 1} = F_{e+1, 1} = F_{e, 1}$. For F3a and F3b, the conditions all hold because we are just shifting indices at most one over. For F3c, the condition becomes showing $0\leq (F'_{m-1, n} - F'_{m,n}) - (F'_{m-1, n-1}- F'_{m,n-1}) \leq 1$ for all $m=4,...,K-1, n=2,...,m-2$. By Equation~\ref{eq:f-edge-condition}, we will have 
\begin{align*}
     &(F'_{m-1, n} - F'_{m,n}) - (F'_{m-1, n-1}- F'_{m,n-1})\\
     \qquad \qquad &= \begin{cases}
    (F_{m-1, n} - F_{m,n}) - (F_{m-1, n-1}- F_{m,n-1}) & \text{ if } 2\leq n < m-2, m \leq e \\
    (F_{m, n+1} - F_{m+1,n+1}) - (F_{m, n}- F_{m+1,n}) & \text{ if } e\leq n < m-2, m \leq K-1 \\
    (F_{m, n} - F_{m+1,n}) - (F_{m, n-1}- F_{m+1,n-1}) & \text{ if } 2 \leq n\leq e, 4 \leq m \leq K-1 
    \end{cases} 
\end{align*}
By F3c on the original F-matrix, F3c will hold for $F^{(-e, -e)}$. Therefore, removing row and column $e$ of an F-matrix corresponding to a tree with edge $(e, e+1)$ is still a valid F-matrix. 
\end{proof}

\begin{proof}[Proof of Theorem~\ref{thm:tree-deg}]
The forwards degree comes directly from the condition in Equation~\ref{eq:f-edge-condition}. \\ 

\noindent To find the backwards degree, let $U(k_i, l_i)$ denote the number of ways a single collapse edge operation can result in an internal node with ranking $i$ that subtends $l_i$ leaves and $k_i$ internal nodes. This is equal to choosing $s \in \{2, 3,..., k_i+l_i-1\}$ elements out of the $k_i+l_i$ total elements to subtend the new internal node where the split is formed. We can think of this problem as splitting $l_i$ unlabeled objects and $k_i$ labeled objects into two distinctly labeled bins (call them Bin 1 and Bin 2), where Bin 1 must have at least 2 elements and no bin can be empty. For a given split size $s$, we then choose $j \in \{0,..., \min(s, l_i)\}$ of them from the unlabeled objects and the remaining $s-j$ objects from the labeled objects, provided there are enough labeled objects. If there are no labeled objects, then the number of possible splits is $l_i-2$, because the original node $i+1$ has to subtend at least two leaves. Hence, we have 
\begin{equation} \label{eq:backwards_degree}
    \text{deg}^-(T_{N,K}) = \sum_{i=1}^K U(k_i, l_i) = \sum_{i=1}^K \left [ \sum_{s=2}^{k_i+l_i-1}  \sum_{j=0}^s  \binom{k_i}{s-j} \mathds{1}(j \leq l_i ) \right ]  + (l_i-2) \mathds{1}(k_i=0) . 
\end{equation}
where by convention, $\binom{0}{a}=0, \binom{a}{b}=0$ for $a < b$ or $b <0$. Note that if at internal node $i$ there is only a bifurcation, then $l_i+ k_i=2$ and $U(k_i, l_i)=0$ as expected. \\

\noindent Now, it remains to show for any $l\geq 0, k \geq 0$, $U(k,l) = (l+1)2^k -k-3 + \mathds{1}(l=0)$. We can separate the claim into three cases: 
\begin{align*}
    l>0, k>0: \qquad U(k,l) & = \sum_{s=2}^{k+l-1} \sum_{j=0}^s  \binom{k}{s-j} \mathds{1}(j \leq l ) = (l+1) 2^k -k-3 \\ 
    l=0, k\geq 2: \qquad U(k,0) &= \sum_{s=2}^{k-1} \binom{k}{s} = 2^k-k-2 \\
    l\geq 2, k=0: \qquad U(0,l) &= l-2
\end{align*}
The cases where $l=0, k\geq 2$ and $k=0, l\geq 2$ are clear by definition, so we focus on the case $l, k>0$. We will prove by induction on $l$. First, we check the base cases of $l=1$ and $l=2$. 
\begin{align*}
    U(k, 1) &= \sum_{s=2}^{k} \sum_{j=0}^s \binom{k}{s-j} \mathds{1}(j \leq 1) \\
    &= \sum_{s=2}^{k} \binom{k}{s} + \binom{k}{s-1} \\
    &= 2^{k+1}-k-3
\end{align*}
For $l=2$, we have 
\begin{align*}
    U(k, 2) &= \sum_{s=2}^{k+1} \sum_{j=0}^s \binom{k}{s-j} \mathds{1}(j \leq 2) \\
    &= \sum_{s=2}^{k+1} \binom{k}{s} + \binom{k}{s-1} + \binom{k}{s-2} \\ 
    &= 2^{k+1} + 2^{k} -k-3 \\
    &= 3 \times 2^k -k-3
\end{align*}
This shows the base cases hold. For the induction step, assume the result holds for $U(k,l)$. We will show $U(k,l+1) - U(k,l) = 2^k$ for any fixed $k$. 
\begin{align*}
    U(k,l+1) - U(k,l) &= \sum_{s=2}^{k+l} \sum_{j=0}^s  \binom{k}{s-j} \mathds{1}(j \leq l+1 ) - \left [ \sum_{s=2}^{k+l-1} \sum_{j=0}^s  \binom{k}{s-j} \mathds{1}(j \leq l )\right ] \\
    &= \left [ \sum_{s=2}^{k+l-1} \sum_{j=0}^s  \binom{k}{s-j} \mathds{1}(j = l+1 )\right ] + \left [\sum_{j=0}^{k+l}  \binom{k}{k+l-j} \mathds{1}(j \leq l+1) \right ] \\
    &=\left [ \sum_{s=2}^{k+l-1} \binom{k}{s-l-1} \right ] +\left [ \sum_{j=0}^{l+1}  \binom{k}{k+l-j} \right ]\\
    &= \left [\sum_{s=l+1}^{k+l-1} \binom{k}{s-l-1}\right ] + \left [\sum_{j=l}^{l+1}  \binom{k}{k+l-j} \right ]\\
    &= \left [ \sum_{s=0}^{k-2} \binom{k}{s}\right ] + \left [\binom{k}{k} + \binom{k}{k+1} \right ] \\
    &= 2^k
\end{align*}
Therefore, if $U(k,l) = (l+1)2^k -k-3$ for a fixed $l$, then we have shown $U(k,l+1) = (l+2) 2^k-k-3$ completing the inductive proof. Note also that in the bifurcation cases of $U(1,1), U(2,0), U(0,2)$, all of them are equal to 0, so that case is subsumed among the more general formulation. Combining everything in Equation~\ref{eq:backwards_degree} gives \[ \text{deg}^-(T_{N,K}) = \sum_{i=1}^K \Big ((l_i+1)2^{k_i} -k_i-3 + \mathds{1}(l_i=0) \Big ). \qedhere\]
\end{proof}

\begin{lemma} \label{lemma:Ukl_comparisons}
The maximum of $U(k,l)$ for $k+l=C$ fixed is obtained at $l=0, k=C$. The maximum of $U(k,l)$ for $2k+l=C$ fixed is obtained at $l=2$ and $l=3$ for $C$ even and odd respectively. 
\end{lemma}
\begin{proof}[Proof of Lemma~\ref{lemma:Ukl_comparisons}]
\textbf{$k+l=C$ fixed:} we compare $U(k,l)$ to $U(k-1, l+1)$: 
\begin{align*}
    U(k,l) - U(k-1, l+1) &= (l+1) 2^{k} -k-3 +\mathds{1}(l =0) - (l+2) 2^{k-1} + (k-1)+3  \\
    &= l \times 2^{k-1} -1 + \mathds{1}(l =0) 
\end{align*}
We see $U(k,l) - U(k-1, l+1) >0$ for $l \geq 1$ and they are equal when $l=0$. This means the maximum of $U(k,l)$ for $k+l=C$ fixed is obtained when $l=0$. \\

\noindent \textbf{$2k+l=C$ fixed:} we compare $U(k,l)$ to $U(k+1, l-2)$: 
\begin{align*}
    U(k,l) - U(k+1, l-2) &= (l+1) 2^k -k-3 - (l-1) 2^{k+1} +k+4 -\mathds{1}(l-2 =0) \\
    &= 3 \times 2^k - l \times 2^k +1 -\mathds{1}(l-2 =0)
\end{align*}
Therefore, for $l=2,3$, we have $U(k,l) > U(k+1, l-2)$ while for $l>3$, $U(k,l) < U(k+1, l-2)$. Hence, for $2k+l=C$ fixed, the maximum of $U(k,l)$ is obtained when $l=2$ for $C$ even and $l=3$ for $C$ odd. 
\end{proof}

\begin{proof}[Proof of Theorem~\ref{thm:max-deg-tree}]
To solve the optimization problem (\ref{eq:M_opt_problem}), we will first find the tree with maximum backwards degree with $K$ internal nodes. We then find the value of $K$ for which the maximum backwards degree occurs. We separate into two cases: $K \leq \lfloor N/2 \rfloor +1$ and $K> \lfloor N/2 \rfloor +1$. The first case of $K \leq \lfloor N/2 \rfloor +1$ is when all non-root internal nodes could descend directly from the root, while the second case of $K> \lfloor N/2 \rfloor +1$ does not allow for this possibility. \\

\noindent For a fixed $K> \lfloor N/2 \rfloor +1$, the largest multifurcating event possible in $T_{N,K}$ is of size $N-K+1$, where one internal node will have $N-K+1$ descendants, while the other internal nodes will have exactly two descendants. The backwards degree is then equal to $U(k,l)$ for some $k,l$ with $k+l=N-K+1$, and the maximum occurs at $l=0$ by Lemma~\ref{lemma:Ukl_comparisons}. Therefore, for fixed $K> \lfloor N/2 \rfloor +1$, the maximum backwards degree is equal to $U(N-K+1, 0) = 2^{N-K+1}-N+K-3$. \\ 

\noindent For $K\leq \lfloor N/2 \rfloor +1$, the vertices of the parameter space of the $k_i$'s are $k_i=0$ for all but one $i\in \{1,...,K\}$. Since we are solving an integer programming problem, the solution occurs at vertex of the parameter space. Clearly, we want as large of a $k_i$ value as possible, and the constraint $k_i \leq K-i$ means that the optimal solution occurs when $k_2=k_3=...=k_K=0$ and $k_1$ is as large as possible. For $K \leq \lfloor N/2 \rfloor +1$, the maximum value of $k_1 = K-1$ can be obtained, where the optimal solution is \[  k_1=K-1, k_2=k_3=...=k_K=0, \;\; l_1= N-2K+2, l_2=l_3=...=l_K=2\] and the maximum backwards degree is equal to $U(K-1, N-2K+1)$. \\

\noindent Finally, we must find take the maximum over $K$: \[ \max \big \{ U(K-1, N-2K+1): K=1,..., \lfloor N/2 \rfloor +1 \big\} \bigcup \big\{U(N-K+1, 0): K=  \lfloor N/2 \rfloor +2 ,..., N\big \}. \] For $K\leq \lfloor N/2 \rfloor +1$, we use Lemma~\ref{lemma:Ukl_comparisons} to find the maximum is obtained when $N-2K+2$ is equal to 2 or 3, depending on whether $N$ is even or odd. If $N$ is even, then $k_1= N/2$ and $l_1 =2$ is the optimal solution with maximum of $2^{N/2} + 2^{(N/2)-1} -(N/2)-2$. If $N$ is odd, then $k_1 = (N-1)/2$ and $l_1=3$ is the optimal solution with maximum of $2^{(N+1)/2} - (N+1)/2 -2$. Both these values are greater than $\max_{K> \lfloor N/2 \rfloor +1} 2^{N-K+1}-N+K-3$. Hence, the maximum backwards degree tree is characterized by 
\begin{align*}
    & k_1=\left \lfloor \frac{N}{2}\right \rfloor, \;\; \; k_2=k_3=...=k_K=0, \\
    &l_1 = 2 + N \text{ mod } 2, \;\; \; l_2=l_3=...=l_K=2 . 
\end{align*}
This is precisely the tree characterized by conditions 1-3 and the backwards degree of this tree is equal to $\big (3+ (N \text{ mod } 2) \big) 2^{\left \lfloor \frac{N}{2} \right \rfloor} - \left \lfloor \frac{N}{2} \right \rfloor - 3$ using Theorem~\ref{thm:tree-deg}. 
\end{proof}

\subsection{Examples of algorithms to find the least upper bound} \label{subsec:lub-algs}

We present an example of Algorithm~\ref{alg:tree_binary_lub} on two binary trees with $N=8$. 
\begin{enumerate}
    \item We start with \[ F(T_X) = \begin{pmatrix}
  2 & 0 & 0 & 0 & 0 & 0 & 0 \\ 
  1 & 3 & 0 & 0 & 0 & 0 & 0 \\ 
  0 & 2 & 4 & 0 & 0 & 0 & 0 \\ 
  0 & 2 & 3 & 5 & 0 & 0 & 0 \\ 
  0 & 1 & 2 & 4 & 6 & 0 & 0 \\ 
  0 & 1 & 2 & 4 & 5 & 7 & 0 \\ 
  0 & 1 & 2 & 3 & 4 & 6 & 8 \\ 
   \end{pmatrix}, \qquad F(T_Y) = \begin{pmatrix}
  2 & 0 & 0 & 0 & 0 & 0 & 0 \\ 
  1 & 3 & 0 & 0 & 0 & 0 & 0 \\ 
  0 & 2 & 4 & 0 & 0 & 0 & 0 \\ 
  0 & 2 & 3 & 5 & 0 & 0 & 0 \\ 
  0 & 1 & 2 & 4 & 6 & 0 & 0 \\ 
  0 & 1 & 2 & 4 & 5 & 7 & 0 \\ 
  0 & 1 & 2 & 4 & 5 & 6 & 8 \\ 
   \end{pmatrix}. \] The largest shared submatrix is \[ S_F(T_X, T_Y) = \begin{pmatrix}
  2 & 0 & 0 & 0 & 0 \\ 
  1 & 3 & 0 & 0 & 0 \\ 
  0 & 2 & 4 & 0 & 0 \\ 
  0 & 1 & 2 & 7 & 0 \\ 
  0 & 1 & 2 & 6 & 8 \\ 
   \end{pmatrix}.\]
   \item Column 3 violates condition F1, so our new matrix is $S_F(T_X, T_Y) \gets \begin{pmatrix}
  2 & 0 & 0 & 0 \\ 
  1 & 3 & 0 & 0 \\ 
  0 & 1 & 7 & 0 \\ 
  0 & 1 & 6 & 8 \\ 
   \end{pmatrix}$.
   \item Column 2 now violates condition F1, so our new matrix is $S_F(T_X, T_Y) \gets \begin{pmatrix}
  2 & 0 & 0 \\ 
  0 & 7 & 0 \\ 
  0 & 6 & 8 \\ 
   \end{pmatrix}$.
   \item Column 1 now violates condition F1, so our new matrix is $S_F(T_X, T_Y) \gets \begin{pmatrix}
  7 & 0 \\ 
  6 & 8 \\ 
   \end{pmatrix}$.
   \item No columns violate the conditions, so the final matrix we return is $S_F(T_X, T_Y)= \begin{pmatrix}
  7 & 0 \\ 
  6 & 8 \\ 
   \end{pmatrix}$. This is the unique LUB of our original $T_X, T_Y$. 
\end{enumerate}

Next, let us consider two multifurcating trees with 7 tips in Figure~\ref{fig:lub-ex} and use Algorithm~\ref{alg:tree_lub}. We also show the LUB via its graphical path on the lattice in Figure~\ref{fig:lub-on-lattice}.

\begin{figure}[H]
    \centering
    \includegraphics[width=0.8\linewidth]{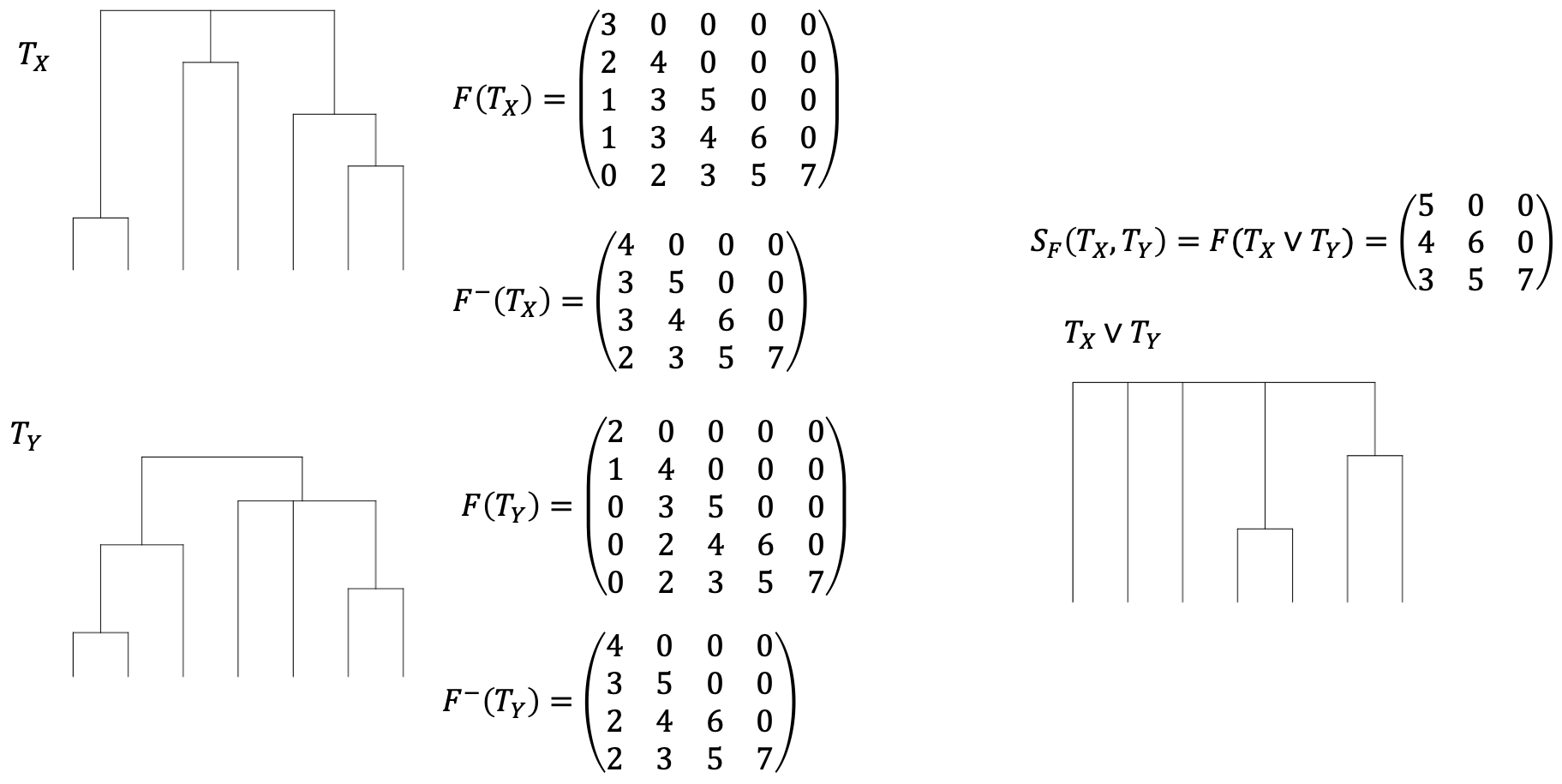}
    \caption{An illustration of how to find the least upper bound of two trees with 7 tips using Algorithm~\ref{alg:tree_lub}. Here, $S_F(T_X, T_Y)$ is an F-matrix, so the algorithm did not need to enter Step 4.} 
    \label{fig:lub-ex}
\end{figure}

\begin{figure}[H]
    \centering
    \includegraphics[width=0.75\linewidth]{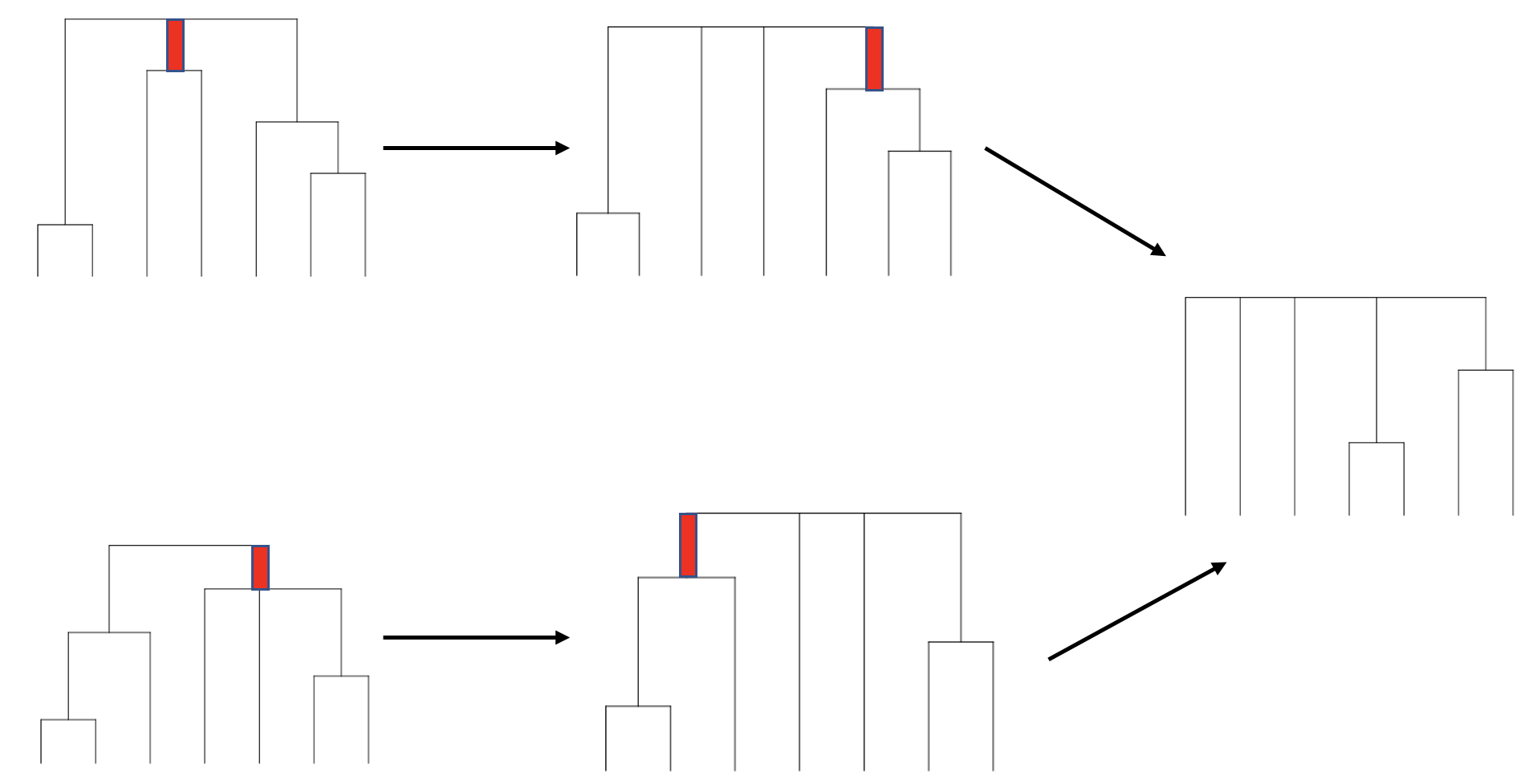}
    \caption{An illustration of the lattice path to find the least upper bound of two trees with 7 tips from Figure~\ref{fig:lub-ex} The red edge means it was collapsed to create the new tree.}
    \label{fig:lub-on-lattice}
\end{figure}


\subsection{Semi-random sampling of a ranked tree shape} \label{sec:semi-random}

\begin{algorithm}[H] 
\caption{Semi-random sampling of a tree with $N$ tips and $K$ internal nodes}
\label{alg:semi_random_sampling}
\vspace{0.1in}
\KwData{}
$N$: the number of tips. \\
$K$: the number of internal nodes.
    
\vspace{0.1in}

\KwResult{}
$F$: a $K \times K$ F-matrix with the $(K,K)$th entry equal to $N$ that corresponds to a tree $T_{N,K}$. 

\vspace{0.1in}

Sample $F_1,..., F_{K-1} \sim \text{Unif}\{2,3,...,N-1\}$ without replacement and rank them in order with $F_1$ the smallest and $F_{K-1}$ the largest. \\
Set $\text{diag}(F)= (F_1,..., F_{K-1}, N)$. \\
\For{$j=1, ...,N-1$}{
\For{$i=j+1,...,N$} {
  Set $F_{i,j} = \max\{0, F_{i-1, j}- 1\}$. 
  }
}
Return $F$. \\
\vspace{0.1in}
\end{algorithm}

\vspace{0.1in}
For our simulation, we generates one semi-random tree with $K$ internal nodes and $N$ tips per $K=1,2,\ldots,N-1$. 

\subsection{String representation and edge collapsing} \label{sec:string-collapse}

We defined the collapse edge operation in Section~\ref{sec:lattice} and related it to the F-matrix row and column deletion. The same operation can be carried out using the string representation too. Let $T_{N,K}$ be a ranked tree shape with $N$ tips and $K$ internal nodes with string representation $\{t, l\}$. Let $T_{N,K-1}^{(-e)}$ be obtained from $T_N^K$ by collapsing edge $(e, e+1)$ for some $e \in {1,...,K-1}$, which is equivalent to the condition $t_{e+1} = e$. When edge $(e,e+1)$ is collapsed, there corresponding action in the string representation is a deletion in $l$ and an addition of terms in $l$. Let $\{t^{(-e)}, l^{(-e)}\}$ denote the string representation of $T_{N,K-1}^{(-e)}$, and we detail how to obtain it from $\{t,l\}$. 
\begin{itemize}
    \item For $i=1,...,e$: $t^{(-e)}_i = t_i$. The parent nodes for the nodes with higher rank than $e$ do not change. 
    \item For $i=e+1,...,K-1$: $t^{(-e)}_i = t_{i+1} - \mathds{1}(t_{i+1} > t_{e+1})$. The parent nodes for the nodes with lower rank than $e$ either remain the same if the parent comes before $e$, or shifted by 1 otherwise because a ranking was deleted.  
    \item For $i=1,..., e-1$: $l^{(-e)}_i = l_i$. The number of leaves descending from nodes with rank higher than $e$ do not change.
    \item $(l^{(-e)}_1,...,l^{(-e)}_{e-1}, l^{(-e)}_e, l^{(-e)}_{e+1},...,l^{(-e)}_{K-1}) = (l_1,...,l_{e-1}, l_e+ l_{e+1}, l_{e+2},...,l_K)$: By removing edge $(e,e+1)$, the leaves originally descending from $e$ and from $e+1$ all descend from $e$ in the new tree, so they are added together. Everything else remains the same. 
\end{itemize}
Figure~\ref{fig:string_collapse_edge} shows the same trees from Figure~\ref{fig:Fmat_collapse_edge} with the string representations. We are still able to define the same edge collapsing operation, but we will use the F-matrix because of its intuitiveness in row/column deletion and ease of implementation.
\begin{figure}[H]
    \centering
    \includegraphics[width=0.8\linewidth]{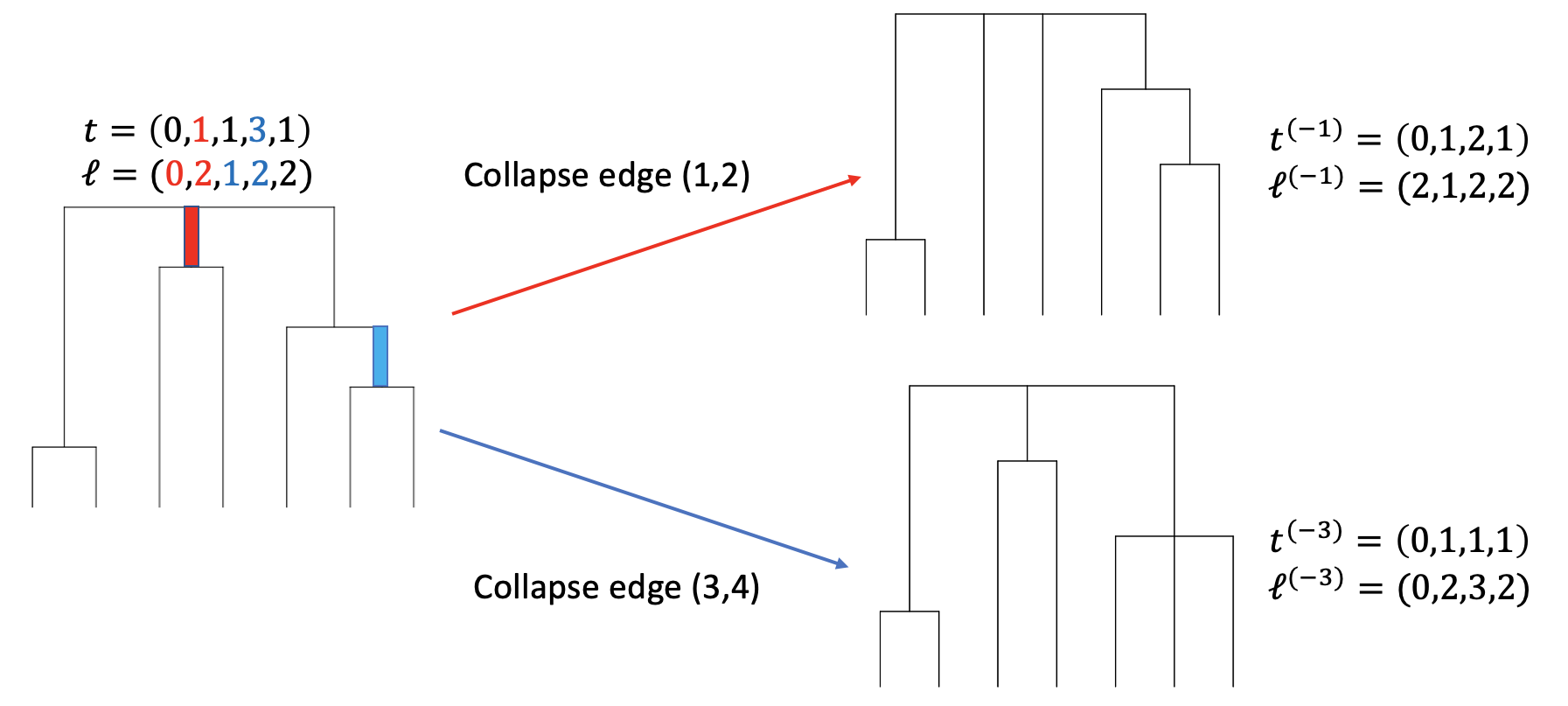}
    \caption{The tree from Figure~\ref{fig:Fmat_collapse_edge} with 7 tips and 5 internal nodes, where the edges $(1,2)$ and $(3,4)$ can be removed. We show both modified tree topology and the corresponding string representation.}
    \label{fig:string_collapse_edge}
\end{figure}

\end{document}